\documentclass[11pt,reqno]{amsart}

\usepackage[english]{babel}
\usepackage[a4paper,twoside]{geometry}
\usepackage{microtype}
\usepackage{enumitem}

\usepackage{csquotes}
\usepackage{longtable}

\usepackage{xspace}
\patchcmd{\subsection}{-.5em}{.5em}{}{}
\usepackage[hyperfootnotes=false]{hyperref} 
\usepackage{color}
\hypersetup{
	colorlinks = true,
	linkcolor = {blue},
	urlcolor = {red},
	citecolor = {blue}
}

\makeatletter
\newcommand{\mysubsubsection}[1]{\subsubsection*{\bfseries #1}}
\renewcommand{\tocsection}[3]{
  \indentlabel{\@ifnotempty{#2}{\ignorespaces#1 #2\quad}}\bfseries#3}
\renewcommand{\tocsubsection}[3]{
  \indentlabel{\@ifnotempty{#2}{\ignorespaces#1 #2\quad}}#3}

\newcommand\@dotsep{4.5}
\def\@tocline#1#2#3#4#5#6#7{\relax
  \ifnum #1>\c@tocdepth
  \else
    \par \addpenalty\@secpenalty\addvspace{#2}
    \begingroup \hyphenpenalty\@M
    \@ifempty{#4}{
      \@tempdima\csname r@tocindent\number#1\endcsname\relax
    }{
      \@tempdima#4\relax
    }
    \parindent\z@ \leftskip#3\relax \advance\leftskip\@tempdima\relax
    \rightskip\@pnumwidth plus1em \parfillskip-\@pnumwidth
    #5\leavevmode\hskip-\@tempdima{#6}\nobreak
    \leaders\hbox{$\m@th\mkern \@dotsep mu\hbox{.}\mkern \@dotsep mu$}\hfill
    \nobreak
    \hbox to\@pnumwidth{\@tocpagenum{\ifnum#1=1\bfseries\fi#7}}\par
    \nobreak
    \endgroup
  \fi}
\AtBeginDocument{
\expandafter\renewcommand\csname r@tocindent0\endcsname{0pt}
}
\def\l@subsection{\@tocline{2}{0pt}{2.5pc}{5pc}{}}
\makeatother

\usepackage[backend=biber,style=alphabetic,bibencoding=utf8,language=auto,autolang=other,
url=false, isbn=false, maxnames=10, giveninits=true]{biblatex}

\usepackage{amsmath}
\usepackage{amsthm}
\usepackage{amssymb}
\usepackage{mathrsfs} 
\usepackage{eucal}
\usepackage{mathtools}

\usepackage{graphicx}
\usepackage{tikz}
\usetikzlibrary{cd}

\newcounter{results}[section]

\theoremstyle{plain}
\newtheorem{theorem}[results]{Theorem}
\newtheorem{lemma}[results]{Lemma}
\newtheorem{proposition}[results]{Proposition}
\newtheorem{corollary}[results]{Corollary}

\theoremstyle{remark}
\newtheorem{remark}[results]{Remark}
\newtheorem{example}[results]{Example}

\theoremstyle{definition}
\newtheorem{definition}[results]{Definition}

\numberwithin{equation}{section}

\newcommand{\R}{\ensuremath{\mathbb R}} 
\newcommand{\N}{\ensuremath{\mathbb N}} 

\newcommand{\eps}{\ensuremath{\varepsilon}} 

\newcommand{\floor}[1]{\left\lfloor #1 \right\rfloor} 
\newcommand{\ceil}[1]{\left\lceil #1 \right\rceil} 
\DeclareMathOperator*{\argmin}{arg\,min} 
\newcommand{\dom}{\ensuremath{\mathrm{D}}}

\newcommand{\upds}{{\frac{\d}{\d s}}^{\kern-3pt +}} 
\newcommand{\updt}{{\frac{\d}{\d t}}^{\kern-3pt +}} 
\newcommand{\lodt}{{\frac{\d}{\d t}}_{\kern-1pt +}} 

\newcommand{\la}{\langle} 
\newcommand{\ra}{\rangle} 

\newcommand{\de}{\ensuremath{\,\mathrm d}} 
\renewcommand{\d}{\mathrm d}  

\newcommand{\intt}[1]{\operatorname{int}\left(#1\right)} 
\renewcommand{\exp}{\operatorname{\mathsf{exp}}} 

\newcommand{\pclo}[1]{{\overline{#1}}} 

\newcommand{\bb}{\boldsymbol b} 
\newcommand{\ggamma}{\boldsymbol\gamma} 
\newcommand{\ii}{\boldsymbol i} 
\newcommand{\mmu}{\boldsymbol\mu} 
\newcommand{\nnu}{\boldsymbol\nu} 
\newcommand{\Ppsi}{\boldsymbol\Psi} 
\newcommand{\rr}{\boldsymbol r} 
\newcommand{\ssigma}{\boldsymbol\sigma} 
\newcommand{\ttau}{\boldsymbol \tau} 
\newcommand{\ttheta}{\boldsymbol\vartheta} 
\newcommand{\Ttheta}{\boldsymbol\Theta} 
\newcommand{\vv}{\boldsymbol v} 
\newcommand{\ww}{\boldsymbol w} 
\newcommand{\xx}{{\boldsymbol x}} 
\newcommand{\yy}{{\boldsymbol y}} 
\newcommand{\zzeta}{\boldsymbol\zeta} 

\newcommand{\m}{\mathsf m} 
\newcommand{\sfv}{\mathsf v} 
\newcommand{\sfx}{\mathsf x} 
\newcommand{\X}{\mathsf X} 
\newcommand{\Y}{\mathsf Y} 

\newcommand{\func}{\ensuremath{\mathcal{F}}}
\newcommand{\lfunc}{\mathcal{G}} 
\newcommand{\interval}{\mathcal I} 
\newcommand{\iso}{\mathcal{J}} 

\newcommand{\EVI}{{\rm EVI}\xspace} 
\newcommand{\wEVI}{{\rm EVI}\xspace} 
\newcommand{\MPVF}{{\rm MPVF}\xspace} 

\newcommand{\lebd}{{\mathcal{L}_B}} 
\newcommand{\relcP}[3]{\prob_{#1}(#3|#2)} 
\newcommand{\prob}{\ensuremath{\mathcal{P}}} 
\DeclareMathOperator{\supp}{supp} 

\newcommand{\CondGammao}[4]{\Gamma_o^{#4}({#2},{#3}|#1)} 
\DeclareMathOperator{\Tan}{Tan} 

\newcommand{\sqm}[1]{\mathsf m_2^2(#1)} 
\newcommand{\rsqm}[1]{\mathsf m_2(#1)} 

\newcommand{\TRd}{\mathsf{T}\R^d} 
\newcommand{\TX}{\mathsf {T\kern-1.5pt X}} 
\newcommand{\TY}{\mathsf {T\kern-1.5pt Y}} 

\newcommand{\rmC}{\mathrm C} 
\DeclareMathOperator{\Cyl}{Cyl} 
\newcommand{\Lip}{\mathrm {Lip}} 
\newcommand{\testsw}[1]{\rmC^{sw}_{2}(#1)} 

\newcommand{\scalprod}[2]{\ensuremath{\langle #1, #2\rangle}} 
\newcommand{\bram}[2]{\ensuremath{\left [ #1, #2\right ]_{r}}} 
\newcommand{\brap}[2]{\ensuremath{\left [ #1, #2\right ]_{l}}} 

\newcommand{\directionalm}[3]{[#1,#2]_{r,#3}} 
\newcommand{\directionalp}[3]{[#1,#2]_{l,#3}} 
\newcommand{\ebrab}[3]{\ensuremath{\left [ #1, #2\right ]_{b,#3}}} 
\newcommand{\directional}[3]{[#1,#2]_{#3}} 

\newcommand{\frF}{{\boldsymbol{\mathrm F}}}
\newcommand{\fF}{{\boldsymbol F}}

\newcommand{\clo}[1]{\operatorname{cl}(#1)} 
\newcommand{\cloco}[1]{\overline{\operatorname{co}}(#1)} 

\DeclareMathOperator{\Adm}{Adm} 

\newcommand{\bry}[1]{\boldsymbol b_{#1}} 

\newcommand{\proj}{\ensuremath{\mathbf{pr}}} 

\newcommand{\Rotn}{{\boldsymbol r}} 

\newcommand{\conv}[1]{\operatorname{co}(#1)} 

\newcommand{\rmS}{\mathrm S} 

\newcommand{\GX}{X} 
\newcommand{\PX}{\X\times \Y} 

\newcommand{\rI}[2]{\mathrm I(#1|#2)} 
\newcommand{\rB}[2]{\mathrm B(#1,#2)} 

\newcommand{\finalstep}[2]{{\mathrm N(#1,#2)}} 

\newcommand{\expcyl}[2]{\mathrm {exp}^{#1#2}} 

\title[Dissipative PVFs and generation of evolution semigroups in Wasserstein spaces]{Dissipative probability vector fields and generation of evolution semigroups in Wasserstein spaces}

\author{Giulia Cavagnari}
\address{Giulia Cavagnari: Politecnico di Milano, Dipartimento di Matematica, Piazza Leonardo Da Vinci 32, 20133 Milano (Italy)}
\email{giulia.cavagnari@polimi.it}

\author{Giuseppe Savar\'e}
\address{Giuseppe Savar\'e: Bocconi University,
  Department of Decision Sciences and BIDSA, Via Roentgen 1, 20136 Milano (Italy)}
\email{giuseppe.savare@unibocconi.it}

\author{Giacomo Enrico Sodini}
\address{Giacomo Enrico Sodini: TUM Fakult\"at f\"ur Mathematik, Boltzmannstrasse 3, 85748 Garching bei M\"unchen (Germany)}
\email{sodini@ma.tum.de}

\subjclass{Primary: 34A06, 34A45; Secondary: 34A12, 34A34, 34A60, 28A50}
 \keywords{Measure differential equations/inclusions in Wasserstein spaces,  probability vector fields, dissipative operators, evolution variational inequality, explicit Euler scheme.}

\begin{document}

\begin{abstract}
  We introduce and investigate a notion of multivalued $\lambda$-dissipative probability vector field (MPVF)
  in the Wasserstein space $\prob_2(\X)$ of Borel probability measures
  on a Hilbert space $\X$.
  Taking inspiration from the theory of dissipative operators in
  Hilbert spaces
  and of Wasserstein gradient flows of geodesically convex functionals,
  we study local and global well posedness of evolution equations
  driven by dissipative MPVFs. Our approach is based
  on a measure-theoretic version of the Explicit Euler scheme,
  for which we prove novel convergence results with optimal
  error estimates under an abstract CFL stability condition,
  which do not rely on compactness arguments and
  also hold when $\X$ has infinite dimension.

  We characterize the limit solutions by a suitable Evolution
  Variational Inequality (EVI), inspired by the B\'enilan notion of integral
  solutions to dissipative evolutions in Banach spaces.
  Existence, uniqueness and stability of EVI solutions
  are then obtained under quite general assumptions, leading to
  the generation of a semigroup
  of nonlinear contractions. 
\end{abstract}

\maketitle
\tableofcontents
\thispagestyle{empty}

\section{Introduction}
The aim of this paper is to study the local and global well posedness
of 
evolution equations for Borel probability measures
driven by a suitable notion of probability vector fields
in an Eulerian framework.

For the sake of simplicity, let us consider here
a finite dimensional Euclidean space $\X$ with scalar product
$\langle\cdot,\cdot\rangle$ and norm $|\cdot|$
(our analysis however will not be confined to finite dimension and
will be carried out in a separable Hilbert space)
and the space $\prob(\X)$ (resp.~$\prob_b(\X)$) of 
Borel probability measures in $\X$ (resp.~with bounded support).

\mysubsubsection{A Cauchy-Lipschitz approach, via vector fields}

A first notion of vector field
can be described by maps 
$\bb:\prob_b(\X)\to \mathrm C(\X;\X)$,
typically taking values in some subset of continuous vector fields in
$\X$ (as
the locally Lipschitz ones of $\Lip_{loc}(\X;\X)$),
and satisfying suitable growth-continuity conditions.
In this respect, the evolution driven by $\bb$ can be described
by a continuous curve $t\mapsto \mu_t\in \prob_b(\X)$,
$t\in [0,T]$,
starting from an initial measure $\mu_0\in \prob_b(\X)$ and
satisfying the continuity equation
\begin{subequations}
  \label{eq:167}
  \begin{align}
    \label{eq:167a}
    \partial_t \mu_t+\nabla\cdot(\vv_t\mu_t)&=0&&\text{in
                                              }(0,T)\times \X,\\
    \label{eq:167b}
    \vv_t&=\bb[\mu_t]&&\text{$\mu_t$-a.e.~for every $t\in (0,T)$},
  \end{align}
\end{subequations}
in the distributional sense, i.e.
\begin{equation}
  \label{eq:166}
  \int_0^T\int_\X \Big(\partial_t \zeta+\langle
  \nabla\zeta,\vv_t\rangle\Big)\,\d\mu_t\,\d t=0,\quad
  \vv_t=\bb[\mu_t],
  \quad\text{for every }
  \zeta\in \rmC^1_c((0,T)\times \X).
\end{equation}
If $\bb$ is sufficiently smooth, solutions
to (\ref{eq:167}a,b) can be obtained by many techniques.
Recent contributions in this direction are given by the papers \cite{Piccoli_2019,Piccoli_MDI,bonnet2020mean,CLOS}, we also mention \cite{RosPic,RosPicMDE} for the analysis in presence of sources.
In particular, in \cite{bonnet2020mean} the aim of the authors is to
develop a suitable Cauchy-Lipschitz theory in Wasserstein spaces for
differential inclusions which generalizes \eqref{eq:167b} to
multivalued maps 
$\bb:\prob_b(\X)\rightrightarrows \Lip_{loc}(\X;\X)$
and requires \eqref{eq:167b}, \eqref{eq:166} to hold for a suitable
measurable selection of $\bb$.
As it occurs in the classical finite-dimensional case, the differential-inclusion approach is suitable to describe the dynamics of control systems, when the velocity vector field involved in the continuity equation depends on a control parameter.

\mysubsubsection{The Explicit Euler method}
A natural approach, that is suitable for a great generalization,
is to approximate (\ref{eq:167}a,b) by
a measure-theoretic version of the Explicit Euler scheme.
Choosing a step size $\tau>0$ and
a partition $\{0,\tau,\cdots,n\tau,\cdots, N\tau\}$
of the interval $[0,T]$, $N=\finalstep T\tau=\ceil{T/\tau}$, we
construct a sequence $M^n_\tau\in \prob_b(\X)$, $n=0,\cdots, N,$ by
the algorithm
\begin{equation}
  \label{eq:168}
  M^0_\tau:=\mu_0,\quad
  M^{n+1}_\tau:=(\ii_\X+\tau \bb^n_\tau)_\sharp M^n_\tau,\quad
  \bb^n_\tau\in \bb[M^n_\tau],
\end{equation}
where $\ii_\X(x):=x$ is the identity map
and $\rr_\sharp\mu$ denotes the push forward of $\mu\in \prob(\X)$
induced by a Borel map $\rr:\X\to \X$ and 
defined by $\rr_\sharp\mu(B):=\mu(\rr^{-1}(B))$
for every Borel set $B\subset \X$.
If $\bar M_\tau$ is the piecewise constant interpolation
of the discrete values $(M^n_\tau)_{n=0}^N$, one can then
study the convergence of $\bar M_\tau$ as $\tau\downarrow0$,
hoping to obtain a solution to (\ref{eq:167}a,b) in the limit.

It is then natural to investigate a few relevant questions:
\begin{enumerate}[label=\rm $\langle$E.\arabic*$\rangle$]
\item
  \label{Q1}
  what is the most general framework where the Explicit Euler
  scheme can be implemented,
\item
  what are the structural conditions
  ensuring its convergence,
  \label{Q2}
\item
  how to characterize the limit solutions
  and their properties.
  \label{Q3}
\end{enumerate}
Concerning the first question \ref{Q1}, one immediately realizes that each iteration
of \eqref{eq:168} actually depends on the probability distribution
on the tangent bundle $\TX$ (which we may identify with $\X\times \X$,
where the second component plays the role of velocity)
\begin{equation*}
  \Phi^n_\tau:=(\ii_\X,\bb^n_\tau)_\sharp M^n_\tau,
  \in\prob(\TX)
\end{equation*}
whose first marginal is $M^n_\tau$.
If we denote by $\sfx,\sfv:\TX\to \X$ the projections $\sfx(x,v)=x,\
\sfv(x,v)=v$, and by 
$  \exp^\tau: \TX \to \X$ the exponential map in the flat space $\X$ $\exp^\tau(x,v):=x+\tau v
$, we recover $M^{n+1}_\tau$ by a single step of ``free motion'' driven
by $\Phi^n_\tau$ and given by
\begin{equation*}
  M^{n+1}_\tau= \exp^\tau_\sharp \Phi^n_\tau=(\sfx+\tau\sfv)_\sharp \Phi^n_\tau.
\end{equation*}
This operation does not depend on the fact that $\Phi^n_\tau$ is
concentrated
on the graph of a map (in this case $\bb^n_\tau\in \bb[M^n_\tau]$): one can
more generally assign a multivalued map $\frF:\prob_b(\X)\rightrightarrows \prob_b(\TX)$
such that for every $\mu\in \prob_b(\X)$
every measure $\Phi\in \frF[\mu]\in \prob_b(\TX)$ has first marginal 
$\mu=\sfx_\sharp \Phi$. 
We call $\frF$ a \emph{multivalued probability vector field} (\MPVF in
the following), which is in good analogy with
a Riemannian interpretation of $\prob_b(\TX)$.
The disintegration $\Phi_x\in \prob_b(\X)$ of $\Phi$ with respect to
$\mu$ provides a (unique up to $\mu$-negligible sets)
Borel family of probability measures on
the space of velocities such that $\Phi=\int_\X \Phi_x\,\d\mu(x)$.
$\Phi$ is induced by 
a vector field $\bb$ only if $\Phi_x=\delta_{\bb(x)}$ is a
Dirac mass for $\mu$-a.e.$x$.
\eqref{eq:168} now reads as
\begin{equation}
  \label{eq:168b}
    M^0_\tau:=\mu_0,\quad
  M^{n+1}_\tau:= \exp^\tau_\sharp \Phi^n_\tau=
(\sfx+\tau \sfv)_\sharp \Phi^n_\tau,\quad
  \Phi^n_\tau\in \frF[M^n_\tau].  
\end{equation}
Besides greater generality, this point of view has other 
advantages: working with the joint distribution $\frF[\mu]$ instead
of the disintegrated vector field $\bb[\mu]$
potentially allows for the weakening of the
continuity assumption with respect to $\mu$.
This relaxation corresponds to the introduction of Young's measures
to study the limit behaviour of weakly converging maps \cite{CdFV}.
Adopting this viewpoint, the classical discontinuous example in $\R$
(see \cite{Filippov}), where $\bb(x)=-\mathrm{sign} (x)$,
admits a natural closed realization as MPVF given by
\[\Phi\in \frF[\mu]
  \quad\Leftrightarrow\quad
  \Phi_x=\begin{cases}\delta_{\bb(x)}&\text{if }x\neq0\\
    (1-\theta)\delta_{-1}+\theta \delta_1&\text{if }x=0
  \end{cases}
\quad\text{for some }\theta\in [0,1].\]
In particular,
$\frF[\delta_0]=\big\{\delta_0\otimes\left((1-\theta)\delta_{-1}+
  \theta\delta_{1}\right)\mid\theta\in [0,1]\big\}$ (see also \cite[Example
  6.2]{Camilli_MDE}).

  The study of measure-driven differential equations/inclusions is not new in the literature \cites{DalMasoRampazzo,SilvaVinter}. However, these studies, devoted to the description of impulsive control systems \cite{Bressan} and mainly motivated by applications in rational mechanics and engineering, have been used to describe evolutions in $\R^d$ rather than in the space of measures.

A second advantage in considering a \MPVF is the consistency with the
theory of Wasserstein gradient flows generated by geodesically convex
functionals
introduced in \cite{ags}
(Wasserstein subdifferentials are particular examples of {\MPVF}s)
and with the multivalued version of the notion of probability vector
fields introduced in 
\cite{Piccoli_2019,Piccoli_MDI}, whose 
originating idea was indeed to describe the uncertainty affecting not
only the state of the system,
but possibly also the distribution of the vector field itself.

A third advantage is to allow for a more intrinsic geometric view,
inspired by Otto's non-smooth Riemannian interpretation of the
Wasserstein space: probability vector
fields provide an appropriate description of infinitesimal deformations of
probability measures, which should be measured by, e.g., the
$L^2$-Kantorovich-Rubinstein-Wasserstein distance
\begin{equation}
  \label{eq:172}
  W_2^2(\mu,\nu):=\min\Big\{\int_{\X\times \X}|x-y|^2\,\d\ggamma(x,y):
  \ggamma\in \Gamma(\mu,\nu)\Big\},
\end{equation}
where
$\Gamma(\mu,\nu)$
 is the set of
  couplings with marginals $\mu$ and $\nu$ respectively.
  It is well known \cite{ags,Villani,santambrogio}
  that
  if $\mu,\nu$ belong to the space 
  $\prob_2(\X)$ of Borel probability measures with finite second moment
\begin{equation*}
  \sqm\mu:=\int_\X |x|^2\,\d\mu(x)<\infty,
\end{equation*}
then the minimum in \eqref{eq:172} is attained in a compact convex set
$\Gamma_o(\mu,\nu)$
and $(\prob_2(\X),W_2)$ is a complete and separable metric space.
Adopting this viewpoint and proceeding by analogy with the theory of
dissipative operators in Hilbert spaces,
a natural class of MPVFs for evolutionary problems should at least
satisfy a $\lambda$-dissipativity condition, $\lambda\in \R$, as
\begin{equation}
  \label{eq:174}
  \forall\, \Phi\in\frF[\mu],\ \Psi\in\frF[\nu],\ \mu\neq\nu :\quad
  W_2(\exp^\tau_\sharp\Phi,\exp^\tau_\sharp\Psi)
  \le (1+\lambda\tau) W_2(\mu,\nu)+o(\tau)\quad\text{as
  }\tau\downarrow0.
\end{equation}
\vspace{-12pt}
\mysubsubsection{Metric dissipativity}
Condition \eqref{eq:174} in the simple case $\lambda=0$ has a clear interpretation in terms of one
step of the Explicit Euler method: it is an asymptotic contraction
as the time step goes to $0$. By using the properties of the
Wasserstein distance, we will
first compute the right derivative of the (squared) Wasserstein
distance
along the deformation $\exp^\tau$
\begin{equation}
  \label{eq:164}
  \begin{aligned}
    \bram\Phi\Psi&:=
    \frac12\frac \d{\d \tau}
    W_2^2(\exp^\tau_\sharp\Phi,\exp^\tau_\sharp\Psi)\Big|_{\tau=0+}
    \\={}&
    \min \Big\{\int_{\TX\times \TX}\langle
    w-v,y-x\rangle\,\d\Ttheta(x,v;y,w): \Ttheta\in \Gamma(\Phi,\Psi),\
    (\mathsf x,\mathsf y)_\sharp\Ttheta\in \Gamma_o(\mu,\nu)\Big\}
  \end{aligned}
\end{equation}
and we will 
show that \eqref{eq:174}
admits the equivalent characterization
\begin{equation}
  \label{eq:176}
  \bram\Phi\Psi
  \le \lambda W_2^2(\mu,\nu)\quad
  \text{for every $\Phi\in \frF[\mu],\ \Psi\in \frF[\nu]$}.
\end{equation}
If we interpret the left hand side of \eqref{eq:176}
as a sort of Wasserstein pseudo-scalar product
of $\Phi$ and $\Psi$ along the direction of
an optimal coupling between $\mu$ and $\nu$,
\eqref{eq:176}
is in perfect analogy with
the canonical definition of $\lambda$-dissipativity (also called one-sided
Lipschitz condition)
for a multivalued map $\fF:\X\rightrightarrows \X$:
\begin{equation}
  \label{eq:77}
  \langle w-v,y-x\rangle\le \lambda |x-y|^2 \quad
  \text{for every $v\in \fF[x],\ w\in \fF[y]$}.
\end{equation}
It turns out that the (opposite of the) Wasserstein subdifferential $\boldsymbol\partial
\mathcal F$ \cite[Section 10.3]{ags} of a geodesically $(-\lambda)$-convex functional $\mathcal
F:\prob_2(\X)\to (-\infty,+\infty]$ is a \MPVF and
satisfies a condition equivalent to \eqref{eq:174} and \eqref{eq:176}.
We also notice that \eqref{eq:176} reduces to \eqref{eq:77}
in the particular case when $\Phi=\delta_{(x,v)},\Psi=\delta_{(y,w)}$
are Dirac masses in $\TX$.

\mysubsubsection{Conditional convergence of the Explicit Euler method}
Differently from the Implicit Euler method, however,
even if a \MPVF satisfies \eqref{eq:176}, every step of the Explicit
Euler scheme \eqref{eq:168b} 
affects the distance by a further quadratic correction
according to the formula
\begin{equation*}
  W_2^2(\exp^\tau_\sharp\Phi,\exp^\tau_\sharp\Psi)
  \le W_2^2(\mu,\nu)+2\tau \bram\Phi\Psi+\tau^2\Big(
  |\Phi|_2^2+|\Psi|_2^2\Big),\quad
  |\Phi|_2^2:=\int_\TX |v|^2\,\d\Phi(x,v),
\end{equation*}
which depends on the order of magnitude of $\Phi$ and $\Psi$, and thus
of $\frF$, at $\mu$ and $\nu$.

Our first main result (Theorems \ref{prop:rate},\ref{theo:strong-solution}), which provides an answer to 
question \ref{Q2},
states that if $\frF$ is a $\lambda$-dissipative \MPVF
according to \eqref{eq:176} then
every family of discrete solutions $(\bar M_\tau)_{\tau>0}$
of \eqref{eq:168b} in an interval $[0,T]$ 
satisfying the abstract CFL condition 
\begin{equation}
  \label{eq:175}
  |\Phi^n_\tau|_2\le L\quad \text{if }0\le n\le N=\finalstep T\tau,
\end{equation}
is uniformly converging to a limit curve $\mu\in
\Lip([0,T];\prob_2(\X))$
starting from $\mu_0$, with a uniform error estimate
\begin{equation}
  \label{eq:177}
  W_2(\mu_t,\bar M_\tau(t))\le CL\sqrt{\tau(t+\tau)}\mathrm
  e^{\lambda_+ t}\quad\text{for every }t\in [0,T]
\end{equation}
and a universal constant $C\le 14$. Apart from the precise value of $C$, the estimate \eqref{eq:177}
is sharp
\cite{Rulla} and reproduces in the measure-theoretic framework the
celebrated Crandall-Liggett
\cite{Crandall-Liggett}
estimate for the generation of dissipative semigroups in Banach
spaces. We derive it by adapting
to the metric-Wasserstein setting
the relaxation and doubling variable techniques
of \cite{NochettoSavare}, strongly inspired by
the ideas of Kru\v{z}kov \cite{Kruzkov} and
Crandall-Evans \cite{Crandall-Evans}.

This crucial result does not require any bound on the support of
the measures 
neither local compactness of the underlying space $\X$, so that we
will prove it in a general Hilbert space, possibly with infinite
dimension. Moreover, if $\mu,\nu$ are two limit solutions
starting from $\mu_0,\nu_0$ we show that
\begin{equation*}
  W_2(\mu_t,\nu_t)\le W_2(\mu_0,\nu_0)\mathrm e^{\lambda t}\quad
  t\in [0,T],
\end{equation*}
as it happens in the case of gradient flows of $(-\lambda)$-convex functions.
Once one has these building blocks, it is not too difficult to
construct a local and global existence theory,
mimicking the standard arguments for ODEs.

\mysubsubsection{Metric characterization of the limit solution}
As we stated in question \ref{Q3},
a further important point is to get an effective characterization
of the solution $\mu$ obtained as limit of the approximation scheme.

As a first property, considered
in \cite{Piccoli_2019,Piccoli_MDI} in the case of a single-valued PVF,
one could hope that $\mu$ satisfies
the continuity equation \eqref{eq:167a}
coupled with the barycentric condition replacing \eqref{eq:167b}
\begin{equation}
  \label{eq:179}
  \vv_t(x)=\int_\TX v\,\d\Phi_t(x,v),\quad
  \Phi_t=\frF[\mu_t].
\end{equation}
This is in fact true, as shown in \cite{Piccoli_2019,Piccoli_MDI} in
the finite dimensional case,
if $\frF$ is single valued and satisfies
a stronger Lipschitz dependence w.r.t.~$\mu$
(see \ref{Pgrowth} in Appendix \ref{sec:cfrPic}).

In the framework of dissipative {\MPVF}s,
we will replace \eqref{eq:179} with its relaxation \emph{\`a la
  Filippov} (see e.g.~\cite[Chapter 2]{Vinter} and \cite[Chapter 10]{AuF})
\begin{equation*}
  \vv_t(x)=\int_\TX v\,\d\Phi_t(x,v)\quad
  \text{for some}\quad
  \Phi_t\in\cloco{\clo\frF[\mu_t]},
\end{equation*}
where $\clo\frF$ is the sequential closure
of the graph of $\frF$ in the strong-weak topology of
$\prob_2^{sw}(\TX)$ (see \cite{NaldiSavare} and Section \ref{subsec:sw} for more
details; in fact, a more restrictive ``directional'' closure could be
considered, see \eqref{eq:deform}) and $\cloco{\clo\frF[\mu]}$ denotes
the closed convex hull of the given section $\clo\frF[\mu]$.

However, even in the case of a single valued map,
\eqref{eq:179} is not enough to characterize the limit solution, as
it has been shown by an interesting example in
\cite{Piccoli_2019,Camilli_MDE} (see also the gradient flow of Example \ref{ex:bary}).

Here we follow the metric viewpoint adopted in \cite{ags} for gradient
flows and we will characterize the limit solutions by a suitable
Evolution Variational Inequality
satisfied by the squared distance function from given test measures.
This approach is also strongly influenced by the
B\'enilan
notion of integral solutions to dissipative evolutions in
Banach spaces \cite{Benilan}. The main idea is that any differentiable solution to
$\dot x(t)\in\fF[x(t)]$ driven by a $\lambda$-dissipative operator in a Hilbert
space as in \eqref{eq:77}
satisfies
\begin{equation*}
  \begin{aligned}
    \frac 12\frac\d{\d t}|x(t)-y|^2
    &= \langle \dot x(t),x(t)-y\rangle=
    \langle \dot x(t)-w,x(t)-y\rangle+\langle w,x(t)-y\rangle
    \\&\le
    \lambda |x(t)-y|^2-\langle w,y-x(t)\rangle\quad\text{for every
    }w\in \frF[y].
  \end{aligned}
\end{equation*}
In the framework of $\prob_2(\X)$, we replace
$w\in \fF[y]$ with $\Psi\in \frF[\nu]$ and
the scalar product
$\langle w,y-x(t)\rangle$ with 
\begin{equation*}
  \bram\Psi{\mu_t}:=\min \Big\{\int_{\TX\times \X}\langle
  w,y-x\rangle\,\d\Ttheta(y,w;x): \Ttheta\in \Gamma(\Psi,\mu_t),\
  (\mathsf y,\mathsf x)_\sharp\Ttheta\in \Gamma_o(\nu,\mu_t)\Big\},
\end{equation*}
as in  \eqref{eq:164}.
According to this formal heuristic, we obtain the $\lambda$-\EVI
characterization of a limit curve $\mu$ as
\begin{equation}
  \tag{$\lambda$-\EVI}
  \label{eq:183}
  \frac 12\frac\d{\d t}W_2^2(\mu_t,\nu)\le \lambda W_2^2(\mu_t,\nu)-
  \bram\Psi{\mu_t}\quad\text{for every }\Psi\in \frF[\nu].    
\end{equation}
As for B\'enilan integral solutions, we can
considerably relax the apriori smoothness assumptions on $\mu$,
just imposing that $\mu$ is continuous and \eqref{eq:183}
holds in the sense of distributions in $(0,T)$.
In this way, we obtain a robust characterization, which is stable
under uniform convergence and also allows for 
solutions taking values in the closure of the domain of $\frF$.
This is particularly important when $\frF$ involves drift terms
with superlinear growth (see Example \ref{ex:rulla}).

The crucial point of this approach relies on a general error estimate,
which extends the validity of \eqref{eq:177} to
a general $\lambda$-\EVI solution $\mu$ and therefore
guarantees its uniqueness, whenever the Explicit Euler method is
solvable, at least locally in time.

Combining local in time existence with suitable global
confinement conditions (see e.g.~Theorem \ref{thm:global-bound})
we can eventually obtain a robust theory for the generation of
a $\lambda$-flow, i.e.~a semigroup $(\rmS_t)_{t\ge0}$ in
a suitable subset $D$ of $\prob_2(\X)$
such that $\rmS_t[\mu_0]$ is the unique $\lambda$-\EVI solution
starting from $\mu_0$ and for every $\mu_0,\mu_1\in D$
\begin{equation*}
  W_2(\rmS_t[\mu_0],\rmS_t[\mu_1])\le W_2(\mu_0,\mu_1)\mathrm
  e^{\lambda t},\quad t\ge0,
\end{equation*}
as in the case of Wasserstein gradient flows of geodesically
$(-\lambda)$-convex functionals.

\mysubsubsection{Explicit vs Implicit Euler method}
In the framework of contraction semigroups generated by
$\lambda$-dissipative
operators in Hilbert or Banach spaces, a crucial role is played by
the Implicit Euler scheme, which has the advantage to be
unconditionally stable, and thus avoids any apriori restriction
on the local bound of the operator, as we did in \eqref{eq:175}.
In Hilbert spaces, it is well known that the solvability of the Implicit Euler scheme is
equivalent
to the maximality of the graph of the operator.

In the case of a Wasserstein gradient flow of a geodesically convex  $\func:\prob_2(\X)\to(-\infty,+\infty]$, every step of the Implicit Euler
method 
(also called JKO/Minimizing Movement scheme \cite{JKO,ags}) can be solved
by a variational approach: $M^{n+1}_\tau$ has to be selected among the
solutions of 
\begin{equation}
  \label{eq:185}
  \min_{M\in \prob_2(\X)}  \frac1{2\tau}W_2^2(M,M^n_\tau)+\mathcal F(M).
\end{equation}
Notice, however, that in this case the \MPVF $\boldsymbol\partial\mathcal F$ is
defined implicitely 
in terms of $\mathcal F$ and
each step of \eqref{eq:185} provides a suitable variational selection
in $\boldsymbol\partial\mathcal F$, leading in the limit to the minimal
selection principle.

In the case of more general dissipative evolutions, it is not at all
clear how to solve the Implicit Euler scheme, in particular
when $\frF[\mu]$ is not concentrated on a map, and to
characterize the
maximal extension of $\frF$ (in the Hilbertian case
the maximal extension of a dissipative operator $\fF$ 
is explicitly computable
at least when the domain of $\fF$ has not empty interior,
see the Theorems of Robert and B\'enilan in \cite{Qi}).
Indeed, the analogy with
the Hilbertian theory does not extend to some properties
which play a crucial role. In particular, 
a dissipative MPVF $\frF$ in $\prob_2(\X)$
is not locally bounded in the interior of its domain (see Example
\ref{ex:unboundF}) and maximality may fail also for single-valued
continuous PVFs
(see Example \ref{ex:contnotmax}).
Even more
remarkably, in the Hilbertian case
a crucial equivalent characterization of dissipativity reads as
\begin{equation*}
  v\in \fF[x],\ w\in \fF[y]\quad\Rightarrow\quad
  |x-y|\le |(x-\tau v)-(y-\tau w)|
\end{equation*}
which implies that the resolvent operators $(I-\tau \fF)^{-1}$
(and every single step of the Implicit Euler scheme) are contractions in $\X$.
On the contrary, if we assume the forward characterizations
\eqref{eq:174} and \eqref{eq:176}
of dissipativity in $\prob_2(\X)$ (with $\lambda=0$)
we \emph{cannot} conclude in general that
\begin{equation}
  \label{eq:187}
  \Phi\in \frF[\mu],\ \Psi\in \frF[\nu]\quad\Rightarrow
  \quad
  W_2(\mu,\nu)\le W_2(\exp^{-\tau}_\sharp \Phi, \exp^{-\tau}_\sharp\Psi),
\end{equation}
since the squared distance map
$f(t):=W^2_2(\exp^{t}_\sharp \Phi, \exp^{t}_\sharp\Psi)$, $t\in \R$,
is not convex in general (see e.g.~\cite[Example 9.1.5]{ags})
and the fact that its right derivative at $t=0$
(corresponding to $ \bram\Phi\Psi$) is $\le 0$ according to \eqref{eq:176}
does not imply
that $f(0)\le f(t)$ for $t<0$ (corresponding to \eqref{eq:187}
for $t=-\tau$).

For these reasons, we decided to approach the investigation of dissipative
evolutions
in $\prob_2(\X)$ by the Explicit Euler method, and we defer
the study of the implicit one to a forthcoming paper.

\mysubsubsection{Plan of the paper}
As we already mentioned, our theory works in a general separable
Hilbert space $\X$: we collect some preliminary material concerning the
Wasserstein distance in Hilbert spaces
and the properties of strong-weak topology for $\prob_2(\TX)$ in Section
\ref{sec:preliminaries}.

In Section \ref{sec:tangent-bundle}, we will study the
semi-concavity properties of $W_2$ along
general deformations induced by the exponential map $\exp^\tau$ and 
we introduce and study the pairings $\bram\cdot\cdot$, $\brap\cdot\cdot$.
We will apply such tools to derive the
precise expressions of the left and right derivatives of $W_2$ along
absolutely continuous curves in $\prob_2(\X)$ in Section
\ref{subsec:left-right-derW}.

In Section \ref{sec:dissipative}, we will introduce and study
the notion of $\lambda$-dissipative \MPVF, in particular
its behaviour along geodesics (Section
\ref{subsec:dissipative-geodesic})
and its extension properties (Section \ref{subsec:extension}).
A few examples are collected in Section \ref{sec:examples}.

The last two sections contain the core of our results.
Section \ref{sec:MDE} is devoted to the notion of
$\lambda$-\EVI solutions and to their properties:
local uniqueness, stability and regularity in Section \ref{sec:locstabeuniq}, global
existence in Section \ref{sec:globexist} and barycentric characterizations in Section \ref{sec:bar}.
Section \ref{sec:EulerScheme}
contains the main estimates for the Explicit Euler scheme: the Cauchy estimates
between two discrete solutions corresponding to different step sizes
in Section \ref{subsec:exist} and the uniform error estimates
between a discrete and a $\lambda$-\EVI solution in Section \ref{subsec:error2}.

\mysubsubsection{Acknowledgments.}
G.S.~and G.E.S.~gratefully acknowledge the support of the Institute of Advanced
Study of the Technical University of Munich.
The authors thank the Department of Mathematics of the University of Pavia where this project was partially carried out. G.S. also thanks IMATI-CNR, Pavia.
G.C.~and G.S. have been supported by the MIUR-PRIN 2017
project \emph{Gradient flows, Optimal Transport and Metric Measure Structures}.  G.C. also acknowledges the partial support of the funds FARB 2016 Politecnico di Milano Prog. TDG6ATEN04.

\section{Preliminaries}
\label{sec:preliminaries}
In this section, we introduce the main concepts and results of Optimal Transport theory that will be extensively used in the rest of the paper. 
We start by listing the adopted notation.
{\small\begin{longtable}{ll}
$\bry{\Phi}$&the barycenter of $\Phi\in\prob(\TX)$ as in Definition \ref{def:bary};\\
$\mathrm B_X(x,r)$&the open ball with radius $r>0$ centered at $x\in X$;\\
$\rmC(X;Y)$&the set of continuous functions from $X$ to $Y$;\\
$\rmC_b(X)$&the set of bounded continuous real valued functions defined in
$X$;\\
$\rmC_c(X)$&the set of continuous real valued functions with compact support;\\
$\Cyl(\X)$& the space of cylindrical functions on $\X$, see Definition \ref{def:Cyl};\\
$\clo{\frF},\conv\frF[\mu]$& the sequential closure and convexification of $\frF$, see Section \ref{subsec:extension};\\
$\cloco\frF[\mu],\hat{\frF}$& sequential closure of convexification and extension of $\frF$, see Section \ref{subsec:extension};\\
$\updt \zeta, \lodt \zeta$&the right upper/lower Dini derivatives of $\zeta$, see \eqref{eq:RupDerW2};\\
$\dom(\frF)$&the proper domain of a set-valued function as in Definition \ref{def:MPVF}\\
$f_\sharp\nu$&the push-forward of $\nu\in\prob(X)$ through the map $f:X\to Y$;\\
$\Gamma(\mu,\nu)$&the set of admissible couplings between $\mu,\nu$, see \eqref{def:admplans};\\
$\Gamma_o(\mu,\nu)$&the set of optimal couplings between $\mu,\nu$,
see Definition \ref{def:wassmom};\\
$ \CondGammao\frF{\mu_0}{\mu_1}{i},\,i=0,1$&the set of optimal
couplings conditioned to $\frF$, see Definition
\ref{def:plangeodomV};\\
$\interval$&an interval of $\R$;\\
$\ii_X(\cdot)$&the identity function on a set $X$;\\
$\rI\mmu\frF$&the set of time instants $t$ s.t. $\sfx^t_\sharp\mmu$ belongs to $\dom(\frF)$, see Definition \ref{def:plangeodomV};\\
$\Lambda,\Lambda_o$&the sets of couplings as in Definition \ref{def:lambda} and Theorem \ref{thm:characterization};\\
$\rsqm{\nu}$&the $2$-nd moment of $\nu\in\prob(X)$ as in Definition \ref{def:wassmom};\\
$|\Phi|_2$&the $2$-nd moment of $\Phi\in\prob(\TX)$ as in \eqref{eq:defsqmPhi};\\
$|\frF|_2(\mu)$&the $2$-nd moment of $\frF$ at $\mu$ as in \eqref{eq:59};\\
$|\dot{\mu}|(t)$&the metric derivative of a locally absolutely continuous curve $\mu$;\\
$\prob(X)$&the set of Borel probability measures on the topological
space $X$;\\
$\prob_b(X)$&the set of Borel probability measures with bounded
support;\\
$\prob_2(X)$&the subset of measures in $\prob(X)$ with finite
quadratic moments;\\
$\prob_2^{sw}(\PX)$&the space $\prob_2(\PX)$ endowed with a weaker topology as in Definition \ref{def:swprobTX};\\
$\relcP{}{\mu}{\TX}$&the subset of $\prob_2(\TX)$ with fixed first marginal $\mu$ as in \eqref{condTanTX};\\
$\bram{\cdot}{\cdot}$, $\brap{\cdot}{\cdot}$&the pseudo scalar products as in Definition \ref{def:scalarprodop};\\
$\directionalm{\Phi}{\ttheta}t$, $\directionalp{\Phi}{\ttheta}t$&the duality pairings as in Definition \ref{def:last};\\
$\directionalm \frF\mmu t$, $\directionalp \frF\mmu t$&the duality pairings as in Definition \ref{def:frfl};\\
$\directional \frF\mmu {0+},\directional \frF\mmu {1-}$&the limiting duality pairings as in Definition \ref{def:directiona};\\
$\supp(\nu)$&the support of $\nu\in\prob(X)$;\\
$\Tan_{\mu}\prob_2(X)$&the tangent space defined in Theorem \ref{thm:tangentv};\\
$W_2(\mu,\nu)$&the $L^2$-Wasserstein distance between $\mu$ and $\nu$, see Definition \ref{def:wassmom};\\
$\X$&a separable Hilbert space;\\
$\TX$&the tangent bundle to $\X$, usually endowed with the strong-weak
topology;\\
$\sfx,\sfv,\exp^t$&the projection and exponential maps defined in \eqref{eq:projandexp};\\
$\sfx^t$&the evaluation map defined in \eqref{eq:mapxt}.\\
\end{longtable}}
In the present paper we will mostly deal with
Borel probability measures defined in (subsets of) some
separable Hilbert space endowed with the strong or a weaker topology.
The convenient setting is therefore provided by Polish/Lusin and
completely regular
topological spaces.

Recall that a topological space $X$
is Polish (resp.~Lusin) if its
topology is induced by a complete and separable metric
(resp.~is coarser than a Polish topology).
We will denote by $\prob(X)$
the set of Borel probability measures on $X$.
If $X$ is Lusin, every measure $\mu\in \prob(X)$
is also a Radon measure, i.e.~it satisfies
\begin{equation*}
  \forall \, B\subset X\text{ Borel, }\forall \, \eps >0 \quad \exists \, K\subset B \text{ compact  s.t. } \mu (B \setminus K) < \eps.
\end{equation*}
$X$ is \emph{completely regular}
if it is Hausdorff and for every closed set $C$ and point $x\in
X\setminus C$ there exists a continuous function $f: X \to [0,1]$
s.t.~$f(x)=0$ and $f(C)=\{1\}$.

Given $X$ and $Y$ Lusin spaces, $\mu \in \prob(X)$ and a Borel function $f: X \to Y$, there is a canonical way to transfer the measure $\mu$ from $X$ to $Y$ through $f$. This is called the \emph{push forward} of $\mu$ through $f$, denoted by $f_{\sharp}\mu$ and defined by $(f_{\sharp}\mu)(B) := \mu(f^{-1}(B))$ for every Borel set $B$ in $Y$, or equivalently
\[ \int_{Y} \varphi \de (f_{\sharp}\mu) = \int_X \varphi \circ f \de \mu \]
for every $\varphi$ bounded (or nonnegative) real valued Borel function
on $Y$.
A particular case occurs if $X=X_1 \times X_2$, $Y=X_i$ and $f=\pi^i$
is the projection on the $i$-th component, $i=1,2$.
In this case, $f$ is usually denoted with $\pi^i$ or $\pi^{X_i}$, and $\pi^{X_i}_{\sharp}\mu$ is called the $i$-th marginal of $\mu$.\\
This notation is particularly useful when dealing with transport plans: given $X_1$ and $X_2$ completely regular spaces and $\mu \in \prob(X_1)$, $\nu \in \prob(X_2)$, we define
\begin{equation}\label{def:admplans}
 \Gamma(\mu, \nu) := \left  \{ \ggamma \in \prob(X_1 \times X_2) \mid \pi^{1}_{\sharp} \ggamma = \mu \, , \, \pi^{2}_{\sharp} \ggamma = \nu \right \},
\end{equation}
i.e.~the set of probability measures on the product space having $\mu$ and $\nu$ as marginals.\\
On $\prob(X)$ we consider the so called \textit{narrow topology} which is the coarsest topology on $\prob(X)$ s.t. the maps $\mu \mapsto \int_X \varphi \de \mu$ are continuous for every $\varphi \in \rm\rmC_b(X)$, the space of real valued and bounded continuous functions on $X$. In this way a net $(\mu_\alpha)_{\alpha\in \mathbb A} \subset \prob(X)$ indexed by a directed set $\mathbb A$ is said to converge narrowly to $\mu \in \prob(X)$, and we write $\mu_\alpha \to \mu$ in $\prob(X)$, if
\begin{equation*} \lim_{\alpha} \int_X \varphi \de \mu_\alpha = \int_X \varphi \de \mu\quad\forall\varphi\in\rm\rmC_b(X). \end{equation*}
We recall the well known Prokhorov's theorem in the context of completely regular topological spaces (see \cite[Appendix]{Schwartz}).
\begin{theorem}[Prokhorov]\label{thm:Prok}
 Let $X$ be a completely regular topological space and let $\mathcal{F} \subset \prob(X)$ be a tight subset i.e.
\[ \text{for every } \eps >0 \text{ there exists } K_{\eps} \subset X \text{ compact s.t. }  \mu(X \setminus K_{\eps}) < \eps\,\;\forall \, \mu\in\mathcal{F}. \]
Then $\mathcal{F}$ is relatively compact in $\prob(X)$ w.r.t.~the narrow topology.
\end{theorem}
It is then relevant to know when a given $\mathcal{F} \subset
\prob(X)$ is tight. If $X$ is a Lusin completely regular topological
space,
then the set $\mathcal{F} = \{ \mu \}\subset\prob(X)$ is tight. Another trivial criterion for tightness is the following: if $\mathcal{F} \subset \prob(X_1 \times X_2)$ is s.t. $\mathcal{F}_i := \{ \pi_{\sharp}^i \ggamma \mid \ggamma \in \mathcal{F} \} \subset \prob(X_i)$ are tight for $i=1,2$, then also $\mathcal{F}$ is tight. We also recall the following useful proposition (see \cite[Remark 5.1.5]{ags}).
\begin{proposition}  Let $X$ be a Lusin
  completely regular topological space and let $\mathcal{F} \subset
  \prob(X)$. Then $\mathcal{F}$ is tight if and only if there exists
  $\varphi : X \to [0, + \infty]$
  with compact sublevels s.t.
\[ \sup_{\mu \in \mathcal{F}} \int_{X} \varphi \de \mu < + \infty.\]
\end{proposition}
We recall the so-called \emph{disintegration theorem} (see e.g. \cite[Theorem 5.3.1]{ags}).
\begin{theorem}\label{thm:disintegr}
Let $\mathbb X, X$ be Lusin completely regular topological spaces, $\mmu\in\prob(\mathbb{X})$ and $r:\mathbb{X}\to X$ a Borel map. Denote with $\mu=r_{\sharp}\mmu\in\prob(X)$. Then there exists a $\mu$-a.e. uniquely determined Borel family of probability measures $\{\mmu_x\}_{x\in X}\subset\prob(\mathbb{X})$ such that $\mmu_x(\mathbb{X}\setminus r^{-1}(x))=0$ for $\mu$-a.e. $x\in X$, and
\[\int_{\mathbb{X}}\varphi(\boldsymbol{x})\de\mmu(\boldsymbol{x})=\int_X\left(\int_{r^{-1}(x)}\varphi(\boldsymbol{x})\de\mmu_x(\boldsymbol{x})\right)\de\mu(x)\]
for every bounded Borel map $\varphi:\mathbb{X}\to\R$.
\end{theorem}
\begin{remark}\label{rmk:disintegr}
When $\mathbb{X}=X_1\times X_2$ and $r=\pi^1$, we can canonically
identify the disintegration $\{\mmu_x\}_{x\in X_1} \subset
\prob(\mathbb{X})$ of $\mmu\in\prob(X_1\times X_2)$ w.r.t.~$\mu=\pi^1_\sharp\mmu$
with a family of probability measures $ \{\mu_{x_1}\}_{x_1\in X_1} \subset
\prob(X_2)$. We write
$\mmu= \displaystyle\int_{X_1}\mu_{x_1}\,\d \mu(x_1)$.
\end{remark}

\subsection{Wasserstein distance in Hilbert spaces}
Let $X$ be a separable (possibly infinite dimensional)
Hilbert space.
We will denote by $\GX^s$ (respt.~$\GX^w$) the Hilbert space endowed with its strong (resp.~weak) topology. Notice that $\GX^w$ is
a Lusin completely regular space.
$\GX^s$ and $\GX^w$
share the same class of Borel sets
and therefore of Borel probability measures, which we will simply
denote by $\prob(\GX)$,
using
$\prob(\GX^s)$ and $\prob(\GX^w)$
only when we will refer to the correspondent topology.
Finally, if $\GX$ has finite dimension then the two topologies coincide.

We now list some properties of Wasserstein spaces and we refer to \cite[\S 7]{ags} for a complete account of this matter.
\begin{definition}\label{def:wassmom} 
  Given $\mu \in \prob(\GX)$
  we define
  \[ \sqm \mu := \int_\GX |x|^2 \de \mu(x), \qquad
    \prob_2(\GX) := \{ \mu \in \prob(\GX) \mid \rsqm\mu < + \infty \}.
  \]
  The $L^2$-Wasserstein distance between $\mu, \mu' \in \prob_2(\GX)$
  is defined as
\begin{align} \label{eq:w21} W_2^2(\mu, \mu') &:= \inf \left \{ \int_{\GX \times \GX} |x-y|^2 \de \ggamma(x,y) \mid \ggamma \in \Gamma(\mu, \mu') \right \}.
\end{align}
\end{definition}
The set of elements of $\Gamma(\mu, \mu')$ realizing the infimum in
\eqref{eq:w21} is denoted with $\Gamma_o(\mu, \mu')$.
We say that a measure $\ggamma\in \prob_2(\GX\times\GX)$ is optimal if
$\ggamma\in
\Gamma_o(\pi^1_\sharp\ggamma,\pi^2_\sharp\ggamma)$.

We will denote by $\rB\mu\varrho$
the open ball centered at $\mu$ with radius $\varrho$ in $\prob_2(X)$.
The metric space $(\prob_2(\GX), W_2)$ enjoys many interesting properties: here we only recall that it is a complete and separable metric space and that $W_2$-convergence (sometimes denoted with $\overset{W_2}{\longrightarrow}$) is stronger than the narrow convergence. In particular, given $(\mu_n)_{n\in\N}\subset\prob_2(\GX)$ and $\mu\in\prob_2(\GX)$, we have \cite[Remark 7.1.11]{ags} that
\begin{equation}
  \label{eq:important}
  \mu_n\overset{W_2}{\to}\mu,\text{ as }n\to+\infty \quad\Longleftrightarrow\quad\begin{cases}\mu_n\to\mu \text{ in }\prob(\GX^s),\\
    \rsqm{\mu_n}\to\rsqm\mu,
  \end{cases}
  \text{ as }n\to+\infty.\\
\end{equation}
Finally, we recall that sequences converging in $(\prob_2(\GX), W_2)$
are tight.
More precisely we have the following characterization of compactness
in $\prob_2(\GX)$.
\begin{lemma}[Relative compactness in $\prob_2(\GX)$]
  \label{le:compactnessP2}
  A subset $\mathcal K\subset\prob_2(\GX)$ is relatively compact 
  w.r.t.~the $W_2$-topology if and only if
  \begin{enumerate}
  \item $\mathcal K$ is tight w.r.t.~$\GX^s$,
  \item $\mathcal K$ is uniformly $2$-integrable, i.e.~
    \begin{equation}
      \label{eq:10}
      \lim_{k\to\infty}\sup_{\mu\in \mathcal K}\int_{|x|\ge k}|x|^2\,\d\mu=0.
    \end{equation}
  \end{enumerate}
\end{lemma}
\begin{proof}
  Tightness is clearly a necessary condition; concerning \eqref{eq:10}
  let us notice that the maps
  $F_k:\prob_2(\GX)\to[0,\infty)$, $F_k(\mu):=\int_{|x|\ge
    k}|x|^2\,\d\mu$ are upper semicontinuous, are decreasing
  w.r.t.~$k$,
  and converge pointwise to $0$ for every $\mu\in \prob_2(\GX)$.
  Therefore, if $\mathcal K$ is relatively compact, they converge
  uniformly to $0$ thanks to Dini's Theorem.
  
  In order to prove that (1) and (2) are also sufficient for relative
  compactness, it is sufficient to check that every sequence
  $(\mu_n)_{n\in \N}$ in
  $\mathcal K$ has a convergent subsequence.
  Applying Prokhorov Theorem \ref{thm:Prok} we can find $\mu\in \prob(\GX)$
  and a convergent subsequence
  $k\mapsto \mu_{n(k)}$ such that $\mu_{n(k)}\to \mu$ in
  $\prob(\GX^s)$.
  Since $\rsqm{\mu_n}$ is uniformly bounded, then $\mu\in \prob_2(\GX)$.
  Applying \cite[Lemma 5.1.7]{ags}, we also get
  $\lim_{k\to\infty}\rsqm{\mu_{n(k)}}=\rsqm{\mu}$
  so that $\lim_{k\to\infty}W_2(\mu_{n(k)},\mu)=0$ by \eqref{eq:important}.  
\end{proof}

\begin{definition}[Geodesics]\label{def:W2geodesic}
A curve $(\mu_t)_{t\in [0,1]}\subset \prob_2(\GX)$ is said to be a \emph{(constant speed) geodesic} if for all $0\leq s\leq t\leq 1$
we have
\[W_2(\mu_s,\mu_t)=(t-s)W_2(\mu_0,\mu_1).\]
We also say that $(\mu_t)_{t\in [0,1]}$ is a \emph{geodesic from $\mu_0$ to $\mu_1$}. We say that $A\subset\prob_2(\GX)$ is a \emph{geodesically convex} set if for any pair $\mu_0,\mu_1\in A$ there exists a geodesic $(\mu_t)_{t\in [0,1]}$ from $\mu_0$ to $\mu_1$ such that $(\mu_t)_{t\in [0,1]}\subset A$.
\end{definition}

We recall also the following useful properties of geodesics (see \cite[Theorem 7.2.1, Theorem 7.2.2]{ags}).
\begin{theorem}[Properties of geodesics] \label{theo:chargeo}
Let $\mu_0,\mu_1\in \prob_2(\GX)$ and $\mmu\in \Gamma_o(\mu_0,\mu_1)$.
Then $(\mu_t)_{t\in [0,1]}$ defined by
\begin{equation}\label{eq:charact_geodesic}
\mu_t := (\sfx^t)_{\sharp} \mmu,\quad t\in[0,1],
\end{equation}
is a (constant speed) geodesic from $\mu_0$ to $\mu_1$, where $\sfx^t:X^2\to X$ is given by, $\sfx^t(x_0,x_1):=(1-t)x_0+tx_1$. Conversely, any (constant speed) geodesic  $(\mu_t)_{t\in [0,1]}$ from $\mu_0$ to $\mu_1$ admits the representation~\eqref{eq:charact_geodesic} for a suitable plan $\mmu\in \Gamma_o(\mu_0,\mu_1)$.\\
Finally, if $(\mu_t)_{t\in [0,1]}$ is a geodesic connecting $\mu_0$ to $\mu_1$, then for every $t \in (0,1)$ there exists a unique optimal plan between $\mu_0$ and $\mu_t$ (resp.~between $\mu_t$ and $\mu_1$) and it is concentrated on a map.
\end{theorem}

\medskip

We define moreover the analogous of $\rmC^{\infty}_c(\R^d)$ when we have $\GX$ in place of $\R^d$.
\begin{definition}[$\Cyl(\GX)$]\label{def:Cyl}
  We denote by $\Pi_d(\GX)$ the space of linear maps $\pi:\GX\to
  \R^d$ of the form $\pi(x)=(\la x,e_1\ra,\cdots,\la x,e_d\ra )$
  for an orthonormal set $\{e_1,\cdots,e_d\}$ of $\GX$.
A function $\varphi: \GX \to \R$ belongs to the space of cylindrical functions on $\GX$, $\Cyl(\GX)$, if it is of the form
\[ \varphi = \psi \circ \pi\]
where $\pi\in \Pi_d(\GX)$ and $\psi \in \mathrm C^\infty_c(\R^d)$.
\end{definition}
We recall the following result (see \cite[Theorem 8.3.1, Proposition
8.4.5 and Proposition 8.4.6]{ags}) characterizing
locally absolutely continuous curves in $\prob_2(\GX)$ defined
in a
(bounded or unbounded) open interval $\interval\subset \R$.
We use the
equivalent notation $\mu(t)\equiv\mu_t$ for the evaluation at time
$t\in \interval$ of a map $\mu:\interval\to \prob_2(\GX)$.

\begin{theorem}[Wasserstein velocity field]
  \label{thm:tangentv}
Let $\mu:\interval \to\prob_2(\GX)$ be a locally absolutely
continuous curve defined in an open interval $\interval\subset \R$.
There exists a Borel vector field
$\vv:\interval\times \GX\to \GX$
and
a set $A(\mu) \subset \interval$ with
$\mathcal{L}(\interval \setminus A(\mu))=0$ such that
for every $t\in A(\mu)$ 
\begin{displaymath}
  \begin{gathered}
  \vv_t\in\Tan_{\mu_t}\prob_2(\GX) :={} \overline{\{ \nabla \varphi \mid
    \varphi \in \Cyl(X) \}}^{L^2_{\mu_t}(X;X)},\\
  \int_{\GX} |\vv_t|^2\,\d\mu_t=|\dot \mu_t|^2=\lim_{h\to
    0}\frac{W_2^2(\mu_{t+h},\mu_t)}{h^2},
\end{gathered}
\end{displaymath}
and the continuity equation
\[\partial_t\mu_t+\nabla\cdot(\vv_t\mu_t)=0\]
holds in the sense of distributions in $\interval\times \GX$.
Moreover, $\vv_t$ is uniquely determined
in $L^2_{\mu_t}(\GX;\GX)$ for $t\in A(\mu)$ and 
\begin{equation}
 \lim_{h \to 0} \frac{W_2((\ii_X+h\vv_t)_{\sharp}\mu_t,
   \mu_{t+h})}{|h|} =0 \quad \text{for every }t \in A(\mu).\label{eq:74}
\end{equation}
\end{theorem}
We conclude this section with a useful property concerning the
upper derivative of the Wasserstein distance, which in fact holds in
every metric space.
\begin{lemma}\label{lem:aeb}
  Let $\mu: \interval \to \prob_2(\X)$, $\nu \in \prob_2(\X)$,
   $t
   \in \interval$, $\ssigma \in \Gamma_o(\mu_t, \nu)$,
   and consider the
   constant speed geodesic $(\nu_{t,s})_{s \in [0,1]}$
   defined by $\nu_{t, s}: =
   (\sfx^s)_{\sharp}\ssigma$ for every $s \in [0,1]$.
   The upper right and left Dini derivatives 
   $b^{\pm}:(0,1] \to \mathbb{R}$ defined by
   \begin{equation*}
     \begin{aligned}
       b^+(s):={}& \frac{1}{2s} \limsup_{h \downarrow 0}
       \frac{W_2^2(\mu_{t+h}, \nu_{t,s}) - W_2^2(\mu_t,
         \nu_{t,s})}{h},\\
        b^-(s):={}& \frac{1}{2s} \limsup_{h
         \downarrow 0} \frac{W_2^2(\mu_{t}, \nu_{t,s}) -
         W_2^2(\mu_{t-h}, \nu_{t,s})}{h}
     \end{aligned}
  \end{equation*}
  are respectively decreasing and increasing in $(0,1]$.
\end{lemma}
\begin{proof}
Take $0\le s'<s\le 1$. Since $(\nu_{t,s})_{s \in [0,1]}$ is a constant speed geodesic from $\mu_t$ to $\nu$, we have
\[ W_2(\mu_t, \nu_{t,s}) = W_2(\mu_t, \nu_{t,s'}) + W_2(\nu_{t,s'}, \nu_{t,s}),\]
then, by triangular inequality
\begin{align*}
W_2(\mu_{t+h}, \nu_{t,s}) - W_2(\mu_t, \nu_{t,s}) &\le W_2(\mu_{t+h}, \nu_{t,s'}) + W_2(\nu_{t,s'}, \nu_{t,s})  - W_2(\mu_t, \nu_{t,s}) \\
&= W_2(\mu_{t+h}, \nu_{t,s'}) - W_2(\mu_t, \nu_{t,s'}).
\end{align*}
Dividing by $h>0$ and passing to the limit as $h\downarrow0$
we obtain that
the function $a:[0,1] \to \mathbb{R}$ defined by
\[ a^+(s):= \limsup_{h \downarrow 0} \frac{W_2(\mu_{t+h}, \nu_{t,s}) - W_2(\mu_t, \nu_{t,s})}{h} \]
is decreasing.
It is then sufficient to observe that for $s>0$ 
\[b^+(s) = a^+(s) \frac{W_2(\mu_t, \nu_{t,s})}{s}= a^+(s) W_2(\mu_t, \nu).\]
The monotonicity property of $b^-$ follows by the same argument.
\end{proof}

\subsection{A strong-weak topology on measures in product spaces}
\label{subsec:sw}
Let us consider the case when $\GX=\X\times \Y$ where $\X,\Y$ are
separable Hilbert spaces.
$\GX$ is naturally endowed with the product Hilbert norm and
$\prob_2(\GX)$ with the corresponding topology induced by the
$L^2$-Wasserstein distance.
However, it will be extremely useful to endow $\prob_2(\GX)$ with a weaker topology
which is related to the strong-weak topology on $\GX$, i.e.~the
product topology of $\X^s\times \Y^w$.
We follow the approach of \cite{NaldiSavare}, to which we 
refer for the proofs of the results presented in this section.

In order to define the topology, we consider the space $\testsw \PX$ of test functions
$  \zeta:\PX\to \R$ such that
\begin{gather*}
  \label{eq:36}
  \zeta\text{ is sequentially continuous in $\X^s\times \Y^w$,}\\
  \forall\,\eps>0\ \exists\,A_\eps\ge 0: |\zeta(x,y)|\le
  A_\eps(1+|x|_\X^2)+\eps |y|_\Y^2\quad\text{for every }(x,y)\in \PX.
\end{gather*}
Notice in particular that functions in $\testsw \PX$ have quadratic
growth. We endow $\testsw X$ with the norm
\begin{equation*}
  \|\zeta\|_{\testsw X}:=\sup_{(x,y)\in X}\frac{|\zeta(x,y)|}{1+|x|_\X^2+|y|_\Y^2}.
\end{equation*}
\begin{remark}
  When $\Y$ is finite dimensional, \eqref{eq:36} is equivalent to the
  continuity of $\zeta$.
\end{remark}

\begin{lemma}
  $(\testsw \PX,\|\cdot\|_{\testsw
  \PX})$ is a Banach space.
\end{lemma}
  \renewcommand{\mmu}{\boldsymbol \mu}
  \begin{definition}[Topology of $\prob_2^{sw}(\PX)$, \cite{NaldiSavare}]\label{def:swprobTX}
    We denote by $\prob_2^{sw}(\PX)$ the space $\prob_2(\PX)$ endowed with the coarsest topology which makes the following functions continuous
    \begin{equation*}
      \mmu\mapsto \int\zeta(x,y)\,\d\mmu(x,y),\quad \zeta \in \testsw \PX.
    \end{equation*}    
  \end{definition}
  It is obvious that the topology of $\prob_2(\PX)$ is finer than the
  topology of $\prob_2^{sw}(\PX)$ and the latter is finer than the
  topology of $\prob(\X^{s}\times \Y^w)$.
  It is worth noticing that 
  \begin{equation*}
    \text{any bounded bilinear form $B:\PX\to \R$ belongs to }\testsw \PX,
  \end{equation*}
  so that for every net $(\mmu_\alpha)_{\alpha\in \mathbb A} \subset \prob(\PX)$
  indexed by a directed set $\mathbb A$, we have
  \begin{equation}
    \label{eq:43}
    \lim_{\alpha\in \mathbb A}\mmu_\alpha=\mmu\quad
    \text{in }\prob_2^{sw}(\PX)\quad
    \Rightarrow\quad
    \lim_{\alpha\in \mathbb A}\int B \,\d\mmu_\alpha=
    \int B \,\d\mmu.
  \end{equation}
  The following proposition justifies the
  interest in the $\prob_2^{sw}(\PX)$-topology.
  \begin{proposition}
    \label{prop:finalmente}
    \
    \begin{enumerate}
    \item If $(\mmu_\alpha)_{\alpha\in\mathbb A}\subset \prob_2(\PX)$ is
      a net indexed by the directed set $\mathbb A$ and $\mmu\in
      \prob_2(\PX)$ satisfies
      \begin{enumerate}
      \item $\mmu_\alpha\to\mmu$ in $\prob(\X^s\times \Y^w)$,
      \item $\displaystyle \lim_{\alpha\in \mathbb A}\int |x|_\X^2\,\d\mmu_\alpha(x,y)=\int |x|_\X^2\,\d\mmu(x,y)$,
      \item $\displaystyle \sup_{\alpha\in\mathbb A} \int|y|_\Y^2\,\d\mmu_\alpha(x,y)<\infty$,
      \end{enumerate}
      then $\mmu_\alpha\to\mmu$ in $\prob_2^{sw}(\PX)$.
      The converse property holds for sequences: if $\mathbb A=\N$
      and $\mmu_n\to\mmu$ in $\prob_2^{sw}(\PX)$ as $n\to\infty$ then
      properties (a), (b), (c) hold.
    \item For every compact set $\mathcal K\subset \prob_2(\X^s)$ and
      every constant $c<\infty$ the sets
      \begin{equation*}
      \mathcal K_c:=\Big\{\mmu\in \prob_2(\PX):\pi^{\X}_\sharp\mmu\in
      \mathcal K,\quad
      \int |y|_\Y^2\,\d\mmu(x,y)\le c\Big\}
    \end{equation*}
    are compact and metrizable in
    $\prob_2^{sw}(\PX)$ (in particular they are sequentially compact).
    \end{enumerate}
  \end{proposition}

  It is worth noticing that the topology $\prob_2^{ws}(\X\times \Y)$ is
  strictly weaker than $\prob_2(\X\times \Y)$ even when $\Y$ is finite
  dimensional. In fact, $\testsw\PX$ does not contain
  the quadratic function $(x,y)\mapsto |y|_\Y^2$, so that convergence of
  the quadratic moment w.r.t.~$y$ is not guaranteed.

\section[Directional derivatives of the Wasserstein distance]{Directional derivatives and probability measures on the tangent bundle}
\label{sec:tangent-bundle}
From now on, we will denote by $\X$ a separable Hilbert space with
norm $|\cdot|$ and scalar product $\la\cdot,\cdot\ra$.
We denote by 
$\TX$ the tangent bundle to $\X$, which is identified with the set $\X\times\X$
with the induced norm $|(x,v)|:=\big(|x|^2+|v|^2\big)^{1/2}$ and the 
strong-weak topology of $\X^s \times \X^w$(i.e.~the product of the strong topology on the
first component and the weak topology on the second one).
We will denote by $\sfx,\sfv:\TX\to\X$ the projection maps and by
$\exp^t:\TX\to\X$ the exponential map defined by
\begin{equation}\label{eq:projandexp}
  \sfx(x,v):=x,\quad \sfv(x,v)=v,\quad
  \exp^t(x,v):=x+tv.
\end{equation}
The set $\prob(\TX)$ is defined thanks to the identification of
$\TX$ with $\X\times \X$ and it is endowed with
the narrow topology induced by the strong-weak topology in $\TX$. For
$\Phi \in \prob(\TX)$ we define
\begin{equation}\label{eq:defsqmPhi}
 |\Phi|_2^2:= \int_{\TX} |v|^2 \de \Phi(x,v).
\end{equation}
We denote by $\prob_2(\TX)$ the subset of $\prob(\TX)$ of
measures for which $\int \big(|x|^2+|v|^2\big)\,\d\Phi<\infty$
endowed with the topology of $\prob_2^{sw}(\TX)$ as in Section
\ref{subsec:sw}. If $\mu \in \prob(\X)$ we will also
consider 
\begin{equation}\label{condTanTX}
 \relcP{}\mu\TX := \Big\{ \Phi \in \prob(\TX) \mid \sfx_{\sharp}\Phi
  =
  \mu \Big \},\quad
\relcP 2{\mu}\TX := \Big\{\Phi\in \relcP{}\mu\TX: |\Phi|_2<\infty\Big\}.
\end{equation}
We will also deal with the product space $\X^2$: we will use the
notation
\begin{equation}\label{eq:mapxt}
  \sfx^t:\X^2\to \X,\quad
  \sfx^t(x_0,x_1):=(1-t)x_0+tx_1,\quad t\in[0,1].
\end{equation}
If $\vv\in L^2_\mu(\X;\X)$
we can consider the probability
\begin{equation}
  \label{eq:24}
  \Phi_{\mu,\vv}:=(\ii_\X,\vv)_\sharp\mu\in \relcP 2\mu\TX.
\end{equation}
In this case we will say that $\Phi$ is concentrated on the graph of
the map $\vv$.
More generally, given a Borel family
of probability measures $(\Phi_x)_{x\in \X}\subset \prob_2(\X)$
satisfying
\begin{equation}
  \label{eq:25}
  \int \Big(\int |v|^2\,\d\Phi_x(v)\Big)\,\d\mu(x)<\infty
\end{equation}
we can consider the probability
\begin{equation}
  \label{eq:26}
  \Phi=
  \int_{\X}\Phi_x\,\d\mu(x)
  \in \relcP 2\mu\TX.
\end{equation}
Conversely, every $\Phi\in \relcP 2\mu\TX$ can be disintegrated
by a Borel family $(\Phi_x)_{x\in \X}\subset \prob_2(\X)$
satisfying \eqref{eq:25} and \eqref{eq:26}. $\Phi$ can be associated
to
a vector field $\vv\in L^2_\mu(\X;\X)$ 
if and only if
for $\mu$-a.e.~$x\in \X$ $\Phi_x=\delta_{\vv(x)}$.
Recalling the disintegration Theorem \ref{thm:disintegr} and Remark \ref{rmk:disintegr}, we give the following definition.
\begin{definition}\label{def:bary}
Given $\Phi\in\relcP 2\mu\TX$, the \emph{barycenter of $\Phi$} is the
function
$\bry{\Phi}\in L^2_\mu(\X;\X)$ defined by
\[\bry{\Phi}(x):=\int_\X v\de\Phi_x(v) \quad\text{for }\mu\text{-a.e. }x\in \X,\]
where $\{\Phi_x\}_{x\in \X}\subset\prob_2(\X)$ is the disintegration of $\Phi$ w.r.t.~$\mu$.\end{definition}
\begin{remark}\label{rmk:barycenter}
Notice that, by the linearity of the scalar product, we get the following identity which will be useful later
\begin{equation}
 \int_{\X} \scalprod{\zzeta(x)}{\bry\Phi(x)} \de \mu(x) =
  \int_{\TX} \scalprod{\zzeta(x)}{v} \de \Phi(x,v) \quad \forall \,
  \zzeta \in L^2_\mu(\X;\X).
  \label{eq:30}
\end{equation}
\end{remark}
\subsection{Directional derivatives of the Wasserstein distance and
  duality pairings}
Our starting point is a relevant semi-concavity property of the
function
\begin{equation}
 f(s,t):= \frac{1}{2} W_2^2(\exp^s_{\sharp} \Phi_0,
 \exp^t_\sharp\Phi_1),\quad s,t\in \R,\label{eq:11}
\end{equation}
with
$\Phi_0,\Phi_1\in\prob_2(\TX)$.
We first state an auxiliary result,
whose proof is based on \cite[Proposition 7.3.1]{ags}.
\begin{lemma}
  \label{prop:731v2} Let $\Phi_0, \Phi_1 \in \prob_2(\TX)$, $s,t \in \R$,
  and let $\ttheta^{s,t} \in \Gamma(\exp^s_{\sharp}\Phi_0,
  \exp^t_{\sharp} \Phi_1)$. Then there exists $\Ttheta^{s,t} \in
  \Gamma(\Phi_0, \Phi_1)$
  such that $(\exp^s, \exp^t)_{\sharp}\Ttheta^{s,t}= \ttheta^{s,t}$.
\end{lemma}
\begin{proof} Define, for every $r,s,t \in \R$, 
  \[ \Sigma^r : \TX \to \TX,\quad \Sigma^r(x, v) = (\exp^r(x, v), v);
    \quad \Lambda^{s,t}: \TX \times \TX  \to \TX \times \TX,\quad
  \Lambda^{s,t}:=(\Sigma^s,\Sigma^t).\]
Consider the probabilities $(\Sigma^s)_{\sharp}\Phi_0,
(\Sigma^t)_{\sharp}\Phi_1$
and $\ttheta^{s,t}$. They are constructed in such a way that there
exists
$\Ppsi^{s,t} \in \prob(\TX \times \TX)$ s.t. 
\[ (\sfx^0,\sfv^0)_{\sharp} \Ppsi^{s,t} = (\Sigma^s)_{\sharp}\Phi_0,
  \quad (\sfx^1,\sfv^1)_{\sharp} \Ppsi^{s,t} =
  (\Sigma^t)_{\sharp}\Phi_1,
  \quad (\sfx^0,\sfx^1)_{\sharp} \Ppsi^{s,t}= \ttheta^{s,t},\]
where we adopted the notation $\sfx^i(x_0,v_0,x_1,v_1):=x_i$ and
$\sfv^i(x_0,v_0,x_1,v_1):=v_i$, $i=0,1$.
We conclude by taking $\Ttheta^{s,t}:= (\Lambda^{-s,-t})_{\sharp} \Ppsi^{s,t}$. 
\end{proof}
\begin{proposition} \label{prop:concavity} Let $\Phi_0,\Phi_1 \in
  \prob_2(\TX)$ with $\mu_1=\sfx_\sharp\Phi_1$ and
  $\varphi^2:=|\Phi_0|_2^2+|\Phi_1|_2^2$,
  let $f:\R^2\to\R$ be the function defined by
  \eqref{eq:11}
  and let 
   $h,g:\R\to\R$ be defined by 
  \begin{equation}
    \label{eq:115}
    h(s):=f(s,s)= \frac 12 W_2^2(\exp^s_{\sharp}\Phi_0,\exp^s_\sharp\Phi_1),\quad
    g(s):=f(s,0)=  \frac{1}{2} W_2^2(\exp^s_{\sharp} \Phi_0,
    \mu_1),\quad
    s\in \R.
  \end{equation}
  \begin{enumerate}
  \item
    The function $(s,t)\mapsto f(s,t)-\frac12\varphi^2( s^2+t^2)$ is
    concave, i.e.~it holds
    \begin{equation}
    \begin{aligned}
      f((1-\alpha)s_0 + \alpha s_1, (1-\alpha)t_0 + \alpha t_1) &\ge
      (1-\alpha)f(s_0,t_0) + \alpha f(s_1,t_1)
      \\&\quad-
      \frac 12\alpha(1-\alpha)\Big[(s_1-s_0)^2 +(t_1-t_0)^2 \Big]\varphi^2
    \end{aligned}\label{eq:12}
  \end{equation}
  for every $s_0, s_1 ,t_0,t_1\in \R$ and every $ \alpha \in [0,1]$.
  \item
    The function $s\mapsto h(s)-\varphi^2 s^2$ is concave.
  \item the function $s\mapsto g(s)-\frac 12s^2|\Phi_0|_2^2$ 
is concave.
\end{enumerate}
\end{proposition}
\begin{proof}
  Let us first prove \eqref{eq:12}.
  We set $s:= (1-\alpha)s_0 + \alpha s_1$,
  $t:=(1-\alpha)t_0+\alpha t_1$ and we apply Lemma \ref{prop:731v2}
  to find $\Ttheta \in \Gamma(\Phi_0,\Phi_1)$ such that
  $(\exp^s,\exp^t)_\sharp \Ttheta\in \Gamma_o(\exp^s_{\sharp}\Phi_0,
  \exp^t_\sharp\Phi_1)$.
Then, recalling the Hilbertian identity
\begin{equation*}
|(1-\alpha)a + \alpha b|^2 = (1-\alpha) |a|^2+\alpha |b|^2-\alpha(1-\alpha) |a-b|^2, \quad a,b\in \X,
\end{equation*}
we have
\begin{align*}
W_2^2&( \exp^{s}_{\sharp} \Phi_0, \exp^t_\sharp \Phi_1) = \int
       |x_0+sv_0-(x_1+tv_1)|^2
       \de \Ttheta =
  \\&=
  \int
  |(1-\alpha)(x_0+s_0v_0)+\alpha
  (x_0+s_1v_0)-
  (1-\alpha)(x_1+t_0v_1)-\alpha
  (x_1+t_1v_1)|^2
  \de \Ttheta 
  \\&=(1-\alpha)\int|x_0+s_0v_0-(x_1+t_0v_1)|^2\de\Ttheta 
  +\alpha\int|x_0+s_1v_0-(x_1+t_1v_1)|^2\de\Ttheta 
  \\&\quad-\alpha(1-\alpha) \int |(s_1-s_0)v_0+(t_1-t_0)v_1|^2 \de \Ttheta  \\
     & \ge (1-\alpha) W_2^2(\exp_{\sharp}^{s_0}
       \Phi_0,\exp^{t_0}_\sharp \Phi_1) +
       \alpha
       W_2^2(\exp_{\sharp}^{s_1}
       \Phi_0,
       \exp^{t_1}_\sharp\Phi_1)
       \\&\quad-
  \alpha(1-\alpha)\Big((s_1-s_0)^2+(t_1-t_0)^2\Big)
  \Big(\int        |v_0|^2        \de        \Phi_0+
       \int |v_1|^2\,\d\Phi_1\Big).
\end{align*}
which is the thesis.
Claims (2) and (3) follow as particular cases when $t=s$ or
$t=0$.
\end{proof}
Semi-concavity is a useful tool to guarantee the existence of 
one-sided partial derivatives at $(0,0)$: for every $\alpha,\beta\in \R$ we have (see e.g.~\cite[Ch.~VI, Prop.~1.1.2]{HUL}) that
\begin{align*}
  f_r'(\alpha,\beta)&=\lim_{\varrho\downarrow0}\frac{f(\alpha
                      \varrho,\beta \varrho)-f(0,0)}\varrho=
                      \sup_{\varrho>0}\frac{f(\alpha\varrho,\beta \varrho)-f(0,0)}\varrho-\frac{\varrho\varphi^2}2(\alpha^2+\beta^2),\\
  f_l'(\alpha,\beta)&=\lim_{\varrho\downarrow0}\frac{f(0,0)-f(-\alpha \varrho,-\beta \varrho)}\varrho=
                      \inf_{\varrho>0}\frac{f(0,0)-f(-\alpha \varrho,-\beta \varrho)}\varrho+\frac{\varrho\varphi^2}2(\alpha^2+\beta^2).
\end{align*}
$f_r'$ (resp.~$f'_l$) is a concave (resp.~convex)
and positively $1$-homogeneous function, i.e.~a superlinear
(resp.~sublinear) function. They satisfy
\begin{align}
  \label{eq:14}
    f_r'(-\alpha,-\beta)&=-f'_l(\alpha,\beta),\quad
    f_l'(\alpha,\beta)\ge f_r'(\alpha,\beta)\quad\text{for every
    }\alpha,\beta\in \R,\\
    \label{eq:15}
    f_r'(\alpha,\beta)&\ge \alpha f_r'(1,0)+\beta f_r'(0,1)\quad
    \text{for every }\alpha,\beta\ge0,\\ \notag
    f(s,t)&\le f(0,0)+f'_r(s,t)-\frac{\varphi^2}2(s^2+t^2)\quad
    \text{for every }s,t\in \R.
  \end{align}
  Notice moreover that
  \begin{equation*}
    f_r'(1,0)=g'_r(0)=\lim_{\varrho\downarrow0}\frac{g(\varrho)-g(0)}\varrho
  \end{equation*}
  where $g$ is the function defined in \eqref{eq:115}; a similar
  representation holds for $f_l'(1,0)$.
  We introduce the following notation for $f'_r$, $f'_l$, $g'_r$ and $g'_l$.
  \begin{definition}
    \label{def:scalarprodop}
  Let $\mu_0,\mu_1 \in \prob_2(\X)$, $\Phi_0 \in \relcP2{\mu_0}\TX$ and
  $\Phi_1\in \relcP2{\mu_1}\TX$. Recalling the definitions
  of $f$ and $g$ given by \eqref{eq:11} and \eqref{eq:115}, we define
\begin{align*}
  \bram{\Phi_0}{\mu_1}
  &:= g'_r(0)=f_r'(1,0)=
    \lim_{s \downarrow 0} \frac{W_2^2(\exp^s_{\sharp}\Phi_0, \mu_1) -
    W_2^2(\mu_0, \mu_1)}{2s}, \\
  \brap{\Phi_0}{\mu_1}
  &:=
    g'_l(0)=f_l'(1,0)=
    \lim_{s \downarrow 0} \frac{W_2^2(\mu_0, \mu_1)-
    W_2^2(\exp^{-s}_{\sharp}\Phi_0, \mu_1)}{2s},
\intertext{and analogously }
  \bram{\Phi_0}{\Phi_1}
  &:=f'_r(1,1)=
    \lim_{t \downarrow 0} \frac{W_2^2(\exp^t_{\sharp}\Phi_0,
    \exp^t_{\sharp}\Phi_1) -    W_2^2(\mu_0, \mu_1)}{2t},\\
  \brap{\Phi_0}{\Phi_1}
  &:= f_l'(1,1)=\lim_{t \downarrow 0}
    \frac{W_2^2(\mu_0, \mu_1)-W_2^2(\exp^{-t}_{\sharp}\Phi_0, \exp^{-t}_{\sharp}\Phi_1)}{2t}. 
\end{align*}
\end{definition}
\begin{remark} \label{rem:reduction}
  Notice that
  $\bram{\Phi_0}{\mu_1} = \bram{\Phi_0}{\Phi_{\mu_1}}$ and $\brap{\Phi_0}{\mu_1} = \brap{\Phi_0}{\Phi_{\mu_1}}$,
  where
  \[\Phi_{\mu_1} = (\ii_\X ,0)_{\sharp}\mu_1 \in \prob_2(\TX).\]
  Moreover, using the notation
  \begin{equation}
    \label{eq:20}
    -\Phi:=J_\sharp \Phi,\quad
    \Phi\in \prob(\TX),\ J(x,v):=(x,-v),
  \end{equation}
  we have
\[\bram{-\Phi_0}{-\Phi_1}=-\brap{\Phi_0}{\Phi_1},\quad\text{
    and }
  \quad\bram{-\Phi_0}{\mu_1}=-\brap{\Phi_0}{\mu_1}.\]
In particular, the properties of $\brap{\cdot}{\cdot}$ (in $\prob_2(\TX) \times \prob_2(\TX) $ or $\prob_2(\TX) \times \prob_2(\X)$) and the ones of $\bram{\cdot}{\cdot}$ in $\prob_2(\TX) \times \prob_2(\X)$
can be easily derived by the corresponding
ones of $\bram{\cdot}{\cdot}$ in $\prob_2(\TX) \times \prob_2(\TX) $.
\end{remark}
Recalling \eqref{eq:15} and \eqref{eq:14} we obtain the following result.
\begin{corollary}
  \label{lem:control} For every $\mu_0, \mu_1 \in \prob_2(\X)$ and for every $\Phi_0\in\relcP2{\mu_0}\TX$, $\Phi_1\in\relcP2{\mu_1}\TX$, it holds
\[ \bram{\Phi_0}{\mu_1} + \bram{\Phi_1}{\mu_0} \le \bram{\Phi_0}{\Phi_1} \quad\text{ and }\quad\brap{\Phi_0}{\mu_1} + \brap{\Phi_1}{\mu_0} \ge \brap{\Phi_0}{\Phi_1}. \]
\end{corollary}

  Let us now show an important equivalent characterization of the 
  quantities we have just introduced.
  As usual we will denote by $\sfx^0,\sfv^0,\sfx^1:\TX\times \X\to\X$ the projection maps
  of a point $(x_0,v_0,x_1) $ in $\TX\times \X$
  (and similarly for $\TX\times \TX$ with $\sfx^0,\sfv^0,\sfx^1, \sfv^1$).
  
  First of all we introduce the following sets.
\begin{definition}\label{def:lambda}
  For every $\Phi_0  \in \prob(\TX)$ with $\mu_0=\sfx_\sharp\Phi_0$ and 
  $\mu_1 \in \prob_2(\X)$
  we set
  \begin{equation*}
  \Lambda(\Phi_0, \mu_1):= \left \{ \ssigma \in \Gamma(\Phi_0, \mu_1)
    \mid
    (\sfx^0,\sfx^1)_{\sharp}\ssigma \in  \Gamma_o(\mu_0, \mu_1) \right
      \}.
  \end{equation*}
Analogously, for every $\Phi_0, \Phi_1 \in \prob(\TX)$ with
$\mu_0=\sfx_\sharp\Phi_0$ and 
$\mu_1=\sfx_\sharp\Phi_1$ in $\prob_2(\X)$
we set
\begin{equation*}
  \Lambda(\Phi_0, \Phi_1):= \left \{ \Ttheta \in \Gamma(\Phi_0, \Phi_1)
                         \mid (\sfx^0,\sfx^1)_{\sharp} \Ttheta\in
                          \Gamma_o(\mu_0, \mu_1)
    \right \}.
  \end{equation*}
\end{definition}
In the following proposition and subsequent corollary, we provide a
useful characterization of the pairings $\bram{\cdot}{\cdot}$ and $\brap{\cdot}{\cdot}$. 
\begin{theorem}
  \label{thm:characterization} For every
  $\Phi_0, \Phi_1 \in \prob_2(\TX)$ and $\mu_1 \in \prob_2(\X)$ we have
\begin{align}
\label{eq:21}\bram{\Phi_0}{\mu_1} & = \min \left \{ \int_{\TX \times
                                   \X} \scalprod{x_0-x_1}{v_0} \de \ssigma
                                   \mid \ssigma \in \Lambda(\Phi_0,
                                   \mu_1) \right \}, \\
  \label{eq:22}
\bram{\Phi_0}{\Phi_1} & = \min \left \{ \int_{\TX \times \TX} \scalprod{x_0-x_1}{v_0-v_1} \de \Ttheta \mid \Ttheta \in \Lambda(\Phi_0, \Phi_1) \right \}.
\end{align}
We denote by 
$\Lambda_o(\Phi_0, \mu_1)$ (resp.~$\Lambda_o(\Phi_0, \Phi_1)$)
the subset of $\Lambda (\Phi_0,\mu_1)$ (resp.~$\Lambda(\Phi_0,\Phi_1)$)
where the minimum in \eqref{eq:21} (resp.~\eqref{eq:22}) is attained.
\end{theorem}
\begin{proof} 
  First, we recall that the minima in the right hand side are attained
  since $\Lambda(\Phi_0,\mu_1)$ and
  $\Lambda(\Phi_0,\Phi_1)$ are compact subsets
  of $\prob_2(\TX\times \X)$ and $\prob_2(\TX\times \TX)$ respectively
  by Lemma \ref{le:compactnessP2} and the integrands are
  continuous functions with quadratic growth. 
Thanks to Remark \ref{rem:reduction}, we only need to prove the second
equality. For every $\Ttheta \in \Lambda(\Phi_0, \Phi_1)$ and setting
$\mu_0=\sfx_\sharp\Phi_0$, $\mu_1=\sfx_\sharp\Phi_1$, we have
\begin{align*}
  W_2^2 &(\exp^t_{\sharp}(\Phi_0), \exp^t_{\sharp}(\Phi_1))
  \le \int_{\TX \times \TX} | (x_0-x_1) +t(v_0-v_1)|^2 \de \Ttheta\\
  &= \int_{\X^2} |x_0-x_1|^2 \de (\sfx^0, \sfx^1)_{\sharp}\Ttheta+
    2t \int_{\TX \times \TX} \scalprod{x_0-x_1}{v_0-v_1} \de \Ttheta+
        t^2 \int_{\X^2} |v_0-v_1|^2 \de \Ttheta \\
&= W_2^2(\mu_0,\mu_1) + 2t \int_{\TX \times \TX} \scalprod{x_0-x_1}{v_0-v_1} \de \Ttheta
    + t^2 \int_{\X^2} |v_0-v_1|^2 \de \Ttheta.
\end{align*}
and this immediately implies 
\begin{align*}
  \bram{\Phi_0}{\Phi_1}\le
\min \left \{ \int_{\TX \times \TX} \scalprod{x_0-x_1}{v_0-v_1} \de \Ttheta \mid \Ttheta \in \Lambda(\Phi_0, \Phi_1) \right \}. 
\end{align*} 
In order to prove the converse inequality,
thanks to Lemma \ref{prop:731v2}, for every $t>0$ we can find $\Ttheta_t \in \Gamma(\Phi_0, \Phi_1)$ s.t. 
\[ (\exp^t, \exp^t)_{\sharp}\Ttheta_t \in \Gamma_o(\exp^t_{\sharp}\Phi_0, \exp^t_{\sharp}\Phi_1). \] 
Then
\begin{equation}\label{proofminbra}
\begin{split}
\frac{W_2^2(\exp^t_{\sharp}\Phi_0, \exp^t_{\sharp}\Phi_1) - W_2^2(\mu_0,\mu_1)}{2t} &\ge \frac{1}{2t}\int_{\TX \times \TX } |(x_0-x_1)+t(v_0-v_1)|^2 \de \Ttheta_t \\
&\quad-\frac{1}{2t} \int_{\TX \times \TX} |x_0-x_1|^2 \de \Ttheta_t \\
& \ge \int_{\TX \times \TX} \scalprod{x_0-x_1}{v_0-v_1} \de \Ttheta_t.
\end{split}
\end{equation}
Since $\Gamma(\Phi_0, \Phi_1)$ is compact in $\prob_2(\TX\times \TX)$,
there exists a vanishing sequence $k\mapsto t(k)$
and $\Ttheta \in \Gamma(\Phi_0, \Phi_1)$ s.t. $\Ttheta_{t(k)} \to \Ttheta$
in
$\prob_2(\TX\times \TX)$.
Moreover it holds $(\exp^{t(k)}, \exp^{t(k)})_{\sharp}\Ttheta_{t(k)} \to
(\sfx^0,\sfx^1)_{\sharp}\Ttheta$ in $\prob(\TX \times \TX)$
so that $(\sfx^0,\sfx^1)_{\sharp}\Ttheta \in  \Gamma_o(\mu_0, \mu_1)$,
and therefore $\Ttheta \in \Lambda(\Phi_0, \Phi_1)$.
The convergence in $\prob_2(\TX\times\TX)$ yields 
\[\lim_k \int_{\TX\times \TX}\scalprod{x_0-x_1}{v_0-v_1} \de \Ttheta_{t(k)} = \int_{\TX\times \TX}\scalprod{x_0-x_1}{v_0-v_1} \de \Ttheta,\]
so that, passing to the limit in \eqref{proofminbra} along the
sequence $t(k)$, we obtain
\begin{align*}
  \bram{\Phi_0}{\Phi_1}\ge 
      \int_{\TX \times \TX} \scalprod{x_0-x_1}{v_0-v_1} \de \Ttheta
      \quad\text{for some }\Ttheta\in \Lambda(\Phi_0,\Phi_1).\qedhere
\end{align*}
\end{proof}
\begin{corollary} \label{cor:charactplus} Let $\Phi_0, \Phi_1 \in \prob_2(\TX)$ and $\mu_1\in \prob_2(\X)$, then
  \begin{align}
    \label{eq:27}
\brap{\Phi}{\mu_1} & = \max \left \{ \int_{\TX \times \X}
                   \scalprod{x_0-x_1}{v_0} \de \ssigma \mid \ssigma \in
                   \Lambda(\Phi_0, \mu_1) \right \},\\ \notag
\brap{\Phi_0}{\Phi_1} & = \max \left \{ \int_{\TX \times \TX} \scalprod{x_0-x_1}{v_0-v_1} \de \Ttheta \mid \Ttheta \in \Lambda(\Phi_0, \Phi_1) \right \}.
\end{align}
\end{corollary}

\subsection{Right and left derivatives of the Wasserstein distance
  along a.c.~curves}
\label{subsec:left-right-derW} 
Let us now discuss the differentiability of the map $\interval \ni t
\mapsto \frac{1}{2}W_2^2(\mu(t), \nu)$ along a locally absolutely
continuous curve $\mu:
\interval \to \prob_2(\X)$, with $\interval$ an open interval of $\R$
and $\nu \in \prob_2(\X)$.
\begin{theorem} \label{thm:refdiff} Let $\mu:\interval \to
  \prob_2(\X)$ be a locally absolutely continuous curve and let
  $\vv:\interval\times \X\to \X$ and $A(\mu)$ be as in Theorem
  \ref{thm:tangentv}. Then, for every $\nu \in \prob_2(\X)$ and every
  $t \in A(\mu)$, it holds
     
    \begin{align}  \label{eq:refdiffa}
\lim_{h \downarrow 0} \frac{W_2^2(\mu_{t+h}, \nu)-W_2^2(\mu_t,
  \nu)}{2h}
      &= \bram{(\ii_\X , \vv_t)_{\sharp}\mu_t}{\nu},\\ \notag
\lim_{h \uparrow 0} \frac{W_2^2(\mu_{t+h}, \nu)-W_2^2(\mu_t, \nu)}{2h}
                                                                         &=
                                                                           \brap{(\ii_\X , \vv_t)_{\sharp}\mu_t}{\nu},
\end{align}
so that the map $s \mapsto W_2^2(\mu_s, \nu)$ is left and right differentiable at every $t \in A(\mu)$.
In particular,

\begin{enumerate}
\item if $t \in A(\mu)$ and $\nu \in \prob_2(\X)$ are s.t. there
  exists a unique optimal transport plan between $\mu_t$ and $\nu$,
  then the map $s \mapsto W_2^2(\mu_s, \nu)$ is
  differentiable at $t$;
\item there exists a subset $A({\mu,\nu})\subset A(\mu)$ of full
  Lebesgue measure such that $s\mapsto W_2^2(\mu_s,\nu)$ is
  differentiable in $A({\mu,\nu})$ and
  \begin{equation*}
    \begin{aligned}
      \frac 12\frac\d{\d t}W_2^2(\mu_t,\nu)
      &= \bram{(\ii_\X ,
        \vv_t)_{\sharp}\mu_t}{\nu}= \brap{(\ii_\X ,
        \vv_t)_{\sharp}\mu_t}{\nu}
      \\&=\int \la \vv_t(x_1),x_1-x_2\ra\,\d\mmu(x_1.x_2)
      \quad\text{for every }\mmu\in \Gamma_o(\mu_t,\nu),\ t\in A({\mu,\nu}).
    \end{aligned}
  \end{equation*}
\end{enumerate}

\end{theorem}
\begin{proof} Let $\nu\in\prob_2(\X)$ and for every $t\in \interval$ we set $\Phi_t := (\ii_\X , \vv_t)_{\sharp}\mu_t\in\prob_2(\TX)$. By Theorem \ref{thm:characterization}, we have
\begin{align*}
\lim_{h \downarrow 0} \frac{W_2^2(\exp_{\sharp}^h\Phi_t, \nu)-W_2^2(\mu_t, \nu)}{2h} &= \bram{(\ii_\X , \vv_t)_{\sharp}\mu_t}{\nu},\\
\lim_{h \uparrow 0} \frac{W_2^2(\exp_{\sharp}^h\Phi_t, \nu)-W_2^2(\mu_t, \nu)}{2h} &= \brap{(\ii_\X , \vv_t)_{\sharp}\mu_t}{\nu}.
\end{align*}
Since $\exp_{\sharp}^h \Phi_t = (\ii_\X +h \vv_t)_{\sharp}\mu_t$, then thanks to Theorem \ref{thm:tangentv} we have that the above limits coincide respectively with the limits in the statement, for all $t \in A(\mu)$.

Claim (1) comes by the characterizations given in Theorem
\ref{thm:characterization} and Corollary
\ref{cor:charactplus}. Indeed, if there exists a unique optimal
transport plan between $\mu_t$ and $\nu$, then $\bram{(\ii_\X ,
  \vv_t)_{\sharp}\mu_t}{\nu}=\brap{(\ii_\X ,
  \vv_t)_{\sharp}\mu_t}{\nu}$.

Claim (2) is a simple consequence of the fact that $s\mapsto
W_2^2(\mu_s,\nu)$ is differentiable a.e.~in $\interval$.     
\end{proof}

\begin{remark}
Thanks to \cite[Proposition 8.5.4]{ags}, in Theorem \ref{thm:refdiff} we can actually replace $\vv$ with any Borel velocity field $\ww$ solving the continuity equation for $\mu$ and s.t.  $\|\ww_t\|_{L^2_{\mu_t}} \in L^1_{loc}(\interval)$. Indeed, we notice that by \cite[Lemma 5.3.2]{ags},
\[\Lambda((\ii_\X , \vv_t)_{\sharp}\mu_t,\nu)=\{(\ii_\X , \vv_t,\ii_\X )_{\sharp}\ggamma\mid\ggamma\in\Gamma_o(\mu_t,\nu)\}.\]
\end{remark}
See Appendix \ref{app:B} for a further discussion about Theorem \ref{thm:refdiff}.

\begin{theorem} \label{thm:refdiff2} Let $\mu^1,\mu^2:\interval
  \to \prob_2(\X)$ be locally absolutely continuous curves and let
  $\vv^1,\vv^2: \interval \times \X\to \X$ be the corresponding
  Wasserstein velocity fields satisfying \eqref{eq:74} in
  $A({\mu^1})$ and $A({\mu^2})$ respectively.
  Then, for every $t \in A({\mu^1})\cap A({\mu^2})$, it holds
  \begin{align*}
\lim_{h \downarrow 0} \frac{W_2^2(\mu^1_{t+h},
  \mu^2_{t+h})-W_2^2(\mu^1_t, \mu^2_t)}{2h} &=
                                              \bram{(\ii_\X , \vv^1_t)_{\sharp}\mu^1_t}{(\ii_\X , \vv^2_t)_{\sharp}\mu^2_t},\\
\lim_{h \uparrow 0} \frac{W_2^2(\mu^1_{t+h},
  \mu^2_{t+h})-W_2^2(\mu^1_t, \mu^2_t)}{2h}
                                                                         &=
                                                                            \brap{(\ii_\X , \vv^1_t)_{\sharp}\mu^1_t}{(\ii_\X , \vv^2_t)_{\sharp}\mu^2_t}.
\end{align*}
In particular, there exists a subset $A\subset A({\mu^1})\cap A({\mu^2})$ of full Lebesgue
measure
such that $s \mapsto W_2^2(\mu^1_s, \mu^2_s)$ is
differentiable in $A$ and
\begin{equation}
  \label{eq:78}
  \begin{aligned}
    \frac12\frac{\d}{\d t}W_2^2(\mu^1_t,\mu^2_t)
    &= \bram{(\ii_\X ,
      \vv^1_t)_{\sharp}\mu^1_t}{(\ii_\X , \vv^2_t)_{\sharp}\mu^2_t}=
    \brap{(\ii_\X , \vv^1_t)_{\sharp}\mu^1_t}{(\ii_\X ,
      \vv^2_t)_{\sharp}\mu^2_t}
    \\&=
    \int \la \vv^1_t-\vv^2_t,x_1-x_2\ra\,\d\mmu(x_1,x_2)
    \quad\text{for every }\mmu\in \Gamma_o(\mu^1_t,\mu^2_t),\ t\in A.
  \end{aligned}
\end{equation}
\end{theorem}
The proof of Theorem \ref{thm:refdiff2} follows by the same argument
of the proof of Theorem \ref{thm:refdiff}.

\subsection{Convexity and semicontinuity of duality parings}
We want now to 
investigate the
semicontinuity and convexity properties of the functionals 
$\bram{\cdot}{\cdot}$ and $\brap{\cdot}{\cdot}$.
\begin{lemma} \label{lem:lsc}
  Let $(\Phi_n )_{n\in\N} \subset \prob_2(\TX)$ be converging to $\Phi$
  in $\prob_2^{sw}(\TX) $,
  and let
  $(\nu_n )_{n\in\N} \subset \prob_2(\X)$ be converging to $\nu$ in $
  \prob_2(\X) $.
  Then
  \begin{equation}
 \liminf_n \bram{\Phi_n}{\nu_n} \ge \bram{\Phi}{\nu}\quad\text{ and
 }\quad\limsup_n \brap{\Phi_n}{\nu_n} \le
 \brap{\Phi}{\nu}.\label{eq:62}
\end{equation}

  Finally, if $(\Phi_n^i)_{n\in \N}$, $i=0,1$, are sequences
  converging to $\Phi^i$ in $\prob_2^{sw}(\TX)$ then
  \begin{equation}\label{eq:152}
    \liminf_{n\to\infty}\bram{\Phi^0_n}{\Phi^1_n}\ge
    \bram{\Phi^0}{\Phi^1},\qquad
    \limsup_{n\to\infty}\brap{\Phi^0_n}{\Phi^1_n}\ge
    \brap{\Phi^0}{\Phi^1}.
  \end{equation}
\end{lemma}
\begin{proof}
  We just consider the proof of the first inequality \eqref{eq:62}; the
  other
  statements follow by similar arguments and by Remark
  \ref{rem:reduction}.
  
We can extract a subsequence of $(\Phi_n)_{n \in\N}$ (not relabeled)
s.t. the $\liminf$ is achieved as a limit. We have to prove that
\begin{equation}
 \lim_n \bram{\Phi_n}{\nu_n} \ge \bram{\Phi}{\nu}.\label{eq:23}
\end{equation}

For every $n\in\N$ take $\ssigma_n \in \Lambda_o(\Phi_n, \nu_n)$
with $\bar\ttheta_n=(\sfx^0,\sfx^1)_\sharp \ssigma_n$, and
observe that the family $(\bar\ttheta_n)_{n\in\N}$ is relatively compact in
$\prob_2(\X^2)$ (since the marginals of $\bar\ttheta_n$ are
converging w.r.t.~$W_2$) so that
$(\ssigma_n)_{n\in\N}$ is relatively compact in $\prob_2^{sws}(\TX\times
\X)$ by Proposition \ref{prop:finalmente}
since the moments $\int |v_0|^2\,\d\ssigma_n(x_0,v_0,x_1)=|\Phi_n|^2_2$ are uniformly bounded
by assumption.
Thus, possibly passing to a further subsequence, we have that
$(\ssigma_n)_{n\in\N}$ converges to some $\ssigma$ in
$\prob_2^{sws}(\TX\times
\X)$. In particular
$\ssigma \in \Lambda(\Phi, \nu)$ since optimality of the $\X^2$
marginals is preserved by narrow convergence.

\eqref{eq:43} then yields
\[  \lim_{n\to\infty} \bram{\Phi_n}{\nu_n} = \lim_{n\to\infty} \int \scalprod{v_0}{x_0-x_1} \de \ssigma_n = \int \scalprod{v_0}{x_0-x_1}\de\ssigma \]
which yields \eqref{eq:23} since the RHS is larger than
$\bram{\Phi}{\nu}$ by Theorem \ref{thm:characterization}.
\end{proof}

\begin{remark} Notice that in the special case in which $\Lambda(\Phi, \nu)$ is a singleton, then the limit exists and it holds
  \[ \lim_{n\to\infty} \bram{\Phi_n}{\nu_n} = \bram{\Phi}{\nu},\qquad \lim_{n\to\infty} \brap{\Phi_n}{\nu_n} = \brap{\Phi}{\nu}.\]
\end{remark}
\begin{lemma}
  \label{le:convexity-pairing}
  For every $\mu,\nu\in \prob_2(\X)$ the maps
  $\Phi\mapsto \bram\Phi\nu$
  and
  $(\Phi,\Psi)\mapsto \bram\Phi\Psi$
  (resp.~$\Phi\mapsto\brap\Phi\nu$
  and $(\Phi,\Psi)\mapsto\brap\Phi\Psi$)
  are convex (resp.~concave) in
  $\relcP 2{\mu}\TX$ and
  $\relcP 2{\mu}\TX\times \relcP 2{\nu}\TX$.
\end{lemma}
\begin{proof}
  We prove the convexity of
  $(\Phi,\Psi)\mapsto\bram\Phi\Psi$ in
  $\relcP 2{\mu}\TX\times \relcP 2{\nu}\TX$;
  the argument of the proofs of
  the other statements are completely analogous.
  
  Let $\Phi_k\in \relcP 2{\mu}\TX$,
  $\Psi_k\in \relcP 2{\nu}\TX$,
  and let $\beta_k\ge0$, 
  with $\sum_k\beta_k=1$, $k=1,\cdots,K$. 
  We set  $\Phi=\sum_{k=1}^K\beta_k\Phi_k$,
  $\Psi=\sum_{k=1}^K\beta_k\Psi_k$,
  For every $k$ let us select
  $\Ttheta_{k}\in \Lambda(\Phi_k,\Psi_k)$ such that
  \begin{displaymath}
    \bram{\Phi_k}{\Psi_k}=
    \int \langle v_1-v_0,x_1-x_0\rangle\,\d\Ttheta_{k}.
  \end{displaymath}
  It is not difficult to check that
  $\Ttheta:=\sum_{k}\beta_{k}\Ttheta_{k}\in
  \Lambda(\Phi,\Psi)$
  so that
  \begin{displaymath}
    \bram{\Phi}{\Psi}
    \le
    \int \langle v_1-v_0,x_1-x_0\rangle\,\d\Ttheta=
    \sum_{k}\beta_{k}\int \langle v_1-v_0,x_1-x_0\rangle\,
    \d\Ttheta_{k}=\sum_k \beta_{k}\bram{\Phi_k}{\Psi_k}.\qedhere
  \end{displaymath}
\end{proof}
\subsection{Behaviour of duality pairings along geodesics}
We have seen that
the duality pairings $\bram\cdot\cdot$ and $\brap \cdot\cdot$
may differ when the collection of optimal plans
$\Gamma_o(\mu_0,\mu_1)$ contains more than one element.
It is natural to expect a simpler behaviour along geodesics.
We will introduce the following
definition, where we use the notation
\begin{equation*}
  \sfx^t(x_0,x_1):=(1-t)x_0+tx_1,\quad
  \sfv^0(x_0,v_0,x_1):=v_0\quad
  \text{for every }(x_0,v_0,x_1)\in \TX\times \X,\,t\in[0,1].
\end{equation*}
\begin{definition}
  \label{def:last}
  For $\ttheta\in \prob_2(\X \times \X)$, $t\in[0,1]$,
  $\vartheta_t=\sfx^t_\sharp\ttheta$ and
  $\Phi\in \relcP2{\vartheta_t}{\TX}$,
  we set
  \begin{equation*}
    \Gamma_t(\Phi,\ttheta):=
    \left \{ \ssigma \in \prob_2(\TX\times \X)
    \mid
    (\sfx^0,\sfx^1)_{\sharp}\ssigma =\ttheta,\quad
    (\sfx^t \circ(\sfx^0,\sfx^1),\sfv^0)_\sharp \ssigma=\Phi\right \},
  \end{equation*}
  which is not empty since $\vartheta_t=\sfx^t_\sharp\ttheta=\sfx_\sharp\Phi$.
 We set
  \begin{align*}
    \ebrab{\Phi}{\ttheta}t
    & := \int
      \scalprod{x_0-x_1}{\bry{\Phi}(\sfx^t (x_0,x_1))} \de \ttheta(x_0,x_1),\\
    \directionalm{\Phi}{\ttheta}t
    & := \min \left \{ \int
      \scalprod{x_0-x_1}{v_0} \de \ssigma(x_0,v_0,x_1)
      \mid \ssigma \in \Gamma_t(\Phi,\ttheta)
      \right \}, \\
    \directionalp{\Phi}{\ttheta}t
    & := \max \left \{ \int
      \scalprod{x_0-x_1}{v_0} \de \ssigma(x_0,v_0,x_1)
      \mid \ssigma \in \Gamma_t(\Phi,\ttheta)
      \right \}.
      \intertext{If moreover $\Phi_0\in\prob_2(\TX|\vartheta_0)$,
      $\Phi_1\in \prob_2(\TX|\vartheta_1)$, $\ttheta\in\Gamma(\vartheta_0,\vartheta_1)$, we define}
    \directionalm{\Phi_0}{\Phi_1}\ttheta
    & :=
      \directionalm{\Phi_0}{\ttheta}0-
      \directionalp{\Phi_1}{\ttheta}1,\\
      \directionalp{\Phi_0}{\Phi_1}\ttheta
    &:=
      \directionalp{\Phi_0}{\ttheta}0-
      \directionalm{\Phi_1}{\ttheta}1.
  \end{align*}
\end{definition}
Notice that, if $\Phi_x$ is the disintegration of $\Phi$ with respect to
$\vartheta_t=\sfx_\sharp\Phi$, we can consider
the barycentric coupling $\ssigma_t:= \int_{\X\times\X}\Phi_{\sfx^t}\,\d \ttheta \in \Gamma_t(\Phi,\ttheta)$, i.e.
\begin{displaymath}
  \int \psi(x_0,v_0,x_1)\,\d\ssigma_t=
  \int\Big[\int \psi(x_0,v_0,x_1)\,\d\Phi_{(1-t)x_0+tx_1}(v_0)\Big]\,\d\ttheta(x_0,x_1)
\end{displaymath}
so that
$\ebrab\Phi\ttheta t=\int \langle{v_0},x_0-x_1\rangle\,\d\ssigma_t$ and
\begin{equation*}
  \directionalm\Phi\ttheta t\le \ebrab\Phi\ttheta t\le \directionalp\Phi\ttheta t.
\end{equation*}

If we define by $\mathsf s:\X^2\to\X^2$ the map $\mathsf s
(x_0,x_1):=(x_1,x_0)$
(with a similar definition  for $\TX\times \X$: $\mathsf
s(x_0,v_0,x_1):=(x_1,v_0,x_0)$) it is easy to check that
\begin{equation*}
  \ssigma\in \Gamma_t(\Phi,\ttheta)\quad\Leftrightarrow
  \quad
  \mathsf s_{\sharp}\ssigma\in \Gamma_{1-t}(\Phi,\mathsf s_\sharp\ttheta)
\end{equation*}
so that
\begin{equation}
  \label{eq:68}
  \directionalm{\Phi}{\ttheta}t=- \directionalp{\Phi}{\mathsf
    s_\sharp\ttheta}{1-t},\quad
  \directionalp{\Phi}{\ttheta}t=- \directionalm{\Phi}{\mathsf s_\sharp\ttheta}{1-t}.
\end{equation}
\eqref{eq:21} and \eqref{eq:27} have simpler versions in two
particular cases, which will be explained in the next remark.
\begin{remark}[Particular cases]
  \label{rem:particular}
  Suppose that $\ttheta\in \prob_2(\X^2)$, $t\in[0,1]$, $\vartheta_t=\sfx^t_\sharp\ttheta$,
  $\Phi\in \relcP2{\vartheta_t}\TX$ and $\sfx^t:\X^2\to\X$ is
  $\ttheta$-essentially injective so that
  $\ttheta$ is concentrated on a Borel map $(X_0,X_1):\X\to
  \X\times \X$,
  i.e.~$\ttheta=(X_0,X_1)_\sharp\vartheta_t$.
  In this case $\Gamma_t(\Phi,\ttheta)$ 
  contains a unique element
  given by
  $(X_0\circ\sfx,\sfv,X_1\circ\sfx)_\sharp\Phi$
  and 
  \begin{equation}
    \label{eq:29}
    \directionalm{\Phi}{\ttheta}t=
    \directionalp{\Phi}{\ttheta}t=
    \ebrab{\Phi}{\ttheta}t=
    \int \la v,X_0(x)-X_1(x)\ra\,\d\Phi(x,v)=
    \int \la \bry\Phi,X_0-X_1\ra\,\d\vartheta_t,
  \end{equation}
  where in the last formula we have applied 
  the barycentric reduction \eqref{eq:30}.
  When $t=0$ and $\ttheta$ is the unique element of
  $\Gamma_o(\vartheta_0,\vartheta_1)$ then
  $X_0(x)=x$ and 
  we obtain
    \begin{equation*}
    \bram{\Phi}{\vartheta_1}=
    \brap{\Phi}{\vartheta_1}=\directionalm{\Phi}{\ttheta}0=\directionalp{\Phi}{\ttheta}0=\int \la v,x-X_1(x)\ra\,\d\Phi(x,v)=
    \int \la \bry\Phi,x-X_1(x)\ra\,\d\vartheta_0(x).
  \end{equation*}
  Another simple case is when $\Phi=\Phi_{\vartheta_t,\ww}$ for some
  vector field $\ww\in L^2_{\vartheta_t}(\X;\X)$ as in \eqref{eq:24}
  (i.e.~its disintegration $\Phi_x$ w.r.t.~$\vartheta_t$ takes the form
  $\delta_{\ww(x)}$ and $\ww=\bry\Phi.$).
  We have
  \begin{equation*}
    \directionalm{\Phi}{\ttheta}t=
    \directionalp{\Phi}{\ttheta}t=
    \int \la \ww((1-t)x_0+tx_1),x_0-x_1\ra\,\d\ttheta(x_0,x_1).
  \end{equation*}
  In particular we get
    \begin{equation*}
    \bram{\Phi}{\vartheta_1}=
    \min\Big\{\int \la \ww(x),x_0-x_1\ra\,\d\ttheta(x_0,x_1)\mid
    \ttheta\in \Gamma_o(\vartheta_0,\vartheta_1)\Big\}.
  \end{equation*}
\end{remark}
An important case in which
the previous Remark \ref{rem:particular} applies
is that of geodesics in $\prob_2(\X)$.
\begin{lemma} \label{lem:starting}
  Let $\mu_0, \mu_1 \in \prob_2(\X)$, $(\mu_t)_{t \in [0,1]}$ be a
  constant speed geodesic induced by an optimal plan $\mmu \in
  \Gamma_o(\mu_0, \mu_1)$ by the relation
  \begin{equation*}
 \mu_t = \sfx^t_{\sharp} \mmu,\quad t\in[0,1],\quad
\sfx^t(x_0,x_1)=(1-t)x_0+tx_1.
\end{equation*}
If $t \in (0,1)$, $\Phi_t\in\prob_2(\TX|\mu_t)$, $\hat\mmu=\mathsf s_\sharp
\mmu\in \Gamma_o(\mu_1,\mu_0)$, then 
\begin{equation}
  \begin{aligned}
    &\frac{1}{1-t} \bram{\Phi_t}{\mu_1} &=& \frac{1}{1-t}
    \brap{\Phi_t}{\mu_1} &=&
    \directionalm{\Phi_t}{\mmu}t&=&\directionalp{\Phi_t}{\mmu}t\\
    =& -\frac{1}{t} \bram{\Phi_t}{\mu_0} &=&
    -\frac{1}{t} \brap{\Phi_t}{\mu_0}&=&
   - \directionalm{\Phi_t}{\hat \mmu}{1-t}&=&-\directionalp{\Phi_t}{\hat\mmu}{1-t}.
  \end{aligned}
\label{eq:64}
\end{equation}
\end{lemma}
\begin{proof} The crucial fact is that $\sfx^t : \X^2 \to \X$
  is
  injective on $\supp(\mmu)$ and thus a bijection on its image $\supp(\mu_t)$. Indeed, take $(x_0, x_1), (x_0', x_1') \in \supp(\mmu)$, then
\begin{align*}
 \left | \sfx^t(x_0, x_1) - \sfx^t(x_0', x_1') \right |^2 &= (1-t)^2|x_0-x_0'|^2 + t^2|x_1-x_1'|^2 + 2t(1-t)\scalprod{x_0-x_0'}{x_1-x_1'} \\
 &\ge  (1-t)^2|x_0-x_0'|^2 + t^2|x_1-x_1'|^2
 \end{align*}
 thanks to the cyclical monotonicity of $\supp(\mmu)$
 (see \cite[Remark 7.1.2]{ags}).

 Then, for every $x \in \supp(\mu_t)$, there exists a unique couple
 $(x_0, x_1)
 =(X_0(x),X_1(x))\in \supp(\mmu)$ s.t. $x=(1-t)x_0 +
 tx_1$, where we refer to Remark \ref{rem:particular} for the definitions of $X_0, X_1$
 (cf. also \cite[Theorem 5.29]{santambrogio}).
 Hence, in the following diagram all maps are bijections:
 \begin{figure}[h]
  \centering
  \begin{tikzcd}[row sep=large, column sep = 5em]
\supp(\mmu_{t0}) &\arrow[l, "(\sfx^t \text{,}\, \sfx^0)", line width=0.3mm, swap]\supp(\mmu) \arrow[r, "(\sfx^t  \text{,}\, \sfx^1)", line width=0.3mm] \arrow[d, "\sfx^t", line width=0.3mm]&\supp(\mmu_{t1}) \\
\quad & \supp(\mu_t) \arrow[ur, "(\ii_\X \text{,}\, X_1)", line width=0.3mm, swap] \arrow[ul, "(\ii_\X  \text{,}\, X_0)", line width=0.3mm] & \quad 
\end{tikzcd}
\end{figure}
\quad \\
where $ \mmu_{t1}  =(\sfx^t,\sfx^1)_\sharp\mmu
=(\ii_\X, X_1)_{\sharp}\mu_t$ is the unique
element of $ \Gamma_o(\mu_t, \mu_1)$ and
$ \mmu_{t0}=(\sfx^t,\sfx^0)_\sharp\mmu=(\ii_\X , X_0)_{\sharp}\mu_t=
(\sfx^{1-t},\sfx^1)_\sharp\hat\mmu$ is the unique element of
$\Gamma_o(\mu_t, \mu_0)$
(see Theorem \ref{theo:chargeo}).
Since 
\[ \frac{x-X_1(x)}{1-t} = \frac{x-x_1}{1-t}= x_0-x_1 = -\frac{x-x_0}{t} = -\frac{x-X_0(x)}{t},\]
and $\Lambda(\Phi_t, \mu_1) = \{ (\ii_{\TX} , X_1
\circ \sfx)_{\sharp} \Phi_t \}$
thanks to Theorem \ref{theo:chargeo}, by Theorem
\ref{thm:characterization} and Corollary \ref{cor:charactplus} we have
\begin{equation*}
\bram{\Phi_t}{\mu_1} = \brap{\Phi_t}{\mu_1}  = \int_{\TX} \scalprod{v}{x-X_1(x)} \de \Phi_t(x,v).
\end{equation*}
Analogously, $\Lambda(\Phi_t, \mu_0) = \{ (\ii_{\TX} , X_0 \circ \sfx)_{\sharp} \Phi_t \}$. Hence
\begin{equation*}
\bram{\Phi_t}{\mu_0} = \brap{\Phi_t}{\mu_0} = \int_{\TX} \scalprod{v}{x-X_0(x)} \de \Phi_t(x,v).
\end{equation*}
Also recalling \eqref{eq:68} and \eqref{eq:29} we conclude.
\end{proof}

\section{Dissipative probability vector fields: the metric viewpoint}
\label{sec:dissipative}

\subsection{Multivalued Probability Vector Fields and \texorpdfstring{$\lambda$}{l}-dissipativity}
\begin{definition}[Multivalued Probability Vector Field - MPVF] \label{def:MPVF}
  A \emph{multivalued probability vector field} $\frF$ is a nonempty subset of
  $\prob_2(\TX)$ with domain $\dom(\frF) := \sfx_\sharp(\frF)=
  \{ \sfx_\sharp\Phi:\Phi\in \frF \}$.
  Given $\mu \in \prob_2(\X)$, we define the \emph{section} $\frF[\mu]$
  of $\frF$ as 
  \[ \frF[\mu] := (\sfx_\sharp)^{-1}(\mu)\cap \frF= \left \{ \Phi \in \frF
      \mid \sfx_{\sharp}\Phi = \mu \right \}. \]
A \emph{selection} $\frF'$ of $\frF$ is a subset of $\frF$ such that $\dom(\frF')=\dom(\frF)$.
  We call $\frF$ a \emph{probability vector field} (PVF) if $\sfx_\sharp $ is injective in $\frF$, i.e.~$\frF[\mu]$ contains a
  unique element for every $\mu\in \dom(\frF)$.
  $\frF$ is a vector field if for 
  every $\mu\in \dom(\frF)$
  $\frF[\mu]$ contains
  a unique element $\Phi$ concentrated on a map,
  i.e.~$\Phi=(\ii_\X, \bry\Phi)_\sharp\mu$.
\end{definition}
\begin{remark}\label{Vmultifunc}
We can equivalently formulate Definition
\ref{def:MPVF}
by considering $\frF$ as a multifunction,
as in the case, e.g., of the Wasserstein subdifferential
$\boldsymbol\partial\mathcal F$ of a function $\mathcal F:\prob_2(\X)\to(-\infty,+\infty]$, see
\cite[Ch.~10]{ags} and the next Section \ref{sec:Exsubdiff}.
According to this viewpoint, a \MPVF is a set-valued map $\frF:\prob_2(\X)\supset \dom(\frF) \rightrightarrows\prob_2(\TX)$ such that $\sfx_{\sharp}\Phi=\mu$ for all $\Phi\in \frF[\mu]$. 
In this way, each section $\frF[\mu]$ is nothing but the image of $\mu\in\dom(\frF)$ through $\frF$.
In this case, \emph{probability vector fields} correspond to single
valued maps: this notion has been used in
\cite{Piccoli_2019}
with the aim of describing a sort of velocity field on $\prob(\X)$, and later in \cite{Piccoli_MDI} dealing with Multivalued Probability Vector Fields (called Probability Multifunctions).
\end{remark}

\begin{definition}[Metrically $\lambda$-dissipative \MPVF] \label{def:dissipative}
  A \MPVF $\frF \subset \prob_2(\TX)$ is (metrically) \emph{$\lambda$-dissipative},
  $\lambda\in \R$, if
    \begin{equation}
      \bram{\Phi_0}{\Phi_1} \le \lambda W_2^2(\mu_0,\mu_1)
      \quad\text{for every }\Phi_0,\Phi_1\in \frF,\ \mu_i=\sfx_\sharp \Phi_i.
  \label{eq:33}
\end{equation}
We say that $\frF$ is (metrically) \emph{$\lambda$-accretive}, if $-\frF=\{-\Phi:\Phi\in \frF\}$
(recall \eqref{eq:20}) is $-\lambda$-dissipative, i.e.
  \begin{displaymath}
    \brap{\Phi_0}{\Phi_1} \ge \lambda W_2^2(\mu_0,\mu_1)
      \quad\text{for every }\Phi_0,\Phi_1\in \frF,\ \mu_i=\sfx_\sharp
      \Phi_i.
  \end{displaymath}
\end{definition}
\begin{remark}
  Notice that \eqref{eq:33} is equivalent to ask for the existence of
  a coupling $\Ttheta\in \Lambda(\Phi_0,\Phi_1)$
  (thus $(\sfx^0,\sfx^1)_\sharp\Ttheta$ is optimal
  between $\mu_0=\sfx_\sharp \Phi_0$ and $\mu_1=\sfx_\sharp\Phi_1$)
  such that
  \begin{equation*}
    \int \la v_1-v_0,x_1-x_0\ra\,\d\Ttheta\le
    \lambda W_2^2(\mu_0,\mu_1)=
    \lambda\int |x_1-x_0|^2\,\d\Ttheta.
  \end{equation*}
  Recalling the discussion of the previous section,
  $\lambda$-dissipativity has a natural metric interpretation:
  for every $\Phi_0,\Phi_1\in \frF$ with $\mu_0=\sfx_\sharp\Phi_0$,
  $\mu_1=\sfx_\sharp\Phi_1$
  we have
  the asymptotic expansion
\begin{equation*}
  W_2^2(\exp^t\Phi_0,\exp^t\Phi_1)\le (1+2\lambda
  t)W_2^2(\mu_0,\mu_1)+o(t)\quad\text{as }t\downarrow0.
\end{equation*}
\end{remark}
\begin{remark}\label{rmk:equivdiss}
  Thanks to  Corollary
  \ref{lem:control}, \eqref{eq:33} implies the weaker condition
  \begin{equation}\label{Hdiss}
    \bram{\Phi_0}{\mu_1}+\bram{\Phi_1}{\mu_0} \le \lambda W_2^2(\mu_0,
    \mu_1),\quad \forall\, \Phi_0,\Phi_1 \in \frF,\ \mu_0=\sfx_\sharp \Phi_0,\ 
    \mu_1=\sfx_\sharp \Phi_1.
  \end{equation}
  It is clear that the inequality of \eqref{Hdiss} implies
  the inequality of \eqref{eq:33} whenever $\Gamma_o(\mu_0,\mu_1)$
  contains only one element. More generally, 
  we will
  see in Corollary \ref{cor:dissipativity} that \eqref{Hdiss} is in fact equivalent to \eqref{eq:33}
  when $\dom(\frF)$ is geodesically
convex (according to Definition \ref{def:W2geodesic}).
\end{remark}

As in the standard Hilbert case, $\lambda$-dissipativity can be
reduced to dissipativity (meaning $0$-dissipativity) by a simple transformation. Let us introduce
the map
\begin{equation*}
  L^\lambda:\TX\to\TX,\quad
  L^\lambda(x,v):=(x,v-\lambda x),
\end{equation*}
observing that for every $\ssigma\in \prob_2(\TX\times \X)$
with $(\sfx^i)_\sharp\ssigma=\mu_i$, $i=0,1$, the
transformed plan $\ssigma^\lambda:=(L^\lambda,\ii_\X)_\sharp
\ssigma$ satisfies
\begin{align}
    \notag
    \int \la v_0,x_0-x_1\ra\,\d\ssigma^\lambda
    &=
      \int \la v_0-\lambda x_0,x_0-x_1\ra\,\d\ssigma
    \\\label{eq:121}
    &=
    \int \la v_0,x_0-x_1\ra\,\d\ssigma
    -\frac\lambda2\int |x_0-x_1|^2\,\d\ssigma 
    +\frac\lambda2\Big(\sqm{\mu_1}-\sqm{\mu_0}\Big).
  \end{align}

  Similarly, if $\Theta\in \prob_2(\TX\times \TX)$
  with $\sfx^i_\sharp\Theta=\mu_i$,
  the plan $\Theta^\lambda:=(L^\lambda,L^\lambda)_\sharp\Theta$ 
  satisfies
  \begin{align}
    \notag
    \int \la v_0-v_1,x_0-x_1\ra\,\d\Theta^\lambda
    &=
      \int \la v_0-v_1-\lambda (x_0-x_1),x_0-x_1\ra\,\d\Theta
    \\\label{eq:121bis}
    &=
      \int \la v_0-v_1,x_0-x_1\ra\,\d\Theta
      -\lambda\int |x_0-x_1|^2\,\d\Theta.
  \end{align}
\begin{lemma}
  \label{le:trick}
  $\frF$ is a $\lambda$-dissipative {\em MPVF}
  (resp.~satisfies \eqref{Hdiss}) if and only if
  $\frF^\lambda:=L^\lambda_\sharp(\frF)=
  \{L^\lambda_\sharp\Phi\mid\Phi\in \frF\}$
  is dissipative (resp.~satisfies \eqref{Hdiss} with $\lambda=0$).
\end{lemma}
\begin{proof}
  Let us first check the case of \eqref{Hdiss}.
  Since $\ssigma\in \Lambda_o(\Phi_0,\mu_1)$ if and only
  if
  $\ssigma^\lambda\in \Lambda_o(L^\lambda_\sharp\Phi_0,\mu_1)$,
  \eqref{eq:121} yields 
  \begin{align*}
    \int \la v_0,x_0-x_1\ra\,\d\ssigma^\lambda=
    \int \la v_0,x_0-x_1\ra\,\d\ssigma
    -\frac\lambda2\Big(\sqm{\mu_0}-\sqm{\mu_1}+W_2^2(\mu_0,\mu_1)\Big)
  \end{align*}
  and therefore
  \begin{equation}\label{eq:rellam}
    \bram{L^\lambda_\sharp \Phi_0}{\mu_1}=
    \bram{\Phi_0}{\mu_1}-\frac\lambda2\Big(\sqm{\mu_0}-\sqm{\mu_1}+W_2^2(\mu_0,\mu_1)\Big).
  \end{equation}
  Using the corresponding identity for $    \bram{L^\lambda_\sharp
    \Phi_1}{\mu_0}$ we obtain that $\frF^\lambda$ is dissipative.

  A similar argument, using the identity \eqref{eq:121bis}, shows
  the equivalence between the $\lambda$-dissipativity of $\frF$ and
  the dissipativity of $\frF^\lambda$.
\end{proof}

Let us conclude this section by showing that $\lambda$-dissipativity can be deduced from a Lipschitz like condition similar to the one considered in \cite{Piccoli_2019} (see Appendix \ref{sec:cfrPic}).

\begin{lemma}
  \label{le:pic1}
  Suppose that the \MPVF $\frF$ satisfies 
  \[\mathcal{W}_2(\frF[\nu], \frF[\nu']) \leq L W_2(\nu,
    \nu'),\quad\forall \, \nu, \nu' \in \dom(\frF),\] where
  $\mathcal{W}_2:\prob_2(\TX)\times\prob_2(\TX)\to[0,+\infty)$ is
  defined by
  \begin{equation*}
    \mathcal{W}_2^2(\Phi_0,\Phi_1) = \inf\left\{ \int_{\TX \times \TX} |v_0 - v_1|^2 \de\Ttheta(x_0,v_0,x_1,v_1) : \Ttheta \in\Lambda(\Phi_0,\Phi_1) \right\},
  \end{equation*}
  with $\Lambda(\cdot,\cdot)$ as in Definition \ref{def:lambda}.
  Then $\frF$ is $\lambda$-dissipative, for $\lambda:=\frac 12(1+L^2)$
\end{lemma}
\begin{proof}
Let $\nu',\nu''\in\dom(\frF) $, then by Theorem \ref{thm:characterization} and Young's inequality, we have
\begin{equation*}
\begin{split}
\bram{\frF[\nu']}{\frF[\nu'']}
&=\min\left\{\int_{\TX\times \TX}\scalprod{x'-x''}{v'-v''}\de\Ttheta\,:\,\Ttheta\in\Lambda(\frF[\nu'],\frF[\nu''])\right\}\\
&\le\frac{1}{2}\left(W_2^2(\nu',\nu'')+\mathcal{W}_2^2(\frF[\nu'],\frF[\nu''])\right)
\\&
\le\frac{L^2+1}{2}\,W_2^2(\nu',\nu''). \qedhere
\end{split}
\end{equation*}
\end{proof}

\subsection{Behaviour of \texorpdfstring{$\lambda$}{l}-dissipative \MPVF along geodesics}
\label{subsec:dissipative-geodesic}
Let us now study the behaviour of a MPVF $\frF$ along geodesics.
Recall that in the case of a dissipative map $\fF:\mathsf H\to \mathsf H$ in a Hilbert
space $\mathsf H$, it is quite immediate to prove that the real function
\begin{equation*}
f(t) := \scalprod{F(x_t)}{x_0-x_1},\quad
  x_t = (1-t)x_0 + tx_1,
  \quad t \in [0,1]
\end{equation*}
is monotone increasing.
This property has a natural counterpart
in the case of measures.
\begin{definition}\label{def:plangeodomV}
  Let $\frF \subset \prob_2(\TX)$, $\mu_0,\mu_1\in\overline{\dom(\frF)}$, $\mmu\in \Gamma_o(\mu_0,\mu_1)$. We define the sets
  \begin{align} \notag
    \rI\mmu\frF:={}&\Big\{t\in [0,1]:\sfx^t_\sharp\mmu\in \dom(\frF)\Big\},\\
    \label{eq:18}
    \CondGammao\frF{\mu_0}{\mu_1}{i}:={}&
                                          \Big\{\mmu\in \Gamma_o(\mu_0,\mu_1):
                                          i\text{ is an accumulation point of }\rI\mmu\frF
                                          \Big\}, i=0,1\\ \notag
    \CondGammao\frF{\mu_0}{\mu_1}{01}:={}&
                                           \CondGammao\frF{\mu_0}{\mu_1}{0}\cap \CondGammao\frF{\mu_0}{\mu_1}{1}.
  \end{align}

Notice that these sets depend on $\frF$ just through $\dom(\frF)$. In
particular, if $\mu_0,\mu_1\in\dom(\frF)$ and $\dom(\frF)$ is open or
geodesically convex according to
Definition \ref{def:W2geodesic} then
$\CondGammao\frF{\mu_0}{\mu_1}{01}\neq\emptyset$.
\end{definition}

\begin{definition}
  \label{def:frfl}
  Let $\frF \subset\prob_2(\TX)$ be a
  \MPVF.
  Let $\mu_0, \mu_1 \in
   \overline{\dom(\frF)}$, $\mmu \in
  \Gamma_o(\mu_0,\mu_1)
$
  and let $\mu_t:=\sfx^t_{\sharp} \mmu$, $t\in [0,1]$.
  For every $t\in \rI\mmu\frF$ we define
\begin{align*}
  \directionalm \frF\mmu t := \sup \left \{
                          \directionalm{\Phi}{\mmu}t \mid \Phi \in \frF[\mu_t]
                          \right \},
                                    \qquad
  \directionalp \frF\mmu t := \inf \left \{
                          \directionalp{\Phi}{\mmu}t 
                          \mid \Phi \in \frF[\mu_t] \right \}.
\end{align*}
\end{definition}

\begin{theorem} \label{theo:propflfr}
   Let us suppose that the \MPVF $\frF$
  satisfies \eqref{Hdiss},
  let $\mu_0,\mu_1\in\overline{\dom(\frF)}$, and let
  $\mmu\in\Gamma_o(\mu_0,\mu_1)
  $ with
  $W^2:=W_2^2(\mu_0,\mu_1)$. Then the following properties hold
\begin{enumerate}
\item $\directionalp \frF\mmu t \le \directionalm \frF\mmu t$ for every
  $t \in (0,1)\cap
  \rI\mmu\frF$;
\item $\directionalm \frF\mmu s \le \directionalp \frF\mmu t+\lambda W^2 (t-s)$
  for every $s,t\in \rI\mmu\frF$, $s < t $;
\item $t\mapsto \directionalm \frF\mmu t+\lambda W^2 t$ and
  $t\mapsto \directionalp \frF\mmu t+\lambda W^2 t $ are increasing respectively in $\rI\mmu\frF\setminus \{1\}$
  and in $\rI\mmu\frF\setminus \{0\}$.
\item the right (resp.~left)
  limits
  of $t\mapsto \directionalm \frF\mmu
  t$ and $
  t\mapsto \directionalp \frF\mmu t$ exist at every right (resp.~left) accumulation point of
  $\rI\mmu\frF$,
  and in those points the right (resp.~left) limits of
  $\directionalm \frF\mmu t$ 
  coincide with the right (resp.~left) limits of $\directionalp \frF\mmu t$.
\item[(5)]
  $\directionalp \frF\mmu t = \directionalm \frF\mmu t$ at every
  interior point 
  $t $ of $\rI\mmu\frF$
  where one of them is continuous.
\end{enumerate}

\end{theorem}
\begin{proof} Throughout all the proof we set $f_r(t):=\directionalm
  \frF\mmu t$ and $f_l(t):=\directionalp \frF\mmu t$.
  Thanks to Lemma \ref{le:trick} and in particular to \eqref{eq:rellam}, it is easy to check that
  it is sufficient to consider the dissipative case $\lambda=0$.
\begin{enumerate}
\item It is a direct consequence of Lemma \ref{lem:starting} and the definitions of $f_r$ and $f_l$.
\item We prove that for every $\Phi \in \frF[\mu_s]$ and $\Phi' \in \frF[\mu_t]$ it holds
 \begin{equation}\label{eq:startp}  \directionalm \Phi\mmu s
  \le
 \directionalp {\Phi'}\mmu t.
\end{equation}
The thesis will follow immediately passing to the $\sup$ over $\Phi \in \frF[\mu_s]$ in the LHS and to the $\inf$ over $\Phi' \in \frF[\mu_t]$ in the RHS. It is enough to prove \eqref{eq:startp} in case at least one between $s,t$ belongs to $(0,1)$. Let us define the map $L: \prob_2(\TX \times \X) \to \R$ as
\[ L(\gamma):= \int_{\TX \times \X} \scalprod{v_0}{x_0-x_1}\de \gamma(x_0,v_0,x_1) \quad \gamma \in \prob_2(\TX \times \X).\]
Observe that, since it never happens that $s=0$ and $t=1$ at the same time, the map $T:\Gamma_s(\Phi, \mmu)\to \Lambda(\Phi,\mu_t)$ defined as
\[ T(\ssigma) := (\sfx^s \circ (\sfx^0,\sfx^1),\sfv^0, \sfx^t \circ (\sfx^0,\sfx^1))_{\sharp} \ssigma\]
is a bijection s.t.~$(t-s)L(\ssigma)=L(T(\ssigma))$ for every $\ssigma \in \Gamma_s(\Phi, \mmu)$. This immediately gives that
\[ (t-s) \directionalm \Phi \mmu s = \bram{\Phi}{\mu_t}.\]
In the same way we can deduce that
\[ (s-t) \directionalp {\Phi'} \mmu t = \bram{\Phi'}{\mu_s}. \]
Thanks to the dissipativity of $\frF$ we get
\[ (t-s)\directionalm \Phi \mmu s - (t-s)\directionalp {\Phi'} \mmu t = \bram{\Phi}{\mu_t} +\bram{\Phi'}{\mu_s} \le 0.\]

\item Combining (1) and (2) we have that for every $s,t\in \rI\mmu\frF$ with $0< s<t <1$ it holds
\begin{equation} \label{eq:monotonicity} 
f_l(s) \le f_r(s) \le f_l(t) \le f_r(t). 
\end{equation}
This implies that both $f_l$ and $f_r$ are increasing in $\rI\mmu\frF\cap (0,1)$. Observe that, again combining (1) and (2), it also holds
\begin{align*}
  f_r(0) &\le f_l(t) \le f_r(t),\\
f_l(t) &\le f_r(t) \le f_l(1)
\end{align*}
for every $t \in \rI\mmu\frF\setminus \{0,1\}$, and then $f_r$ is increasing in
$\rI\mmu\frF\setminus\{1\}$ and $f_l$ is increasing in $\rI\mmu\frF\setminus \{0\}$.
\item
  It is an immediate consequence of \eqref{eq:monotonicity}.
\item It is a straightforward consequence of (4).
  \qedhere
\end{enumerate}
\end{proof}
Thanks to the previous Theorem
\ref{theo:propflfr} the next definition is well posed.
\begin{definition}
  \label{def:directiona}
   Let us suppose that the \MPVF $\frF$
  satisfies \eqref{Hdiss},
  let $\mu_0,\mu_1\in \overline{\dom(\frF)}$.
  \begin{align*}
 \text{If 
    $\mmu\in \CondGammao\frF{\mu_0}{\mu_1}{0}$ we set}
    \quad \directional \frF\mmu {0+}
    &:= 
      \lim_{t\downarrow0}\directionalm \frF\mmu t
      =\lim_{t\downarrow0}\directionalp \frF\mmu t
    \\
      \text{If 
    $\mmu\in \CondGammao\frF{\mu_0}{\mu_1}{1}$ we set}
    \quad \directional \frF\mmu {1-}
    &:=
      \lim_{t\uparrow1}\directionalm \frF\mmu t=
      \lim_{t\uparrow1}\directionalp \frF\mmu t.
  \end{align*}
\end{definition}
\begin{corollary}
  Let us keep the same notation of Theorem
  {\em \ref{theo:propflfr}}
  and let $s\in \rI\mmu\frF\cap (0,1)$
  with $\Phi\in \frF[\mu_s]$.
  \begin{enumerate}
  \item
    If $\mmu\in \CondGammao\frF{\mu_0}{\mu_1}0$,
    we have that
    \begin{align}
    \label{eq:91}
      \directional \frF\mmu {0+}
    &\le 
      \directionalp\Phi\mmu s+\lambda s W^2= \directionalm\Phi\mmu s+\lambda s W^2;
    \end{align}
      if moreover
      $\Phi_0\in \frF[\mu_0]$ then
      \begin{equation}
        \label{eq:54}
        \bram{\Phi_0}{\mu_1}\le  \directionalm{\Phi_0}{\mmu}{0}\le    \directional \frF\mmu{0+}.
  \end{equation}
    \item
        If $\mmu\in \CondGammao\frF{\mu_0}{\mu_1}1$,
    we have that
    \begin{align*}
      \directionalp\Phi\mmu s-\lambda (1-s) W^2
      =
      \directionalm\Phi\mmu s-\lambda (1-s) W^2
      \le
      \directional \frF\mmu {1-};
    \end{align*}
    if moreover $\Phi_1\in \frF[\mu_1]$ then
    \begin{equation}
      \label{eq:54bis}
      \directional \frF\mmu{1-}\le \directionalp{\Phi_1}{\mmu}1\le
    -\bram{\Phi_1}{\mu_0}
    \end{equation}
  \item
    In particular, for every $\Phi_0\in \frF[\mu_0]$, $\Phi_1\in
  \frF[\mu_1]$ and $\mmu\in \CondGammao\frF{\mu_0}{\mu_1}{01}$
  we obtain
  \begin{equation}
    \label{eq:92}
    \directionalm{\Phi_0}{\Phi_1}\mmu\le
    \directional \frF\mmu{0+}-
    \directional \frF\mmu{1-}\le \lambda W^2_2(\mu_0,\mu_1).
  \end{equation}
\end{enumerate}

\end{corollary}
\eqref{eq:92} immediately yields the following property.
\begin{corollary}
  \label{cor:dissipativity}
  Suppose that a \MPVF $\frF$ satisfies
  \begin{equation}
    \label{eq:19}
    \text{for every $\mu_0,\mu_1\in \dom(\frF)$
      the set $\CondGammao\frF{\mu_0}{\mu_1}{01}$ of \eqref{eq:18} is not empty}
  \end{equation}
  (e.g. if $\dom(\frF)$ is open or geodesically convex),
  then $\frF$ is $\lambda$-dissipative if and only if it satisfies
  \eqref{Hdiss}.
\end{corollary}

\begin{proposition}\label{prop:scalmscal}
  Let $\frF\subset\prob_2(\TX)$ be a \MPVF satisfying \eqref{Hdiss},
  let $\mu_0\in\overline{\dom(\frF)}$
  and
  let $\Phi\in \relcP2{\mu_0}{\TX}$.
Consider the following statements
\begin{enumerate}[label=$(\mathrm{P\arabic*})$]
\item \label{eq:35bis}
    $\bram\Phi{\mu}+\bram\Psi{\mu_0}\le \lambda W_2^2(\mu_0,\mu)$
    for every $\Psi\in \frF$ with $\mu=\sfx_\sharp\Psi$;
\item\label{eq:131}
    for every $\mu \in \dom(\frF)$ there exists $\Psi\in
    \frF[\mu]$ s.t. $\bram\Phi{\mu}+\bram\Psi{\mu_0}\le \lambda W_2^2(\mu_0,\mu)$;
\item\label{eq:53bis}
    $\directionalm{\Phi}{\mmu}0 \le \directional\frF{\mmu}{0+}$ for every
    $\mu_1\in \overline{\dom(\frF)}$, $\mmu\in \CondGammao\frF{\mu_0}{\mu_1}0$;
\item\label{eq:53}
    $\directionalm{\Phi}{\mmu}0 \le \directional\frF{\mmu}{0+}$ for every
    $\mu_1\in \dom(\frF)$, $\mmu\in \CondGammao\frF{\mu_0}{\mu_1}0$;
\item\label{eq:132bis} $\directionalm{\Phi}{\mmu}0 \le \lambda W_2^2(\mu_0,\mu_1)+
    \directional\frF{\mmu}{1-}$
      for every
    $\mu_1\in \overline{\dom(\frF)}$, $\mmu\in \CondGammao\frF{\mu_0}{\mu_1}1$;
\item\label{eq:132} $\directionalm{\Phi}{\mmu}0 \le \lambda W_2^2(\mu_0,\mu_1)+
    \directional\frF{\mmu}{1-}$
      for every
    $\mu_1\in \dom(\frF)$, $\mmu\in \CondGammao\frF{\mu_0}{\mu_1}1$.
\end{enumerate}
Then the following hold
\begin{enumerate}
\item \ref{eq:35bis} $\Rightarrow$ \ref{eq:131} $\Rightarrow$ \ref{eq:53bis} $\Rightarrow$ \ref{eq:53};
\item \ref{eq:35bis} $\Rightarrow$ \ref{eq:131} $\Rightarrow$ \ref{eq:132bis} $\Rightarrow$ \ref{eq:132};
\item\label{item:ex4.171} if for every $\mu_1\in \dom(\frF)$
  $\CondGammao\frF{\mu_0}{\mu_1}0\neq \emptyset$, then \ref{eq:53} $\Rightarrow$ \ref{eq:35bis} (in particular, \ref{eq:35bis}, \ref{eq:131}, \ref{eq:53bis}, \ref{eq:53} are equivalent);
\item\label{item:ex4.172} if
    for every $\mu_1\in \dom(\frF)$
    $\CondGammao\frF{\mu_0}{\mu_1}1\neq \emptyset$, then \ref{eq:132} $\Rightarrow$ \ref{eq:35bis} (in particular, \ref{eq:35bis}, \ref{eq:131}, \ref{eq:132bis}, \ref{eq:132} are equivalent).
\end{enumerate}
\end{proposition}

\begin{proof}
We first prove that \ref{eq:131} $\Rightarrow$ \ref{eq:53bis},\ref{eq:132bis}.   
  Let us choose an arbitrary $\mu_1\in\overline{\dom(\frF)}$;
  by the definition of $\directionalm \frF\mmu t$ and arguing as in the proof of Theorem \ref{theo:propflfr}(2),
  for all $\mmu \in
  \Gamma_o(\mu_0,\mu_1)$ and $t\in \rI\mmu\frF$
  there exists $\Psi\in \frF[\mu_{t}]$ such that 
 \begin{align*}
\directionalm\Phi\mmu0&=\frac 1t\bram{\Phi}{\mu_t}\le
    -\frac{1}{t} \bram\Psi{\mu_0}+t\lambda W_2^2(\mu_0,\mu_1)=
    \directionalm{\Psi}{\mmu}t+t\lambda W_2^2(\mu_0,\mu_1)\\
&\le \directionalm\frF\mmu t +t\lambda W_2^2(\mu_0,\mu_1)
\end{align*}
  where we also used \eqref{eq:64}.
  If $\mmu\in \CondGammao\frF{\mu_0}{\mu_1}0$,
  by passing to the limit as $t\downarrow0$ we get \ref{eq:53bis}.

  In the second case, assuming that $\mmu\in \CondGammao\frF{\mu_0}{\mu_1}1$,
  we can pass to the limit as $t\uparrow1$ and we get \ref{eq:132bis}.  

We now prove item \eqref{item:ex4.171}. Let $\mu_1\in\dom(\frF)$, $\Psi\in\frF[\mu_1]$, $\mmu\in \CondGammao\frF{\mu_0}{\mu_1}0$, $s\in \rI\mmu\frF\cap (0,1)$, $\Phi_s\in \frF[\mu_s]$, with $\mu_s=\sfx^s_\sharp\mmu$. Assuming \ref{eq:53} and using \eqref{eq:54}, \eqref{eq:91}, \eqref{eq:64} and \eqref{Hdiss}, we have
\begin{align*}
\bram\Phi{\mu_1}&\le\directionalm{\Phi}{\mmu}0 \le \directional\frF{\mmu}{0+}\le\directionalm{\Phi_s}{\mmu}s+\lambda s W_2^2(\mu_0,\mu_1)\\
&=\frac{1}{1-s}\bram{\Phi_s}{\mu_1}+\lambda sW_2^2(\mu_0,\mu_1)\le -\frac{1}{1-s}\bram\Psi{\mu_s}+\lambda (1+s)W_2^2(\mu_0,\mu_1).
\end{align*}
By Lemma \ref{lem:lsc}, letting $s\downarrow0$ we get \ref{eq:35bis}.
Item \eqref{item:ex4.172} follows by \eqref{eq:54}, \eqref{eq:54bis}.
\end{proof}

\subsection{Extensions of dissipative \MPVF}
\label{subsec:extension}
Let us briefly study a few simple properties about
extensions of $\lambda$-dissipative {\MPVF}s.
The first one concerns the sequential closure in $\prob_2^{sw}(\TX)$
(the sequential closure may be smaller than the topological closure,
but see Proposition \ref{prop:finalmente}):
given $A\subset\prob_2(\TX)$, we will denote
by $\clo A$ its sequential closure defined by
\begin{equation*}
    \clo{A}:=\Big\{\Phi\in \prob_2(\TX):
    \exists\,\Phi_n\in A:\Phi_n\to\Phi\ \text{in }\prob_2^{sw}(\TX)\Big\}.
\end{equation*}

\begin{proposition}
  \label{prop:closure}
  If $\frF$ is a $\lambda$-dissipative \MPVF then
  its sequential closure
$\clo{\frF}$
  is $\lambda$-dissipative as well.
\end{proposition}
\begin{proof}
  If $\Phi^i$, $i=0,1$, belong to $\clo\frF$,
  we can find sequences $\Phi^i_n\in\frF$ such that
  $\Phi^i_n\to \Phi^i$ in $\prob_2^{sw}(\TX)$
  as $n\to\infty$, $i=0,1$. It is then sufficient to pass to the limit in the inequality
  \begin{displaymath}
    \bram{\Phi^0_n}{\Phi^1_n}\le \lambda W_2^2(\mu^0_n,\mu^1_n),\quad
    \mu^i_n=\sfx_\sharp \Phi^i_n
  \end{displaymath}
  using the lower semicontinuity property \eqref{eq:152}
  and the fact that convergence in $\prob_2^{sw}(\TX)$ yields
  $\mu^i_n\to \sfx_\sharp\Phi^i$ in $\prob_2(\X)$ as $n\to\infty$.
\end{proof}
A second result concerns the convexification of the sections of
$\frF$.
For every $\mu\in \dom(\frF)$ we set
\begin{align*}
  \conv\frF[\mu]:={}
  &
  \text{the convex hull of }\frF[\mu]=
  \Big\{\sum_k\alpha_k\Phi_k:\Phi_k\in \frF[\mu],
                      \alpha_k\ge 0,\sum_k\alpha_k=1\Big\},\\
    \cloco\frF[\mu]:=
    {}&
        \clo{\conv\frF[\mu]}.                    
\end{align*}
Notice that if $\frF[\mu]$ is bounded in $\prob_2(\TX)$ then
$\cloco\frF[\mu]$ coincides with the closed convex hull of
$\frF[\mu]$.
\begin{proposition}
  If $\frF$ is $\lambda$-dissipative, then
  $\conv\frF$ and $\cloco\frF$ are $\lambda$-dissipative as well.
\end{proposition}
\begin{proof}
  By Proposition \ref{prop:closure} and noting that $\cloco\frF\subset\clo{\conv\frF}$,
  it is sufficient to prove that $\conv\frF$ is $\lambda$-dissipative.
  By Lemma \ref{le:trick} it is not restrictive to assume $\lambda=0$.
  Let $\Phi^i\in \conv\frF[\mu_i]$, $i=0,1$;
  there exist positive coefficients $\alpha^i_k$, $k=1,\cdots,K$,
  with $\sum_k\alpha^i_k=1$, 
  and elements $\Phi^i_k\in \frF[\mu^i]$, $i=0,1$,
  such that $\Phi^i=\sum_{k=1}^K\alpha^i_k\Phi^i_k$.
  Setting $\beta_{h,k}:=\alpha^0_h\alpha^1_k$, we can apply
  Lemma \ref{le:convexity-pairing}
  and we obtain
    \begin{displaymath}
      \bram{\Phi^0}{\Phi^1}
      =
      \Big[{\sum_{h,k}\beta_{h,k}\Phi^0_h},{\sum_{h,k}\beta_{h,k}\Phi^1_k}
      \Big]_r
    \le
    \sum_{h,k}\beta_{h,k}\bram{\Phi^0_h}{\Phi^1_k}\le 0.\qedhere
  \end{displaymath}
\end{proof}
As a last step, we want to study the
properties of the extended \MPVF
\begin{equation}\label{eq:maxHilb}
  \begin{aligned}
    \hat{\frF}:=\Big\{&\Phi \in\prob_2(\TX):\mu=\sfx_\sharp\Phi\in \overline{\dom(\frF)},\\&
    \bram{\Phi}{\nu}+\bram{\Psi}{\mu} \le \lambda
      W_2^2(\mu,\nu) \quad\forall \,\Psi\in \frF,\ \nu=\sfx_\sharp\Psi
    \Big\}.
  \end{aligned}
\end{equation}
It is obvious that $\frF\subset \hat\frF$; if the domain of $\frF$ satisfies the geometric condition
\eqref{eq:157}, the following result shows that $\hat\frF$ provides
the maximal $\lambda$-dissipative extension of $\frF$.
\begin{proposition}
  \label{prop:maximal}
  Let $\frF$ be a $\lambda$-dissipative \MPVF.
  \begin{enumerate}[label=\rm(\alph*)]
  \item If $\frF'\supset \frF$ is $\lambda$-dissipative
    with $\dom(\frF')\subset
    \overline{\dom(\frF)}$, then $\frF'\subset
    \hat\frF$. In particular $\cloco{\clo{\frF}}\subset\hat\frF$.
  \item
    $\widehat{{\clo\frF}}=\hat\frF$ and
    $\widehat{{\conv\frF}}=\hat\frF$.
  \item $\hat\frF$ is sequentially closed and $\hat\frF[\mu]$ is
    convex
    for every $\mu\in \dom(\hat\frF)$.
  \item If $\dom(\frF)$ satisfies \eqref{eq:19}, then
    the restriction of $\hat\frF$ to $\dom(\frF)$ is
    $\lambda$-dissipative and
    for every $\mu_0,\mu_1\in \dom(\frF)$
    \begin{equation}
      \label{eq:158}
      \directional\frF{\mmu}{0+}=
      \directional{\hat\frF}{\mmu}{0+},\quad
      \directional\frF{\mmu}{1-}=
      \directional{\hat\frF}{\mmu}{1-}
      \quad\text{for every }
      \mmu\in \CondGammao\frF{\mu_0}{\mu_1}{01}.
    \end{equation}
  \item If $\mu_0\in \overline{\dom(\frF)},\
      \mu_1\in \dom(\frF)$
      and $\CondGammao\frF{\mu_0}{\mu_1}1\neq\emptyset$
      then
      \begin{equation*}
        \Phi_i\in \hat\frF[\mu_i]\quad\Rightarrow\quad
        \bram{\Phi_0}{\Phi_1}\le \lambda W_2^2(\mu_0,\mu_1).
      \end{equation*}
    \item If
      \begin{equation}
        \text{for every $\mu_0,\mu_1\in \overline{\dom(\frF)}$ 
        the set $\CondGammao\frF{\mu_0}{\mu_1}{01}$ is not
        empty,}
      \label{eq:157}  
      \end{equation}
      then $\hat \frF$ is $\lambda$-dissipative as well
      and for every $\mu_0,\mu_1\in \overline{\dom(\frF)}$
      \eqref{eq:158} holds.
  \end{enumerate}
\end{proposition}
\begin{proof}
  Claim (a) is obvious since every $\lambda$-dissipative extension $\frF'$
  of
  $\frF$ in $\overline{\dom(\frF)}$ satisfies $\frF'\subset
  \hat\frF$.

  (b) Let us prove that if $\Phi\in \hat \frF$
  then
  $\Phi\in \widehat{\clo\frF}$. If $\Psi\in \clo\frF$ we
  can find a sequence $\Psi_n\in \frF$ converging to
  $\Psi$ in $\prob_2^{sw}(\TX)$ as $n\to\infty$.
  We can then pass to the limit in the inequalities
  \begin{displaymath}
    \bram{\Phi}{\nu_n}+\bram{\Phi_n}{\mu}\le \lambda
    W_2^2(\mu,\nu_n),\quad
    \mu=\sfx_\sharp\Phi,\ \nu_n=\sfx_\sharp\Psi_n,
  \end{displaymath}
  using the lower semicontinuity results of Lemma \ref{lem:lsc}. We conclude since $\overline{\dom(\frF)}=\overline{\dom(\clo\frF)}$.

  In order to prove that $\Phi\in \hat\frF\ \Rightarrow\ \Phi\in
  \widehat{{\conv\frF}}$ we take
  $\Psi=\sum\alpha_k\Psi_k\in \conv{\frF}$;
  for some $\Psi_k\in \frF[\nu]$, $\nu=\sfx_\sharp\Psi\in \dom(\frF)$, and positive coefficients
  $\alpha_k$, $k=1,\cdots,K$, with $\sum_{k}\alpha_k=1$.
  Taking a convex combination of the inequalities
  \begin{displaymath}
    \bram{\Phi}{\nu}+\bram{\Psi_k}\mu\le \lambda W_2^2(\mu,\nu),\quad
    \text{for every }k=1,\cdots,K,
  \end{displaymath}
  and using Lemma \ref{le:convexity-pairing}
  we obtain
  \begin{displaymath}
    \bram{\Phi}{\nu}+\bram{\Psi}\mu
    \le
    \sum_k\alpha_k\Big(\bram{\Phi}{\nu}+\bram{\Psi_k}\mu\Big)
    \le \lambda W_2^2(\mu,\nu).
  \end{displaymath}
  The proof of claim (c) follows by a similar argument.

  (d)
  Let $\mu_i\in \dom(\frF)$, 
  $\Phi_i\in \hat\frF[\mu_i]$, $i=0,1$, and
  $\mmu\in \CondGammao\frF{\mu_0}{\mu_1}{01}$.
  The implication \ref{eq:35bis}$\Rightarrow$\ref{eq:53}
  of Proposition \ref{prop:scalmscal}
  applied to $\mmu$ and to $\mathsf s_\sharp\mmu$ 
  yields
  \begin{equation*}
    \directionalm{\Phi_0}{\mmu}0\le
    \directional\frF\mmu{0+},\quad
    \directionalm{\Phi_1}{\mathsf s_\sharp \mmu}0\le
    \directional\frF{\mathsf s_\sharp\mmu}{0+}=
    -\directional\frF{\mmu}{1-}
  \end{equation*}
  so that \eqref{eq:92} yields
  \begin{displaymath}
    \bram{\Phi_0}{\Phi_1}\le
    \directionalm{\Phi_0}{\mmu}0+
    \directionalm{\Phi_1}{\mathsf s_\sharp \mmu}0
    \le
    \directional\frF\mmu{0+}-
    \directional\frF{\mmu}{1-}\le
    \lambda W_2^2(\mu_0,\mu_1).
  \end{displaymath}
  In order to prove \eqref{eq:158}
  we observe that $\frF \subset \hat \frF$ so that, for every $\mmu\in
  \CondGammao{\frF}{\mu_0}{\mu_1}{01}$ and every 
  $t\in \rI\mmu{\frF}$,
   we have 
  $\directionalm{\frF}\mmu t \le \directionalm{\hat \frF}\mmu t$ and
  $\directionalp{\frF}\mmu t \ge  \directionalp{\hat \frF}\mmu t$,
  hence \eqref{eq:158} is a consequence of Definition
  \ref{def:directiona} and Theorem \ref{theo:propflfr}.
  
  The proof of claim (f) follows by the same argument.

  In the case of claim (e),
  we use
  the implication \ref{eq:35bis}$\Rightarrow$\ref{eq:132}
  of Proposition \ref{prop:scalmscal}
  applied to $\mmu$
  and the implication \ref{eq:35bis}$\Rightarrow$\ref{eq:53bis}
  applied to $\mathsf s_\sharp\mmu$, obtaining
  \begin{equation*}
    \directionalm{\Phi_0}{\mmu}0\le
    \lambda W_2^2(\mu_0,\mu_1)
    +\directional\frF\mmu{1-},\quad
    \directionalm{\Phi_1}{\mathsf s_\sharp \mmu}0\le
    \directional\frF{\mathsf s_\sharp\mmu}{0+}=
    -\directional\frF{\mmu}{1-}
  \end{equation*}
  and then 
  \begin{displaymath}
    \bram{\Phi_0}{\Phi_1}\le
    \directionalm{\Phi_0}{\mmu}0+
    \directionalm{\Phi_1}{\mathsf s_\sharp \mmu}0
    \le
    \lambda W_2^2(\mu_0,\mu_1).\qedhere
  \end{displaymath}
\end{proof}

\section{Examples of \texorpdfstring{$\lambda$}{l}-dissipative MPVFs}\label{sec:examples}
In this section we present significant examples of $\lambda$-dissipative MPVFs
which are interesting for applications.

\subsection{Subdifferentials of \texorpdfstring{$\lambda$}{l}-convex functionals}\label{sec:Exsubdiff}
Recall that
a functional
$\func:\prob_2(\X)\to(-\infty,+\infty]$ is 
$\lambda$-(geodesically) convex on $\prob_2(\X)$ (see \cite[Definition 9.1.1]{ags})
if
for any $\mu_0,\mu_1$ in the proper domain
$D(\func):=\{\mu\in\prob_2(\X)\mid\func(\mu)<+\infty\}$
there exists $\mmu\in\Gamma_o(\mu_0,\mu_1)$ such that
\[\func(\mu_t)\le (1-t)\func(\mu_0)+t\func(\mu_1)-\frac{\lambda}{2}t(1-t) W_2^2(\mu_0,\mu_1)\qquad\forall \, t\in[0,1],\]
where $(\mu_t)_{t\in[0,1]}$ is the constant speed geodesic induced by $\mmu$, i.e. $\mu_t=\sfx^t_{\sharp}\mmu$.

The \emph{Fr\'echet subdifferential $\boldsymbol{\partial}\func$} of $\func$ 
\cite[Definition 10.3.1]{ags}
is a \MPVF which can be characterized \cite[Theorem 10.3.6]{ags} by
\begin{equation*}
  \Phi\in \boldsymbol{\partial}\func[\mu]
  \quad\Leftrightarrow\quad
  \mu\in D(\func),\
  \func(\nu)-\func(\mu)\ge -\brap{\Phi}{\nu}
+\frac{\lambda}{2}W_2^2(\mu,\nu)
\quad\text{for every }\nu\in D(\func).
\end{equation*}
According to the notation introduced in \eqref{eq:20}, we set
\begin{equation}\label{eqdef:Vphi}
  -\boldsymbol{\partial}\func[\mu]=
  J_\sharp \boldsymbol{\partial}\func[\mu],\quad
  J(x,v)=(x,-v),
\end{equation}
and we have
the following result.

\begin{theorem}
  If $\func:\prob_2(\X)\to(-\infty,+\infty]$ is a proper, lower semicontinuous and $\lambda$-convex functional,
  then
  $-\boldsymbol{\partial}\func$ is a \emph{$(-\lambda)$-dissipative}
  \MPVF.
\end{theorem}

Referring to \cite{ags}, here we list interesting and explicit examples of
$(-\lambda)$-dissipative MPVFs
induced by proper, lower semicontinuous and $\lambda$-convex functionals,
focusing on the cases
when $\dom(\boldsymbol\partial\func)=\prob_2(\X).$ 
\begin{enumerate}
\item \emph{Potential energy.}  Let $P:\X \to \R$ be
  a l.s.c. and $\lambda$-convex functional
  satisfying
  \begin{displaymath}
    |\partial^o P(x)| \le C(1+|x|) \quad \text{for every }x \in \X,
  \end{displaymath}
  for some constant $C>0$,
  where $\partial^o P(x)$ is the element of minimal norm in $\partial P(x)$.
  By \cite[Proposition 10.4.2]{ags}
  the PVF
  \begin{equation*}
    \frF[\mu]:=(\ii_{\X} , -\partial^o P)_{\sharp} \mu,
    \quad
    \mu\in \prob_2(\X),
  \end{equation*}
  is a $(-\lambda)$-dissipative selection of
  $-\boldsymbol\partial\mathcal F_P$ for the
  potential energy functional
  \[ \mathcal F_P(\mu) := \int_\X P \de \mu, \quad \mu \in \prob_2(\X).\]
\item \emph{Interaction energy.}
  If $W:\X \to [0,+\infty)$ is an even, differentiable, and $\lambda$-convex
  function for some $\lambda \in \R$, whose differential has a linear growth, 
  then, by \cite[Theorem 10.4.11]{ags}, the PVF
  \begin{displaymath}
    \frF[\mu]:= (\ii_{\X} , (-\nabla W \ast \mu))_{\sharp}\mu,
    \quad
    \mu\in \prob_2(\X),
  \end{displaymath}
  is a $(-\lambda)$-dissipative selection of $-\boldsymbol\partial\mathcal F_W$,
  the opposite of the Wasserstein subdifferential of
  the interaction energy functional  
  \[ \func_W(\mu) := \frac{1}{2} \int_{\X^2} W(x-y) \de (\mu \otimes \mu)(x,y) , \quad \mu \in \prob_2(\X).\]
\item \emph{Opposite Wasserstein distance.} Let $\bar{\mu} \in \prob_2(\X)$ be fixed and consider the functional $\func_{\text{Wass}} : \prob_2(\X) \to \R$ defined as
\[ \func_{\text{Wass}}(\mu) := - \frac{1}{2} W_2^2(\mu, \bar{\mu}), \quad \mu \in \prob_2(\X),\]
which is geodesically $(-1)$-convex \cite[Proposition 9.3.12]{ags}.
Setting
\[ \boldsymbol b(\mu) := \argmin \left \{ \int_\X |\boldsymbol b(x)-x|^2 \de \mu:
    \boldsymbol b=\bry\ggamma\in
     L^2_\mu(\X;\X), 
    \ \ggamma \in \Gamma_o(\mu, \bar{\mu}) \right \} ,\]
the PVF
\begin{equation*}
  \frF[\mu]:=(\ii_\X, \ii_\X-\boldsymbol b(\mu) )_{\#}\mu,\quad
  \mu\in \prob_2(\X)
\end{equation*}
is a selection of $-\boldsymbol{\partial}\func_{\text{Wass}}(\mu)$ and it is therefore
$1$-dissipative.
\end{enumerate}

\subsection{MPVF concentrated on the graph of a multifunction} \label{subsec:graphB}
The previous example of Section \ref{sec:Exsubdiff} has a natural generalization in terms of dissipative graphs in $\X\times \X$
\cite{AuC,AuF,BrezisFR}.
We consider a (not empty) $\lambda$-dissipative set
$F\subset \X\times \X$, i.e.~satisfying
\begin{equation*}
  \scalprod{v_0-v_1}{x_0-x_1}\le\lambda |x_0-x_1|^2
  \quad\text{for any }(x_0,v_0),\ (x_1,v_1)\in F.
\end{equation*}
The corresponding \MPVF defined as
\[\frF:=\left\{\Phi\in\prob_2(\TX)\mid \Phi\text{ is concentrated on }F
  \right\}\]
is $\lambda$-dissipative as well.
In fact, if
$\Phi_0,\Phi_1\in \frF$ with $\nu_i=\sfx_\sharp\Phi_i$, $i=0,1$,
and $\Ttheta\in \Lambda(\Phi_0,\Phi_1)$
then $(x_0,v_0,x_1,v_1)\in F\times F$ $\Ttheta$-a.e., so that 
\begin{align*}
                                \int_{\TX\times \TX} \scalprod{v_0-v_1}{x_0-x_1}\de\Ttheta(x_0,v_0,x_1,v_1)
  \le \lambda\int_{\TX\times \TX}|x_0-x_1|^2\de\Ttheta
     =\lambda W_2^2(\nu_0,\nu_1).
\end{align*}
since $(\sfx^0,\sfx^1)_\sharp\Ttheta\in \Gamma_o(\nu_0,\nu_1)$.
Taking the supremum w.r.t.~$\Ttheta\in \Lambda(\Phi_0,\Phi_1)$ we
obtain $\brap{\Phi_0}{\Phi_1}\le \lambda W_2^2(\nu_0,\nu_1)$
which is even stronger than $\lambda$-dissipativity.
If $\dom(F)=\X$
then $\dom(\frF)$ contains $\prob_{\rm c}(\X)$, the set of Borel probability measures with compact support.
If $F$ has also a linear growth, then it is easy to check that
$\dom(\frF)=\prob_2(\X)$ as well.

Despite the analogy just shown with dissipative operators in Hilbert
spaces, there are important differences with the Wasserstein
framework, as highlighted in the following examples. The main point here is that the dissipativity property
of Definition \ref{def:dissipative} does not force
the sections $\sfv_\sharp \frF[\mu]$ to belong to the tangent space
$\Tan_\mu\prob_2(\X)$.

\begin{example}\label{ex:unboundF}
  Let $\X=\R^2$,
  let $B:=\{ x \in \R^2 \mid |x| \le 1\}$ be the closed unit ball, 
  let $\lebd$ be the (normalized) Lebesgue measure on $B$,
  and let $\Rotn:\R^2\to\R^ 2$, $\Rotn(x_1,x_2)=(x_2,-x_1)$ be the
  anti-clockwise rotation of $\pi/2$ degrees. We define the \MPVF
\[ \frF[\nu] = \begin{cases} (\ii_{\R^2}, 0)_{\sharp} \nu, \quad &\text{ if } \nu \in \prob_2(\R^2) \setminus \{ \lebd\}, \\ \left \{ (\ii_{\R^2}, a\Rotn)_{\sharp}\lebd \mid a \in \R \right \}, \quad &\text{ if } \nu = \lebd. \end{cases} \]
Observe that $\dom(\frF) = \prob_2(\R^2)$ and $\frF$ is obviously unbounded at $\nu= \lebd$. $\frF$ also satisfies \eqref{Hdiss} with $\lambda=0$ (hence it is dissipative): it is enough to check that 
\begin{equation}
\bram{(\ii_{\R^2}, a\Rotn)_{\sharp} \lebd}{\nu} =0 \quad \text{ for
  every } \nu \in \prob_2(\R^2), \, a \in \R.\label{eq:31}
\end{equation}
To prove \eqref{eq:31}, we notice that the optimal transport plan from
$\lebd$ to $\nu$ is concentrated on a map and optimal maps belong to
the tangent space $\Tan_{\lebd}\prob_2(\R^2)$ \cite[Prop.~8.5.2]{ags}; by Remark
\ref{rem:particular}
we have just to check that 
\begin{equation*} 
\int_{\R^2} \scalprod{\Rotn(x)}{\nabla \varphi(x)}\de \lebd(x) =0 \quad \forall \varphi \in \rmC^{\infty}_c(\R^2),
\end{equation*}
that is a consequence of the Divergence Theorem on $B$. This example
is in contrast with the Hilbertian theory of dissipative operators
according to which an everywhere defined dissipative operator is locally bounded (see \cite[Proposition 2.9]{BrezisFR}).
\end{example}

\begin{example}
  \label{ex:contnotmax}
In the same setting of the previous example, let us define the \MPVF
\[ \frF[\nu] = (\ii_{\R^2}, \Rotn)_{\sharp} \nu, \quad
  \Rotn(x_1,x_2)=(x_2,-x_1),\quad
  \nu \in \prob_2(\R^2).\]
It is easy to check that $\frF$ is dissipative and Lipschitz
continuous (as a map from $\prob_2(\R^2)$ to
$\prob_2(\mathrm{T}\R^2)$). Moreover, arguing as in Example
\ref{ex:unboundF}, we can show that $(\ii_{\R^d}, 0)_{\sharp} \lebd
\in \hat \frF [\lebd]$, where $\hat\frF$ is defined in
\eqref{eq:maxHilb}. This is again in contrast with the Hilbertian
theory of dissipative operators, stating that a single valued,
everywhere defined, and continuous dissipative operator coincides with
its
maximal extension (see \cite[Proposition 2.4]{BrezisFR}).
\end{example}

\subsection{Interaction field induced by a dissipative map}
Let us consider the Hilbert space $\Y=\X^n$, $n\in\N$, endowed with the scalar product
$\langle \boldsymbol x,\boldsymbol y\rangle:=\frac 1n \sum_{i=1}^n \langle x_i,y_i\rangle$,
for every $\boldsymbol x=(x_i)_{i=1}^n,\ \boldsymbol y=(y_i)_{i=1}^n\in \X^n$.
We identify $\TY$ with $(\TX)^n$
and we denote by $\sfx^i,\sfv^i$ the $i$-th coordinate maps.
Every permutation $\sigma:\{1,\cdots,n\}\to\{1,\cdots,n\}$ in $\mathrm{Sym}(n)$
operates on $\Y$
by the obvious formula
$\sigma(\boldsymbol x)_i=x_{\sigma(i)}$, $i=1,\cdots,n$,
$\boldsymbol x\in \Y$.

Let $G: \Y\to\Y$ be a Borel $\lambda$-dissipative
map bounded on bounded sets (this property is always true
if $\Y$ has finite dimension) and satisfying
\begin{equation}
  \label{eq:67}
  \xx\in \dom(G)\quad\Rightarrow\quad
  \sigma(\xx)\in \dom(G),\ 
  G(\sigma(\xx))=\sigma(G(\xx))\quad
  \text{for every permutation $\sigma$}.
\end{equation}
Denoting by $(G^1,\cdots,G^n)$ the components of $G$, by $\sfx^i$ the projections from $\Y$ to $\X$ and by $\mu^{\otimes n}=\bigotimes_{i=1}^n\mu$,
the \MPVF
\begin{equation*}
  \frF[\mu]:=(\sfx^1,G^1)_\sharp \mu^{\otimes n}
  \quad\text{with domain }
  \dom(\frF):=\prob_b(\X)
\end{equation*}
is $\lambda$-dissipative as well.
In fact, if $\mu,\nu\in \dom(\frF)$,
$\Phi=(\sfx^1,G^1)_\sharp \mu^{\otimes n}$
and $\Psi= (\sfx^1,G^1)_\sharp\nu^{\otimes n}$,
and $\ggamma\in \Gamma_o(\mu,\nu)$, we can
consider the plan
$\boldsymbol \beta:=
P_\sharp \ggamma^{\otimes n}\in \Gamma(\mu^{\otimes n},\nu^{\otimes
  n})$,
where $P((x_1,y_1),\cdots,(x_n,y_n) ):=
((x_1,\cdots,x_n),(y_1,\cdots, y_n))$.
Considering the map
$H^1(\xx,\yy):=(x_1,G^1(\xx),y_1,G^1(\yy))$
we have
$\Ttheta:=H^1_\sharp \boldsymbol\beta\in \Lambda(\Phi,\Psi)$, so that
\begin{align*}
  \bram\Phi\Psi
  &
    \le 
  \int \langle v_1-w_1,x_1-y_1\rangle\,\d\Ttheta(x_1,v_1,y_1,w_1)
  =
    \int \langle G^1(\xx)-G^1(\yy),x_1-y_1\rangle
  \,\d\boldsymbol
    \beta(\xx,\yy)
  \\&
  =\frac 1n\sum_{k=1}^n
  \int \langle G^k(\xx)-G^k(\yy),x_k-y_k\rangle
  \,\d\boldsymbol
  \beta(\xx,\yy)
  =
  \int \langle G(\xx)-G(\yy),\xx-\yy\rangle
  \,\d\boldsymbol
  \beta(\xx,\yy)
\end{align*}
where we used \eqref{eq:67} and the invariance of
$\boldsymbol \beta$ with respect to permutations.
The $\lambda$-dissipativity of $G$ then yields
\begin{align*}
  \int \langle G(\xx)-G(\yy),\xx-\yy\rangle
  \,\d\boldsymbol
  \beta(\xx,\yy)&
    \le
    \lambda
    \int |\xx-\yy|^2_\Y
  \,\d\boldsymbol
    \beta(\xx,\yy)
=
   \lambda \frac1n\sum_{k=1}^n
    \int |x_k-y_k|^2_\Y
  \,\d\boldsymbol
    \beta(\xx,\yy)\\
   & =
  \lambda\frac1n\sum_{k=1}^n
    \int |x_k-y_k|^2_\Y
  \,\d\boldsymbol
    \ggamma(x_k,y_k)
    =\lambda W_2^2(\mu,\nu).
\end{align*}
A typical example when $n=2$ is provided by
\begin{displaymath}
  G(x_1,x_2):=(A(x_1-x_2),A(x_2-x_1))
\end{displaymath}
where $A:\X\to \X$ is a Borel, locally bounded, dissipative
and antisymmetric map 
satisfying $A(-z)=-A(z)$.
We easily get
\begin{align*}
  \langle
  &
    G(\xx)-G(\yy),\xx-\yy\rangle
  \\&=
    \frac 12
    \Big(\langle A(x_1-x_2)-A(y_1-y_2),x_1-y_1\rangle
    -
    \langle A(x_1-x_2)-A(y_1-y_2),x_2-y_2\rangle\Big)
  \\&  = \frac{1}{2}
    \langle A(x_1-x_2)-A(y_1-y_2),x_1-x_2-(y_1-y_2)\rangle
  \le 0.
\end{align*}
In this case
\begin{displaymath} \frF[\mu]=(\ii_\X,\boldsymbol a[\mu])_\sharp\mu,\quad
  \boldsymbol a[\mu](x)=\int_\X A(x-y)\,\d\mu(y)
  \quad\text{for every }x\in \X.
\end{displaymath}

\section{Solutions to Measure Differential Inclusions}\label{sec:MDE}
\subsection{Metric characterization and EVI}
Let $\interval$ denote
an arbitrary (bounded or unbounded) interval in $\R$.

The aim of this section is to study a suitable notion of solution to
the following
differential inclusion in the $L^2$-Wasserstein space of probability measures
\begin{equation}
  \label{eq:CP}
  \dot{\mu}(t) \in \frF[\mu(t)],\qquad t\in \interval,
\end{equation}
driven by a \MPVF $\frF$ as in Definition \ref{def:MPVF}. In
particular, we will address the usual Cauchy problem when \eqref{eq:CP}
is supplemented by a given initial condition.

Measure Differential Inclusions have been introduced in \cite{Piccoli_MDI} extending to the multi-valued framework the theory of Measure Differential Equations developed in \cite{Piccoli_2019}. In these papers, the author aims to describe the evolution of curves in the space of probability measures under the action of a so called \emph{probability vector field} $\frF$ (see Definition \ref{def:MPVF} and Remark \ref{Vmultifunc}). However, as exploited also in \cite{Camilli_MDE}, 
the definition of solution to \eqref{eq:CP} given in \cites{Piccoli_2019,Piccoli_MDI,Camilli_MDE} is too weak and it does not enjoy uniqueness property which is recovered only at the level of the semigroup through an approximation procedure.

From the Wasserstein viewpoint, the simplest way to interpret \eqref{eq:CP} is
to ask for a 
locally absolutely continuous curve
$\mu:\interval\to\prob_2(\X)$
to satisfy
\begin{equation}
  \label{eq:CPW}
  (\ii_\X , \vv_t)_{\sharp} \mu_t \in {\frF}[\mu_t] \quad
    \text{for a.e. }t \in
    \interval,
\end{equation}
where $\vv$ is the Wasserstein metric velocity vector associated to $\mu$ (see Theorem \ref{thm:tangentv}).
Even in the case of a regular PVF, however, \eqref{eq:CPW} is too strong, since there is no reason why
a given $\frF[\mu_t]$ should be associated to a vector field of the tangent space $\Tan_{\mu_t}\prob_2(\X)$. 
Starting from \eqref{eq:CPW}, we thus introduce a weaker definition of solution to
\eqref{eq:CP},
modeled on the so-called EVI formulation for gradient flows,
which will eventually suggest, as a natural formulation of \eqref{eq:CP}, the relaxed version of \eqref{eq:CPW}
as a differential inclusion with respect to
the extension $\hat \frF$ of $\frF$ introduced in \eqref{eq:maxHilb}.

We start from this simple remark:
whenever $\frF$ is $\lambda$-dissipative, recalling Theorem
\ref{thm:refdiff} and Remark \ref{rmk:equivdiss}, one easily sees that
every locally absolutely continuous solution according to the above definition \eqref{eq:CPW} also satisfies
the Evolution Variational Inequality ($\lambda$-\wEVI)
\begin{equation} \label{eq:EVI}
  \tag{$\lambda$-\wEVI}
  \frac{1}{2}
  \frac{\de}{\de t} W_2^2(\mu_t, \nu) \le \lambda W_2^2(\mu_t, \nu)
  - \bram{\Phi}{\mu_t}\quad
  \text{in }\mathscr D'\big(\intt\interval\big),
\end{equation}
for every $\nu \in \dom(\frF)$ and every $\Phi \in \frF[\nu]$,
where $\bram{\cdot}{\cdot}$ is the functional pairing in Definition
\ref{def:scalarprodop} (in fact, \eqref{eq:EVI} holds a.e.~in
$\interval$).
This provides a heuristic motivation for the following definition.

\begin{definition}[$\lambda$-Evolution Variational Inequality]
  Let $\frF$ be a \MPVF and let $\lambda \in \R$.
  We say that a continuous curve $\mu: \interval  \to \overline{\dom(\frF)}$
  is a \emph{$\lambda$-\wEVI solution} to \eqref{eq:CP} for the \MPVF $\frF$ if 
  \eqref{eq:EVI} 
  holds for every $\nu
\in \dom(\frF)$ and every $\Phi \in \frF[\nu]$.\\
A $\lambda$-\wEVI solution $\mu$ is said to be a \emph{strict solution} if $\mu_t\in \dom(\frF)$ for every $t\in \interval$,
$t > \inf \interval$.
\\
A $\lambda$-\wEVI solution $\mu$ is said to be a \emph{global solution} if $\sup\interval=+\infty$.
\end{definition}
In Example \ref{ex:rulla} we will clarify the interest in imposing no
more than continuity in the above definition.

Recall that the right 
upper and lower Dini derivatives of 
  a function $\zeta:\interval \to \R$ are defined for every $t \in \interval$, $t < \sup \interval$ by
\begin{equation}\label{eq:RupDerW2}
  \updt \zeta(t)
  \limsup_{h \downarrow 0}
  \frac{\zeta(t+h)-\zeta(t)}h,
  \qquad
  \lodt \zeta(t)
  \liminf_{h \downarrow 0}
  \frac{\zeta(t+h)-\zeta(t)}h.  
\end{equation}

\begin{remark} 
  Arguing as in \cite[Lemma A.1]{MuratoriSavare} and
using the lower semicontinuity of the map $t\mapsto \bram{\Phi}{\mu_t}$,
the distributional inequality of \eqref{eq:EVI}
can be equivalently reformulated
in terms of the right upper or lower Dini derivatives
  of the squared distance function
  and requiring the condition to hold for every
  $t\in \intt{\interval}$:
  \begin{align}
    \label{eq:81}
    \tag{$\lambda$-EVI$_1$}
    \frac12\updt W_2^2(\mu_t,\nu)
    &\le \lambda W_2^2(\mu_t,\nu)-
      \bram{\Phi}{\mu_t}
    &
    &\text{for every }
          t\in \intt{\interval},\ \Phi\in \frF,\ \nu=\sfx_\sharp\Phi,\\ \label{eq:81n}
    \tag{$\lambda$-EVI$_2$}
      \frac12\lodt W_2^2(\mu_t,\nu)
      &\le \lambda W_2^2(\mu_t,\nu)-
      \bram{\Phi}{\mu_t}
          &&
    \text{for every }
          t\in \intt{\interval},\ \Phi\in \frF,\ \nu=\sfx_\sharp\Phi.
  \end{align}
A further equivalent formulation
  \cite[Theorem 3.3]{MuratoriSavare}
  involves the difference quotients:
  for every $s,t\in \interval$, $s<t$
  \begin{equation}
    \label{eq:140}
    \tag{$\lambda$-EVI$_3$}
    {\mathrm e^{-2\lambda(t-s)}}\, W_2^2(\mu_t,\nu)-
    W_2^2(\mu_s,\nu)\le -2 \int_s^t \mathrm e^{-2\lambda (r-s)}\bram{\Phi}{\mu_r}\,\d r
    \quad \text{for every }
          \Phi\in \frF,\ \nu=\sfx_\sharp\Phi.
  \end{equation}
  Finally, 
  if $\mu$ is also locally absolutely continuous, then
  \eqref{eq:81} and \eqref{eq:81n} are also equivalent to
  \begin{equation*}
    \begin{aligned}
      \frac12\frac\d{\d t} W_2^2(\mu_t,\nu)
      &\le \lambda W_2^2(\mu_t,\nu)-
      \bram{\Phi}{\mu_t}
    \end{aligned}
    \quad
    \text{for a.e.~}
    t\in \interval
    \text{ and every}\ \Phi\in \frF,\ \nu=\sfx_\sharp\Phi.
  \end{equation*}
\end{remark}

The following Lemma provides a further insight.
\begin{lemma}
  \label{le:easy}
  Let 
  $\frF$ be a
  $\lambda$-dissipative \MPVF 
  and let $\mu: \interval \to \overline{\dom(\frF)}$
  be a continuous $\lambda$-\wEVI solution to \eqref{eq:CP}.
  We have
  \begin{subequations}
  \begin{align} \label{eq:SEVIa}
      \frac{1}{2}\updt
      W_2^2(\mu_t, \nu)
      \le \directional\frF{\mmu}{0+}
                        \quad
                        &\text{for every $\nu\in \overline{\dom(\frF)}$, $t\in \intt{\interval}$,
                          $\mmu\in \CondGammao\frF{\mu_t}\nu0$,}\\
    \label{eq:142}
                        \frac{1}{2}\updt
      W_2^2(\mu_t, \nu)
      \le \lambda W_2^2(\mu_t,\nu)+\directional\frF{\mmu}{1-}
                        \quad
                        &\text{for every $\nu\in \overline{\dom(\frF)}$, $t\in \intt{\interval},$
                          $\mmu\in \CondGammao\frF{\mu_t}\nu1$.}
  \end{align}
  \end{subequations}
If moreover $\mu$ is locally absolutely continuous 
  with Wasserstein velocity
  field
  $\vv$ satisfying \eqref{eq:74}
  for every $t$ in the subset $A(\mu)\subset \interval$ of full
  Lebesgue measure,
  then
    \begin{subequations}
  \begin{align}
    \label{eq:76}
    \bram{(\ii_\X ,
      \vv_t)_{\sharp}\mu_t}{\nu}
    &\le \lambda W_2^2(\mu_t,\nu)-
    \bram{\Phi}{\mu_t}
    &&\text{if $t\in A(\mu)$,\ $\Phi\in \frF,\ \nu=\sfx_\sharp\Phi$},
    \\
          \label{eq:87b}
    \directionalm{(\ii_\X ,
      \vv_t)_{\sharp}\mu_t}{\mmu}0
    &\le \directional\frF{\mmu}{0+}
    && \text{if $t\in A(\mu)$, } \nu\in \overline{\dom(\frF)},\ \mmu\in
       \CondGammao\frF{\mu_t}{\nu}{0},
       \\
          \label{eq:87c}
    \directionalm{(\ii_\X ,
      \vv_t)_{\sharp}\mu_t}{\mmu}0
    &\le
      \lambda W_2^2(\mu_t,\nu)+
    \directional\frF{\mmu}{1-}
    && \text{if $t\in A(\mu)$, } \nu\in \overline{\dom(\frF)},\ \mmu\in
       \CondGammao\frF{\mu_t}{\nu}{1}.
  \end{align}
\end{subequations}
\end{lemma}
\begin{proof}
  In order to check \eqref{eq:76} it is sufficient to
  combine \eqref{eq:refdiffa} of Theorem \ref{thm:refdiff}
  with
  \eqref{eq:81}.
  \eqref{eq:87b} and \eqref{eq:87c} then follow applying Proposition \ref{prop:scalmscal}.  
  Let us now prove \eqref{eq:SEVIa}:
  let us fix $\nu \in \overline{\dom(\frF)}$ and $t \in \intt{\interval}$. 
  Take $\mmu \in   \Gamma_o(\mu_t,\nu)$
  and define the constant speed geodesic
  $(\nu_{s})_{s \in [0,1]}$ by $\nu_{s}: = (\sfx^s)_{\sharp}\mmu$,
  thus in particular $\nu_{0}=\mu_t$ and $\nu_{1}=\nu$. Then by Lemma
  \ref{lem:aeb}, for every
  $s\in     \rI\mmu\frF\cap (0,1) $
  and $\Phi_s\in \frF(\nu_s)$ we have
\begin{equation*}
\begin{split}
  \frac 12\updt W_2^2(\mu_t, \nu) &
  \le \frac{1}{2s} \updt W_2^2(\mu_t, \nu_s)
  \le - \frac{1}{s} \bram{\Phi_s}{\mu_t} + \frac{\lambda}{s} W_2^2(
  \mu_t,\nu_{s})
  \\&\le \directionalm\frF{\mmu}{s}+\lambda s W_2^2(\mu_t,\nu),
\end{split}
\end{equation*}
where the second inequality comes
from \eqref{eq:81}.
Taking $\mmu \in \CondGammao\frF{\mu_t}{\nu}{0}$ and passing to the limit as $s\downarrow0$ we get
\eqref{eq:SEVIa}. Analogously for \eqref{eq:142}.
\end{proof}
We can now give an interpretation of
absolutely continuous $\lambda$-\EVI solutions
in terms of differential inclusions.

\begin{theorem}\label{theo:equivwssol} Let $\frF$ be a
  $\lambda$-dissipative \MPVF
  and let $\mu: \interval \to \overline{\dom(\frF)}$ be a locally absolutely
  continuous 
  curve.
  \begin{enumerate}
  \item If $\mu$ satisfies the differential
    inclusion \eqref{eq:CPW} driven by any $\lambda$-dissipative
    extension of $\frF$ in
     ${\dom(\frF)}$, 
    then $\mu$ is also a $\lambda$-\wEVI solution to \eqref{eq:CP} for $\frF$.
  \item $\mu$ is a $\lambda$-\EVI solution
    of \eqref{eq:CP} for $\frF$
    if and only if
    \begin{equation}
      \label{eq:156}
      (\ii_\X , \vv_t)_{\sharp} \mu_t \in {\hat\frF}[\mu_t] \quad
    \text{for a.e. }t \in
    \interval.
    \end{equation}
  \item If $\dom(\frF)$ satisfies \eqref{eq:19} and
    $\mu_t\in \dom(\frF)$
    for a.e.~$t\in \interval$,
    then the following properties are equivalent:
    \begin{itemize}
    \item[-] $\mu$ is a $\lambda$-\wEVI solution to \eqref{eq:CP} for $\frF$.
    \item[-] $\mu$ satisfies \eqref{eq:87b}.
    \item[-] $\mu$ is a $\lambda$-\wEVI solution to \eqref{eq:CP} for the restriction of $\hat\frF$ to $\dom(\frF)$.
    \end{itemize}
  \item
    If $\frF$ satisfies \eqref{eq:157} 
    then
    $\mu$ is a $\lambda$-\wEVI solution to \eqref{eq:CP} for $\frF$ if and only if
    it is a $\lambda$-\EVI solution to \eqref{eq:CP} for $\hat\frF$.
  \end{enumerate}
\end{theorem}
\begin{proof}
  (1)
  It is sufficient to apply Theorem \ref{thm:refdiff}
  and the definition of $\lambda$-dissipativity.

  The left-to-right implication $\Rightarrow$
  of (2) follows by \eqref{eq:76} of Lemma \ref{le:easy}
  and the definition of $\hat \frF$.

  Conversely, if
  $\mu$ satisfies \eqref{eq:156}, $\nu \in \dom(\frF)$, $\Phi\in
  \frF[\nu]$, 
  then Theorem
  \ref{thm:refdiff}
  and the definition of $\hat \frF$ 
  yield 
\begin{align*}
  \frac{1}{2}\frac{\de}{\de t}
  W_2^2(\mu_t, \nu)
  &=
    \bram{(\ii_\X , \vv_t)_{\sharp} \mu_t}\nu
  \le \lambda W_2^2(\mu_t, \nu) - \bram{\Phi}{\mu_t}
  \quad\text{a.e.~in }\interval.
 \end{align*}
 Claim (3) is an immediate consequence of Lemma \ref{le:easy}, Proposition \ref{prop:maximal}(d) and Proposition \ref{prop:scalmscal}.

 Claim (4) is a consequence of Proposition \ref{prop:maximal}(f) and the $\lambda$-dissipativity of $\hat \frF$.
\end{proof}

\begin{proposition} \label{prop:gfvsevi}Let $\func: \prob_2(\X) \to (-\infty, + \infty]$ be a proper, lower semicontinuous and $\lambda$-convex functional and let $\mu \in \rmC(\interval; \dom(\boldsymbol{\partial}\func))$ be a locally absolutely continuous curve. Then
\begin{enumerate}
    \item if $\mu$ is a Gradient Flow for $\func$ i.e.
    \[ (\ii_\X, \vv_t)_{\sharp}\mu_t \in -\boldsymbol{\partial} \func(\mu_t) \quad \text{ a.e. } t \in \interval,\]
    then $\mu$ is a $(-\lambda)$-\EVI solution of \eqref{eq:CP} for the \MPVF $ -\boldsymbol{\partial} \func$ as in \eqref{eqdef:Vphi};
    \item if $\mu$ is a $(-\lambda)$-\EVI solution of \eqref{eq:CP} for the \MPVF $ -\boldsymbol{\partial} \func$ and the domain of $\boldsymbol{\partial} \func$ satisfies
    \begin{equation*}
          \text{ for a.e. } t \in \interval, \,  \CondGammao{\boldsymbol{\partial} \func}{\mu_t}\nu0 \ne \emptyset \quad \forall \nu \in \dom(\boldsymbol\partial \func),
    \end{equation*}
    then $\mu$ is a Gradient Flow for $\func$.
\end{enumerate}
\end{proposition}
\begin{proof}
The first assertion is a consequence Theorem \ref{theo:equivwssol}(1). We prove the second claim; by \eqref{eq:87b} we have that for a.e.~$t \in \interval$ it holds
\[ \bram{(\ii_\X ,\vv_t)_{\sharp}\mu_t}{\nu} \le \directionalm{(\ii_\X ,\vv_t)_{\sharp}\mu_t}{\mmu}0 \le \directional{-\boldsymbol{\partial} \func}{\mmu}{0+} \quad \forall \, \nu \in \dom(\func) \, \, \forall \, \mmu \in \CondGammao{\boldsymbol{\partial} \func}{\mu_t}\nu0. \]
We show that for every $\nu_0, \nu_1 \in \dom(\boldsymbol{\partial} \func)$ and every $\nnu \in \CondGammao\frF{\nu_0}{\nu_1}{0}$
\begin{equation}\label{eq:gradflow} \directional{-\boldsymbol{\partial} \func}{\nnu}{0+} \le \func(\nu_1) - \func(\nu_0) - \frac{\lambda}{2}W_2^2(\nu_0, \nu_1).
\end{equation}
To prove that, we take $s \in \rI\nnu{\boldsymbol{\partial} \func}\cap (0,1)$ and $\Phi_s \in - \boldsymbol{\partial} \func(\nu_s)$. By definition of subdifferential we have
\[ \bram{\Phi_s}{\nu_1} \le \func(\nu_1)-\func(\nu_s)- \frac{\lambda}{2} W_2^2(\nu_s, \nu_1)\]
where $\nu_s = \sfx^{s}_{\sharp}\nnu$. Dividing by $(1-s)$, using \eqref{eq:64} and passing to the infimum w.r.t.~$\Phi_s \in - \boldsymbol{\partial} \func(\nu_s)$ we obtain
\[\directional{-\boldsymbol{\partial} \func}{\nnu}{r,s} \le \frac{1}{1-s}\left (\func(\nu_1) - \func(\nu_s) \right ) - \frac{\lambda(1-s)}{2}W_2^2(\nu_0, \nu_1).\]
Passing to the limit as $s \downarrow 0$ and using the lower semicontinuity of $\func$ lead to the result. Once that \eqref{eq:gradflow} is established we have that for a.e. $t \in \interval$ it holds
\begin{equation} \label{eq:almostdefsubd}  \bram{(\ii_\X ,\vv_t)_{\sharp}\mu_t}{\nu} \le \func(\nu) - \func(\mu_t) - \frac{\lambda}{2}W_2^2(\mu_t, \nu) \text{ for every } \nu \in \dom(\boldsymbol\partial \func).
\end{equation}
To conclude it is enough to use the lower semicontinuity of the LHS (see Lemma \ref{lem:lsc}) and the fact that $\dom(\boldsymbol\partial \func)$ is dense in $\dom(\func)$ in energy: indeed we can apply \cite[Corollary 4.5]{NaldiSavare} and \cite[Lemma 3.1.2]{ags} to the proper, lower semicontinuous and convex functional $\func^{\lambda}: \prob_2(\X) \to (-\infty, + \infty]$ defined as
\[ \func^{\lambda}(\nu)=\func(\nu) - \frac{\lambda}{2}\sqm{\nu}\]
to get the existence, for every $\nu \in \dom(\func)$, of a family $(\nu_\tau)_{\tau>0} \subset \dom(\func^{\lambda}) = \dom(\func)$ s.t.
\[ \nu_{\tau} \to \nu, \quad \func^{\lambda}(\nu_{\tau}) \to \func^{\lambda}(\nu) \quad \text{ as } \tau \downarrow 0.\]
Of course $\func(\nu_{\tau}) \to \func(\nu)$ as $\tau \downarrow 0$ and, applying \cite[Lemma 10.3.4]{ags}, we see that $\nu_{\tau} \in \dom(\boldsymbol\partial \func^{\lambda})$. However $\boldsymbol \partial \func^{\lambda} = L^{\lambda}_{\sharp} \boldsymbol \partial \func$ (see \eqref{eq:rellam}) so that $\nu_\tau \in \dom(\boldsymbol \partial \func)$. We can thus write \eqref{eq:almostdefsubd} for $\nu_{\tau}$ in place of $\nu$ and pass to the limit as $\tau \downarrow 0$, obtaining that, by definition of subdifferential, $(\ii_\X ,\vv_t)_{\sharp}\mu_t \in -\boldsymbol \partial \func (\mu_t)$ for a.e.~$t\in \interval$.
\end{proof}

We derive a further useful a priori bound for $\lambda$-\EVI
solutions.
\begin{proposition}
  Let $\frF$ be a $\lambda$-dissipative \MPVF and let $T \in (0, + \infty]$.
  Every $\lambda$-\wEVI solution $\mu :[0,T) \to
  \overline{\dom(\frF)}$
  with initial datum $\mu_0\in \dom(\frF)$ satisfies the a priori bound
  \begin{equation}
    \label{eq:128bis}
    W_2(\mu_t,\mu_0)\le 2 |\frF|_2(\mu_0) \int_0^t\mathrm e^{\lambda
      s}\,\d s,\quad \text{for all }t \in [0, T),
  \end{equation}
  where 
  \[ |\frF|_2(\mu):={}\inf\Big\{|\Phi|_2:\Phi\in \frF[\mu]\Big\}\quad \text{for every }\mu\in \dom(\frF).\]
\end{proposition}
\begin{proof}
  Let $\Phi\in \frF(\mu_0)$. \eqref{eq:EVI} with $\nu:=\mu_0$ then
  yields
  \begin{equation*}
    \updt W_2^2(\mu_t,\mu_0)-2\lambda W_2^2(\mu_t,\mu_0)\le
    -2\bram\Phi{\mu_t}\le 2|\Phi|_2\,W_2(\mu_t,\mu_0), \quad \text{for every }t\in [0,T).
  \end{equation*}
  We can then apply the estimate of \cite[Lemma 4.1.8]{ags} to obtain
  \begin{equation*}
    \mathrm e^{-\lambda t}
    W_2(\mu_t,\mu_0)\le 2|\Phi|_2\int_0^t \mathrm e^{-\lambda s}\,\d s,\quad \text{for all }t\in [0,T),
  \end{equation*}
  which in turn yields \eqref{eq:128bis}.
\end{proof}

We conclude this section with a result showing the robustness of the notion of $\lambda$-\EVI solution.
\begin{proposition}
  \label{prop:stability}
  If $\mu_n: \interval\to \overline{\dom(\frF)}$ is a sequence of
  $\lambda$-\EVI solutions locally uniformly converging to $\mu$ as $n\to\infty$,
  then $\mu$ is a $\lambda$-\EVI solution.
\end{proposition}
\begin{proof}
  $\mu$ is a continuous curve defined in $\interval$ with values in $\overline{\dom(\frF)}$.
  Using pointwise convergence, the lower semicontinuity of
  $\mu\mapsto \bram\Phi\mu$ of Lemma \ref{lem:lsc}, and Fatou's Lemma,
  it is easy to pass to the limit in the equivalent characterization \eqref{eq:140}
  of $\lambda$-\EVI solutions, written for $\mu_n$.
\end{proof}
\subsection{Local existence of \texorpdfstring{$\lambda$}{l}-\EVI solutions by the Explicit Euler Scheme}

In order to prove the existence of a $\lambda$-\EVI solution to \eqref{eq:CP}, our strategy is to employ an approximation argument through an Explicit Euler scheme as it occurs for ODEs. \\
In the following $\floor{\cdot}$ and $\ceil{\cdot}$ denote the floor and the ceiling functions respectively.

\begin{definition}[Explicit Euler Scheme]\label{def:EEscheme}
  Let $\frF$ be a \MPVF and suppose we are given
  a step size $\tau>0$, an initial datum $\mu^0_\tau\in\dom(\frF)$,
  a bounded interval $[0,T]$, corresponding to
  the final step $\finalstep T\tau:=\ceil{T/\tau},$
  and a stability bound $L>0$.
  A sequence
  $(M^n_\tau,\fF_\tau^n)_{0\le n\le \finalstep T\tau}\subset \dom(\frF)\times \frF $ is a
  \emph{$L$-stable solution to the Explicit Euler Scheme in $[0,T]$ starting from $\mu^0_\tau\in \dom(\frF)$} if 
\begin{equation}
  \label{eq:EE}
  \tag{EE}
  \left\{
\begin{aligned}
  M_{\tau}^0& = \mu^0_\tau ,\\
  \fF_{\tau}^n &\in \frF[M_{\tau}^n],\
  |\fF_\tau^n|_2 \le L
  &&0\le n<\finalstep T\tau,\\
  M_{\tau}^{n} &= (\exp^{\tau})_{\sharp} \fF_{\tau}^{n-1} 
  &&1\le n\le \finalstep T\tau.
\end{aligned}
\right.
\end{equation}

We define the following two different interpolations of the sequence $(M^n_\tau,\fF_\tau^n)$:
\begin{itemize}
\item the affine interpolation:
\begin{equation}\label{eq:affineM}
  M_{\tau}(t) := (\exp^{t-n\tau})_{\sharp} \fF_{\tau}^n \text{\quad if } t \in [n\tau, (n+1)\tau]
  \text{ for some } n \in \N,\ 0\le n<\finalstep T\tau,
\end{equation}
\item the piecewise constant interpolation:
\begin{equation*}
\begin{split}
\bar{M}_{\tau}(t) &:= M^{\floor{t/\tau}}_{\tau}, \quad t\in [0,T], \\
\fF_{\tau}(t) &:= \frF^{\floor{t/\tau}}_{\tau}, \quad t\in [0,T].
\end{split}
\end{equation*}
\end{itemize}
We will call $\mathscr E(\mu^0_\tau,\tau,T,L)$ (resp.~$\mathscr M(\mu^0_\tau,\tau,T,L)$)
the (possibly empty) 
set of all the curves $(M_\tau,\fF_\tau)$ (resp.~$M_\tau$) arising from the
solution of \eqref{eq:EE}.
\end{definition}
The affine interpolation can be trivially written as
\[M_\tau(t)=\left(\exp^{t-\floor{t/\tau}\tau}\right)_{\sharp}\left(\fF_\tau(t)\right),\quad
  t\in [0,T],\]
and $M_\tau$ satisfies the uniform Lipschitz bound
\begin{equation}
  \label{eq:80}
  W_2(M_\tau(t),M_\tau(s))\le L|t-s|\quad
  0\le s\le t\le T,\quad
  M_\tau\in \mathscr E(\mu_0,\tau,T,L).
\end{equation}
Notice that, since in general $\frF[\mu]$ is not reduced to a singleton,
the sets $\mathscr E(\mu_0,\tau,T,L)$ and $\mathscr M(\mu_0,\tau,T,L)$
may contain more than one element (or may be empty).
Stable solutions to the Explicit Euler scheme generated by a $\lambda$-dissipative
\MPVF exhibit a nice behaviour, which is clarified by the following important result,
which will be proved in Section \ref{sec:EulerScheme} (see Proposition
\ref{prop:samestep} and Theorems \ref{theo:convergence},
\ref{prop:rate}, \ref{theo:strong-solution}), with a more accurate
estimate of the error constants $A(\vartheta)$.
We stress that in the next statement
$A(\vartheta)$ solely depend on $\vartheta$
(in particular, it is independent of $\lambda, L, T,\tau,\eta,
M_\tau,M_\eta$).
\begin{theorem} \label{thm:apriori-estimate} Let $\frF$ be a
  $\lambda$-dissipative \MPVF.
  \begin{enumerate}
  \item For every 
    $\mu_0,\mu_0' \in \dom(\frF)$,
    every 
    $M_\tau\in \mathscr M(\mu_0,\tau,T,L)$, $M_\tau'\in \mathscr M(\mu_0',\tau,T,L)$
    with $\tau\lambda_+\le 2$ we have
    \begin{equation}
      \label{eq:71}
      W_2(M_\tau(t),M_\tau'(t))
      \le \mathrm e^{\lambda t}W_2(\mu_0,\mu_0') +8L\sqrt {t\tau}\Big(1+|\lambda|\sqrt{t\tau}\Big)\mathrm e^{\lambda_+ t}
        \quad\text{for every }t\in [0,T].
    \end{equation}
      \item For every $\vartheta>1$
  there exists a constant $A(\vartheta)$ such that
  if $M_\tau\in \mathscr M(M^0_\tau,\tau,T,L)$
  and $M_\eta\in \mathscr M(M^0_\eta,\eta,T,L)$
  with $\lambda_+(\tau+\eta)\le 1$ then 
  \begin{equation*}
    W_2(M_\tau(t),M_\eta(t))\le
    \Big(\vartheta W_2(M^0_\tau,M^0_\eta)+
    A(\vartheta) L\sqrt{(\tau+\eta)(t+\tau+\eta)}\Big)\mathrm
    e^{\lambda_+\, t},\quad
    t\in [0,T].
  \end{equation*}
\item
  For every $\vartheta>1$
  there exists a constant $A(\vartheta)$ 
  such that
  if $\mu\in \rmC([0,T];\overline{\dom(\frF)})$ is a $\lambda$-\EVI solution
  and $M_\tau\in \mathscr M(M^0_\tau,\tau,T,L)$
  then 
  \begin{equation}
        \label{eq:139}
  W_2(\mu(t), M_{\tau}(t))\le
  \Big(\vartheta\, W_2(\mu_0,M^0_\tau)+
  A(\vartheta)L\sqrt{\tau(t+\tau)}\Big)
  \mathrm e^{\lambda_+ t}
  \quad
  \text{for every }t\in [0,T].
\end{equation}
\item If $n\mapsto \tau(n)$ is a vanishing sequence of
  time steps, $(\mu_{0,n})_{n\in \N}$ is a sequence in $\dom(\frF)$ converging to
  $\mu_0\in \overline{\dom(\frF)}$ in $\prob_2(\X)$ and
  $M_n\in \mathscr M(\mu_{0,n},\tau(n),T,L)$,
  then $M_n$ is uniformly converging to a limit curve
  $
  \mu\in \Lip([0,T];\overline{\dom(\frF)})$ which
  is a $\lambda$-\EVI solution
  starting from $\mu_0$.
\end{enumerate}
\end{theorem}
If we assume that the Explicit Euler scheme is locally solvable,
Theorem \ref{thm:apriori-estimate}
provides a crucial tool to obtain local existence and uniqueness of $\lambda$-\EVI solutions.
\begin{definition}[Local and global solvability of \eqref{eq:EE}]
  We say that the Explicit Euler Scheme \eqref{eq:EE} associated to a \MPVF $\frF$
  is \emph{locally solvable}
  at $\mu_0\in \dom(\frF)$ if there exist strictly positive constants
  $\ttau,T,L$ such that
  $\mathscr E(\mu_0,\tau,T,L)$ is not empty for every $\tau\in (0,\ttau)$.\\
  We say that \eqref{eq:EE} is \emph{globally solvable} at 
  $\mu_0\in \dom(\frF)$ if for every $T>0$ there exist strictly positive constants
  $\ttau,L$ such that
  $\mathscr E(\mu_0,\tau,T,L)$ is not empty for every $\tau\in (0,\ttau)$.
\end{definition}
Let us now state the main existence result for
$\lambda$-\EVI solutions. Given $T \in (0, + \infty]$ and $\mu:[0,T)\to\prob_2(\X)$ we
denote by $|\dot\mu|_+(t)$ the right upper metric derivative
\begin{equation*}
  |\dot\mu|_+(t):=
  \limsup_{h\downarrow0}\frac{W_2(\mu_{t+h},\mu_t)}h.
\end{equation*}
\begin{theorem}[Local existence and uniqueness]
  \label{thm:existence}
  Let $\frF$ be a $\lambda$-dissipative \MPVF.
  \begin{enumerate}[label=\rm(\alph*)]
  \item If
    the Explicit Euler Scheme is locally solvable at
    $\mu_0\in \dom(\frF)$, then there exists $T>0$
    and a unique $\lambda$-\EVI solution $\mu\in \Lip([0,T];\overline{\dom(\frF)})$ starting from $\mu_0$,
    satisfying
    \begin{equation}
      \label{eq:88}
      t\mapsto \mathrm e^{-\lambda t}|\dot \mu|_+(t)
      \quad\text{is decreasing in }[0,T).
    \end{equation}
  If $\mu':[0,T']\to \overline{\dom(\frF)}$ is any other $\lambda$-\EVI solution starting from $\mu_0$
  then $\mu(t)=\mu'(t)$ if $0\le t\le T\land T'$.
\item If
  the Explicit Euler Scheme is locally solvable in $\dom(\frF)$ and 
    \begin{equation}
      \label{eq:86}
      \begin{gathered}
        \text{for any local $\lambda$-\EVI solution $\mu$ starting from $\mu_0\in \dom(\frF)$}\\
        \text{there exists }\delta>0:\quad
        t\in [0,\delta]\quad\Rightarrow\quad\mu(t)\in \dom(\frF),
    \end{gathered}
    \end{equation}
    then for every $\mu_0 \in \dom(\frF)$ there exist a unique
    maximal time $T\in (0,\infty]$ and a
    unique strict $\lambda$-\EVI solution $\mu\in \Lip_{\rm loc}([0,T);\dom(\frF))$ starting from $\mu_0$,
    which satisfies \eqref{eq:88} and
    \begin{equation}
      \label{eq:79}
      T<\infty\quad\Rightarrow\quad
      \lim_{t\uparrow T}\mu_t\not\in\dom(\frF).
    \end{equation}
    Any other $\lambda$-\EVI solution $\mu':[0,T')\to \overline{\dom(\frF)}$
    starting from $\mu_0$ coincides with $\mu$
    in $[0,T\land T')$.
  \end{enumerate}
\end{theorem}
\begin{proof}
  (a)
  Let $\ttau, T,L$ positive constants such that
  $\mathscr E(\mu_0,\tau,T,L)$ is not empty for every $\tau\in (0,\ttau)$.
  Thanks to Theorem \ref{thm:apriori-estimate}(2),
  the family $M_\tau\in \mathscr E(\mu_0,\tau,T,L)$
  satisfies the Cauchy condition in $\mathrm C([0,T];\prob_2(\X))$
  so that there exists a unique limit curve $\mu=\lim_{\tau\downarrow0}M_\tau$
  which is also Lipschitz in time, thanks to the a-priori bound
  \eqref{eq:80}.
  Theorem \ref{thm:apriori-estimate}(4) shows that $\mu$ is a $\lambda$-\EVI solution
  starting from $\mu_0$ and
  the estimate \eqref{eq:139} of Theorem \ref{thm:apriori-estimate}(3)
  shows that any other $\lambda$-\EVI solution 
  in an interval $[0,T']$ starting from $\mu_0$ should coincide with
  $\mu$ in $[0,T'\land T]$.

  Let us now check \eqref{eq:88}: we fix $s,t$ such that $0\le s<t<T$
  and $h\in (0,T-t)$,
  and we set $s_\tau:=\tau\floor{s/\tau}$, $h_\tau:=\tau \floor{h/\tau}$.
  The curves $r\mapsto M_\tau(s_\tau+r)$,
  $r\mapsto M_\tau(s_\tau+h_\tau+r)$ belong
  to $\mathscr M(M_\tau(s_\tau),\tau,t-s,L)$ and
  $\mathscr M(M_\tau(s_\tau+h_\tau),\tau,t-s,L)$, so that \eqref{eq:71} yields
  \begin{equation*}
    W_2(M_\tau(s_\tau+t-s),M_\tau(s_\tau+h_\tau+(t-s)))\le \mathrm e^{\lambda (t-s)}
    W_2(M_\tau(s_\tau),M_\tau(s_\tau+h_\tau))+B\sqrt\tau,
  \end{equation*}
  for $B=B(\lambda, L, \ttau,T)$.
  Passing to the limit as $\tau\downarrow0$ we get
  \begin{equation*}
    W_2(\mu(t),\mu(t+h))\le \mathrm e^{\lambda (t-s)}
    W_2(\mu(s),\mu(s+h)).
  \end{equation*}
  Dividing by $h$ and passing to the limit as $h\downarrow0$ we get
  \eqref{eq:88}.
  
  \medskip\noindent
  (b)
  Let us call $\mathcal S$ the collection of
  $\lambda$-\EVI solutions $\mu:[0,S)\to \dom(\frF)$
  starting from $\mu_0$ with values in $\dom(\frF)$ and defined in some interval
  $[0,S)$, $S=S(\mu)$. Thanks to \eqref{eq:86} and the previous claim
  the set $\mathcal S$ is not empty.

  It is also easy to check that two curves $\mu',\mu''\in \mathcal S$
  coincide
  in the common domain $[0,S)$ with $S:=S(\mu')\land S(\mu'')$:
  in fact the set $\{t\in [0,S):\mu'(r)=\mu''(r) \text{ if }0\le r\le t\}$ contains $t=0$,
  is closed since $\mu',\mu''$ are continuous, and it is also open
  since if $\mu'=\mu''$ in $[0,t]$ then
  the previous claim and the fact that $\mu'(t)=\mu''(t)\in \dom(\frF)$
  show that $\mu'=\mu''$ also in a right neighborhood of $t$.
  Since $[0,S)$ is connected, we conclude that $\mu'=\mu''$ in $[0,S)$.

  We can thus define $T:=\sup\{S(\mu):\mu\in \mathcal S\}$
  obtaining that there exists a unique $\lambda$-\EVI solution
  $\mu$ starting from $\mu_0$ and defined in $[0,T)$
  with values in $\dom(\frF)$.
  
  If $T<\infty$, since $\mu$ is Lipschitz in $[0,T)$ thanks to \eqref{eq:88},
  we know that
  there exists the limit $\bar\mu:=\lim_{t\uparrow T}\mu(t)$ in $\prob_2(\X)$.
  If $\bar \mu\in \dom(\frF)$ we can extend $\mu$ to a $\lambda$-\EVI solution
  with values in $\dom(\frF)$ and defined in an interval $[0,T')$ with $T'>T$,
  which contradicts the maximality of $T$.
\end{proof}

Recall that
a set $A$ in a metric space $X$ is locally closed
if every point of $A$ has a neighborhood $U$ such that
$A\cap U = \bar{A} \cap U$. Equivalently, $A$ is the intersection of an open and a closed subset of $X$.
In particular, open or closed sets are locally closed.
\begin{corollary}
  \label{cor:uniqueness}
  Let $\frF$ be a $\lambda$-dissipative \MPVF for which the Explicit Euler Scheme is locally solvable
  in $\dom(\frF)$.
  If $\dom(\frF)$ is locally closed
    then for every $\mu_0\in \dom(\frF)$ there exists a unique
    maximal strict $\lambda$-\EVI solution $\mu\in \Lip_{\rm loc}([0,T);\dom(\frF))$, $T \in (0, + \infty]$, satisfying \eqref{eq:79}.
\end{corollary}

Let us briefly discuss
the question of local solvability of the Explicit Euler scheme.
The main constraints of the Explicit Euler construction relies on the
a priori stability bound and in the
condition $M_\tau^n\in \dom(\frF)$ for every step $0\le n\le \finalstep T\tau$.
This constraint is feasible if at each measure $M^n_\tau$, $0\le n<\finalstep T\tau$, the set
$\Adm_{\tau,L}(M^n_\tau)$ defined by
\begin{equation*}
  \Adm_{\tau,L}(\mu):=\Big\{\Phi\in \frF[\mu]: |\Phi|_2 \le L
  \quad
  \text{and}\quad \exp_{\sharp}^{\tau} \Phi \in \dom(\frF)
  \Big\}
\end{equation*}
is not empty.
If $\dom(\frF)$ is open and $\frF$ is locally bounded, then it is easy to check that
the Explicit Euler scheme is locally solvable (see Lemma \ref{le:open-easy}).
We will adopt the following notation:
\begin{align}
  \label{eq:59}
  |\frF|_2(\mu):={}&\inf\Big\{|\Phi|_2:\Phi\in \frF[\mu]\Big\}\quad
                     \text{for every }\mu\in \dom(\frF),
\end{align}
and we will also introduce
the upper semicontinuous envelope $  |\frF|_{2\star}$ of
the function $|\frF|_2$:
i.e.
\begin{equation*}
  \begin{aligned}
    |\frF|_{2\star}(\mu):={}&\inf_{\delta>0}\sup\Big\{|\frF|_2(\nu):
    \nu\in \dom(\frF),\ W_2(\nu,\mu)\le
    \delta\Big\}
    \\={}&\sup
    \Big\{\limsup_{k\to\infty}|\frF|_2(\mu_k):\mu_k\in \dom(\frF),\ \mu_k\to \mu\text{ in
    }\prob_2(\X)\Big\}.
  \end{aligned}
\end{equation*}

\begin{lemma}
  \label{le:open-easy}
  If $\frF$ is a $\lambda$-dissipative \MPVF, $\mu_0\in \mathrm
  {Int}(\dom(\frF))$
  and $\frF$ is bounded in a neighborhood of $\mu_0$,
  i.e.
  there exists $\varrho>0$ such that $|\frF|_2$ is bounded in $\rB{\mu_0}\varrho$, then
  the Explicit Euler scheme is locally solvable at $\mu_0$
  and the locally Lipschitz solution $\mu$ given by Theorem
  \ref{thm:existence}(a) satisfies
  \begin{equation}
    |\dot\mu|_+(t)
    \le e^{\lambda t} |\frF|_{2\star}(\mu_0)
  \quad \forall \, t \in [0,T).\label{eq:103}
\end{equation}
In particular, if $\dom(\frF)$ is open and
$\frF$ is locally bounded,
for every $\mu_0\in \dom(\frF)$ there exists a unique
maximal $\lambda$-\EVI solution $\mu\in \Lip_{\rm loc}([0,T);\prob_2(\X))$
satisfying \eqref{eq:79} and \eqref{eq:103}.
\end{lemma}
\begin{proof}
  Let $\mu_0\in \mathrm{Int}(\dom(\frF))$ and let $\varrho, L>0$ so that
  $|\frF|_2(\mu)<L$ for every $\mu\in \rB{\mu_0}\varrho$.
  We set 
  $T:=\varrho/(2L)$, $\ttau=T\land 1$ and we perform 
  a simple induction argument to prove that
  $W_2(M^n_\tau,\mu_0)\le L n\tau<\varrho$ if $n\le \finalstep T\tau$
  so that we can always find an element $\fF^n_\tau\in \frF_{\tau,L}$.
  In fact, if $W_2(M^n_\tau,\mu_0)<Ln\tau$ and $n<\finalstep T\tau$
  then
  $W_2(M^{n+1}_\tau,\mu_0)\le W_2(M^{n+1}_\tau,M^n_\tau)+ W_2(M^n_\tau,\mu_0)
  \le L(n+1)\tau$.
  \eqref{eq:88} shows that
  $|\dot\mu_t|_+\le L\mathrm e^{\lambda t}$
  for every $L>|\frF|_{2\star}(\mu_0)$, so that we obtain \eqref{eq:103}.
\end{proof}
More refined estimates will be discussed in the next sections.
Here we will show another example, tailored to the case of measures
with bounded support.
\begin{proposition}
  \label{prop:local-bounded}
  Let $\frF$ be a $\lambda$-dissipative \MPVF
  such that $\dom(\frF)\subset \prob_{\rm b}(\X)$ and
  for every $\mu_0\in \dom(\frF)$ there exist $\varrho>0$, $L>0$
  such that
  for every $\mu\in \prob_{\rm b}(\X)$
  \begin{equation*}
    \supp(\mu)\subset \supp(\mu_0)+\mathrm B_\X(\varrho)\quad\Rightarrow\quad
    \exists \Phi\in \frF[\mu]: \supp(\sfv_\sharp\Phi)\subset \mathrm B_\X(L).
  \end{equation*}
  Then for every $\mu_0\in \dom(\frF)$ there exists $T\in (0,+\infty]$ and a unique
  maximal strict $\lambda$-\EVI solution $\mu\in \Lip_{\rm{loc}}([0,T);\dom(\frF))$
  satisfying \eqref{eq:79}.
\end{proposition}
\begin{proof}
  Arguing as in the proof of Lemma \ref{le:open-easy},
    it is easy to check that
  setting $T:=\varrho/4L$, $\boldsymbol{\tau} = T \wedge 1$
  we can find a discrete solution $(M_\tau,\fF_\tau)\in \mathscr E(\mu_0,\tau,T,L)$
  satisfying the more restrictive condition
  $\supp(M^n_\tau)\subset \supp(\mu_0)+\mathrm B_\X(Ln\tau)\subset
  \supp(\mu_0)+\mathrm B_\X(\varrho/2)$, $\supp(\sfv_\sharp \fF^n_\tau)\subset \mathrm B_\X(L)$
  so that
  the Explicit Euler scheme is locally solvable and $M_\tau$ satisfies the uniform bound
    \begin{equation}
    \label{eq:87}
    \supp(M_\tau(t))\subset \supp(\mu_0)+\mathrm B_\X(\varrho/2)\quad\text{for every }t\in [0,T].
  \end{equation}
  Theorem \ref{thm:existence} then yields the existence of a local solution, and
  Theorem \ref{thm:apriori-estimate}(3) shows that the local solution
  satisfies the same bound \eqref{eq:87} on the support, so that \eqref{eq:86} holds.
\end{proof}

\subsection{Stability and uniqueness}\label{sec:locstabeuniq}
In the following theorem we prove a stability result for $\lambda$-\EVI
solutions 
of \eqref{eq:CP}, as it occurs in the classical Hilbertian case
scenario. We distinguish three cases: the first one
assumes that the Explicit Euler scheme is locally solvable in
$\dom(\frF)$. 
\begin{theorem}[Uniqueness and Stability]\label{theo:uniqsol}
  Let $\frF$ be a $\lambda$-dissipative \MPVF
  such that the Explicit Euler scheme is locally solvable in $\dom(\frF)$, and 
  let $\mu^1, \mu^2: [0,T) \to \overline{\dom(\frF)}$, $T\in (0, + \infty]$, be
  $\lambda$-\wEVI solutions to \eqref{eq:CP}.
  If 
  $\mu^1$ is strict,
  then
  \begin{equation}
    \label{eq:146}
    W_2(\mu^1_t,\mu^2_t)\le W_2(\mu^1_0,\mu^2_0)\mathrm e^{\lambda_+ \,t}
    \quad
    \text{for every }t\in [0,T).
  \end{equation}
  In particular, if $\mu^1_0=\mu^2_0$ then $\mu^1\equiv\mu^2$ in $[0,T)$.\\
  If $\mu^1,\mu^2$ are both strict, then
  \begin{equation}
    W_2(\mu_t^1, \mu_t^2) \le W_2(\mu^1_0, \mu^2_0)
    \mathrm  e^{\lambda t}
 \quad\text{for all } t \in [0,T).\label{eq:94bis}
\end{equation}
\end{theorem}
\begin{proof}
  In order to prove \eqref{eq:146},
  let us fix $t\in (0,T)$.
  Since the Explicit Euler scheme is locally solvable and $\mu^1_t\in \dom(\frF)$,
  there exist $\boldsymbol \tau,\delta,L$
  such that $\mathscr M(\mu^1_t,\tau,\delta,L)$
  is not empty
  for every $\tau\in (0,\ttau)$.
  If $M^1_\tau\in \mathscr M(\mu^1_t,\tau,\delta,L)$, then \eqref{eq:139} yields
  \begin{equation*}
    \begin{aligned}
      W_2(\mu^1_{t+h},\mu^2_{t+h})
      &\le
      W_2(M^1_{\tau}(h),\mu^2_{t+h})+
      W_2(M^1_{\tau}(h),\mu^1_{t+h})
      \\&
      \le
      \vartheta W_2(\mu^1_t,\mu^2_t)\mathrm e^{\lambda_+ h}+
      B\sqrt\tau\quad
    \text{if }0\le h\le\delta,
    \end{aligned}
  \end{equation*}
  for $B=B(\lambda, L, \ttau,\delta)$
  Passing to the limit as $\tau\downarrow0$
  we obtain
  \begin{equation*}
    W_2(\mu^1_{t+h},\mu^2_{t+h})\le 
    \vartheta
    W_2(\mu^1_t,\mu^2_t) \mathrm e^{\lambda_+ h}
  \end{equation*}
  and a further limit as $\vartheta\downarrow1$ yields
  \begin{equation*}
    W_2(\mu^1_{t+h},\mu^2_{t+h})\le 
    W_2(\mu^1_t,\mu^2_t) \mathrm e^{\lambda_+ h}\quad\text{for every }h\in [0,\delta],
  \end{equation*}
  which implies that the map $t\mapsto \mathrm e^{-\lambda_+ t}W_2(\mu^1_t,\mu^2_t)$ is decreasing in $[t,t+\delta]$.
  Since $t$ is arbitrary, we obtain \eqref{eq:146}.

  \medskip\noindent  
  In order to prove the estimate \eqref{eq:94bis}
  (which is better than \eqref{eq:146} when $\lambda<0$),
  we argue in a similar way, using \eqref{eq:71}.

  As before, for a given $t\in (0,T)$, 
  since the Explicit Euler scheme is locally solvable and $\mu^1_t,\mu^2_t\in \dom(\frF)$,
  there exist $\boldsymbol \tau,\delta,L$
  such that $\mathscr M(\mu^1_t,\tau,\delta,L)$
  and
  $\mathscr M(\mu^2_t,\tau,\delta,L)$ are not empty
  for every $\tau\in (0,\ttau)$.
  If $M^i_\tau\in \mathscr M(\mu^i_t,\tau,\delta,L)$, for $i=1,2$, \eqref{eq:71} and \eqref{eq:139} then yield
  \begin{equation*}
    \begin{aligned}
      W_2(\mu^1_{t+h},\mu^2_{t+h})
      &\le
      W_2(\mu^1_{t+h},M^1_\tau(h))+
      W_2(M^1_{\tau}(h),M^2_\tau(h))+
      W_2(\mu^2_{t+h},M^2_\tau(h))
      \\&
      \le \mathrm e^{\lambda h}
    W_2(\mu^1_t,\mu^2_t)+B\sqrt \tau\quad
    \text{if }0\le h \le \delta,
    \end{aligned}
  \end{equation*}
  for $B=B(\lambda, L, \ttau,\delta)$.
  Passing to the limit as $\tau\downarrow0$
  we obtain
  \begin{equation*}
    W_2(\mu^1_{t+h},\mu^2_{t+h})\le \mathrm e^{\lambda h}
     W_2(\mu^1_t,\mu^2_t)
   \end{equation*}
   which implies that the map $t\mapsto \mathrm e^{-\lambda t}W_2(\mu^1_t,\mu^2_t)$ is decreasing in $(0,T)$.
 \end{proof}
 It is possible to prove \eqref{eq:94bis}
 by a direct argument depending on the definition of $\lambda$-\EVI solution
 and a geometric condition on $\dom(\frF)$. The simplest situation
 deals with absolutely continuous curves.
 \begin{theorem}[Stability for absolutely continuous solutions]\label{theo:uniqsol2}
  Let $\frF$ be a $\lambda$-dissipative \MPVF
  and let $\mu^1, \mu^2: [0,T) \to \overline{\dom(\frF)}$, $T\in (0, + \infty]$, be
  \emph{locally absolutely continuous}
  $\lambda$-\wEVI solutions to \eqref{eq:CP}.
  If $\CondGammao\frF{\mu^1_t}{\mu^2_t}0 \ne \emptyset$ for a.e. $t \in (0,T)$, then \eqref{eq:94bis} holds.
In particular, if $\mu^1_0=\mu^2_0$ then
$\mu^1\equiv\mu^2$ in $[0,T)$.
\end{theorem}
\begin{proof}
  Since $\mu^1,\mu^2$ are locally absolutely continuous curves, we can apply
  Theorem
  \ref{thm:refdiff2} and find a subset $A\subset A({\mu^1})\cap
  A({\mu^2})$ of full Lebesgue measure such that
  \eqref{eq:78} holds and $\CondGammao\frF{\mu^1_t}{\mu^2_t}0 \ne \emptyset$ for every $t \in A$.
  Selecting $\mmu_t\in \CondGammao\frF{\mu^1_t}{\mu^2_t}0$, we have 
  \begin{equation*}
    \frac12\frac\d{\d t}W_2^2(\mu^1_t,\mu^2_t)
    =
      \int \la \vv_t^1(x_1),x_1-x_2\ra\,\d\mmu_t(x_1,x_2)+
      \int \la \vv_t^2(x_2),x_2-x_1\ra\,\d\mmu_t(x_1,x_2).
    \end{equation*}
    Using
  \eqref{eq:87b}, \eqref{eq:87c}, for every $t\in A$  we get 
  \begin{displaymath}
   \frac12\frac\d{\d t}W_2^2(\mu^1_t,\mu^2_t) = \bram{(\ii_\X ,
        \vv_t)_{\sharp}\mu_t^1}{\mu_t^2}\le
    \directional\frF{\mmu_t}{0+}+
    \lambda W_2^2(\mu^1_t,\mu^2_t)
    +\directional\frF{\mathsf s_\sharp \mmu_t}{1-}=
    \lambda W_2^2(\mu^1_t,\mu^2_t),
  \end{displaymath}
  where we also used the property
\begin{equation*}
  \directional{\frF}{\mathsf s_\sharp\mmu_t}{1-}=-\directional{\frF}{
    \mmu_t}{0+}.
 \qedhere
\end{equation*}
\end{proof}
The last situation deals with comparison between
an absolutely continuous and a merely continuous $\lambda$-EVI
solution.
The argument is technically more involved
and takes inspiration from the proof of \cite[Theorem
1.1]{NochettoSavare}:
we refer to the Introduction of \cite{NochettoSavare} for an
explanation of the heuristic idea.
Since it is also at the core of the discrete estimates of Theorem
\ref{thm:apriori-estimate}, we
present it here in the easier
continuous setting. 
\begin{theorem}[Refined stability]
  \label{theo:duesolcomp}Let $T>0$ and let
  $\mu^1 \in
  \mathrm{AC}([0,T]; \dom(\frF))$ and $\mu^2 \in \rmC([0,T];
  \overline{\dom(\frF)})$ be $\lambda$-\EVI solutions for the
  $\lambda$-dissipative \MPVF $\frF$. If
  at least one of the following properties hold:
  \begin{enumerate}
  \item 
    $\CondGammao\frF{\mu^1_r}{\mu^2_s}{0} \ne \emptyset \text{
      for every } s \in (0,T) \text{ and }r\in [0,T) \setminus N$
    with $N\subset (0,T),\ \mathcal L^1(N)=0$;
    \item
      $\mu^1$ satisfies \eqref{eq:CPW},
  \end{enumerate} then 
\[ W_2(\mu^1_t, \mu^2_t) \le e^{\lambda t} W_2(\mu^1_0, \mu^2_0) \quad \text{ for every } t \in [0,T]. \]
\end{theorem}
\begin{proof} 
We extend $\mu^1$ in $(-\infty, 0)$ with the constant value
$\mu^1_0$,
we denote by $\vv$ the Wasserstein velocity field
associated to $\mu^1$ (and extended to $0$ outside $A(\mu^1)$)
and
we define the functions $w,f,h:(-\infty,T]\times [0,T]\to \R$ by
\[ w(r,s):= W_2(\mu^1_r, \mu^2_s)\]
\[ f(r,s) := \begin{cases} 2|\frF|_2(\mu_0^1) w(0,s) \quad &\text{ if
    } r<0, \\ 0 &\text{ if } r \ge 0, \end{cases}\qquad
  h(r,s):=
\begin{cases} 0\quad &\text{ if
    } r<0, \\ 2\bram{(\ii_\X,\vv_r)_\sharp\mu^1_r}{\mu^2_s} &\text{ if
    } r \ge 0.
  \end{cases}\]

Theorem \ref{thm:refdiff} yields
\begin{align}\label{eq:firstone}
  \frac{\partial}{\partial r} w^2(r,s) =
  h(r,s)
  \quad &\text { in }
          \mathcal{D}'(-\infty,T)
          \text{ for
          every } s \in
          [0,T].
\end{align}
In case (1) holds,
writing \eqref{eq:142} for $\mu^2$ with $\nu=\mu^1_r$
with $r\in (-\infty,T]\setminus N$,
then for every $\mmu_{rs} \in \CondGammao\frF{\mu^1_r}{\mu^2_s}{0}$ we obtain
\begin{align}
\label{eq:secondone}
  \upds  w^2(r,s) \le 2\lambda w^2(r,s)
  -2\directional\frF{\mmu_{rs}}{0+} \quad &\text { for } s\in
                                            (0,T)\text{ and }
                                            r \in (-\infty,T)\setminus
                                            N.
\end{align}
On the other hand \eqref{eq:87b} yields
\begin{equation}
  \label{eq:32}
  \begin{aligned}
    -2\directional\frF{\mmu_{rs}}{0+}
    &\le -2
    \directionalm{(\ii_\X,\vv_r)_\sharp\mu^1_r}{\mmu_{rs}}{0} \le
    -2\bram{(\ii_\X,\vv_r)_\sharp\mu^1_r}{\mu^2_s} \quad\text{for
      every }r\in A(\mu^1)\setminus N,\\
    -2\directional\frF{\mmu_{rs}}{0+}
    &\le 2|\frF|_2(\mu_0^1) w(0,s)=f(r,s) \quad\text{for every }r<0.
  \end{aligned}
\end{equation}
Combining \eqref{eq:secondone} and
\eqref{eq:32} we obtain
\begin{equation*}
  \upds w^2(r,s)\le 2\lambda w^2(r,s)
  +f(r,s)-h(r,s)\quad\text{for }s\in (0,T),\ r\in (-\infty,0]\cup
  A(\mu^1)\setminus N.
\end{equation*}
Since $|h(r,s)|\le 2 |\dot \mu^1_r|\,w(r,s)$,
applying Lemma \ref{lem:distrib} we get
\begin{equation}
  \label{eq:35biss}
  \frac{\partial}{\partial s} w^2(r,s)\le 2\lambda w^2(r,s)
  +f(r,s)-h(r,s)\quad\text{in }\mathcal D'(0,T)\text{ for a.e.~}r\in (-\infty,T].
\end{equation}
\eqref{eq:35biss} can also be deduced in case (2) using \eqref{eq:CPW}.

By multiplying both inequalities \eqref{eq:firstone} and
\eqref{eq:35biss}
by $e^{-2\lambda s}$ we get
\begin{align*}
  \frac{\partial}{\partial r} \Big(e^{-2\lambda s}w^2(r,s)\Big) =
  e^{-2\lambda s}h(r,s)\quad &\text { in } \mathcal{D}'(-\infty,T) \text{ and every } s \in [0,T], \\
\label{eq:secondone}
  \frac{\partial}{\partial s} \Big(e^{-2\lambda s}w^2(r,s) \Big)\le
  e^{-2\lambda s}\big(f(r,s)-h(r,s)\big)
  \quad &\text { in } \mathcal{D}'(0,T) \text{ and a.e.~} r \in (-\infty,T].
\end{align*}
We fix $t\in [0,T]$ and $\eps>0$ and we apply the Divergence theorem
in \cite[Lemma 6.15]{NochettoSavare} on the two-dimensional strip
$Q_{0,t}^{\eps}$ as in Figure \ref{F:strip},
\[ Q_{0,t}^{\eps} := \{ (r,s) \in \R^2 \mid 0\le s \le t \, , \, s-\eps \le r \le s \}, \]
 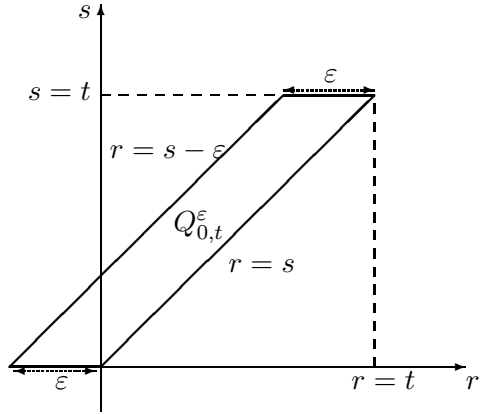
\begin{figure}[htbp]
    \unitlength.6cm
    \centering
 \begin{picture}(10,10)
  \put(1,0){\vector(0,1)9}
  \put(0,1){\vector(1,0)9}
  \multiput(7,1)(0,.4){15}{\line(0,1){.2}}
  \multiput(1,7)(.4,0){15}{\line(1,0){.2}}
  \put(5.2,7.1){\vector(-1,0){.2}}
  \put(6.8,7.1){\vector(1,0){.2}}
  \multiput(5.1,7.1)(.1,0){18}{\line(1,0){.05}}
  \put(5.9,7.3){$\varepsilon$}
  \put(-0.7,0.9){\vector(-1,0){.2}}
  \put(0.7,0.9){\vector(1,0){.2}}
  \multiput(-0.7,0.9)(.1,0){15}{\line(1,0){.05}}
  \put(0,.5){$\varepsilon$}
  \put(6.5,.5){$r=t$}
  \put(-0.6,6.9){$s=t$}
  \thicklines
  \put(1,1){\line(1,1){6}}
  \put(-1,1){\line(1,1){6}}
  \put(-1,1){\line(1,0){2}}
  \put(5,7){\line(1,0)2}
\put(2.6,4){$Q^\varepsilon_{0,t}$}
\put(9,.5){$r$}
\put(.5,8.7){$s$}
\put(3.8,3.1){$r=s$}
\put(1.2,5.6){$r=s-\varepsilon$}
\end{picture}\caption{Strip $Q^\eps_{0,t}$ corresponding to penalization
about the diagonal $\{r=s\}$.}
\label{F:strip}
\end{figure}
and we get
\begin{equation*}
\int_{t-\eps}^t e^{-2\lambda t} w^2(r,t)\de r \le \int_{-\eps}^0 w^2(r,0) \de r + \iint_{Q^{\eps}_{0,t}} e^{-2\lambda s} f(r,s)\de r \de s.
\end{equation*} 
Using 
\[ w(t,t) \le \int_{r}^t |\dot{\mu}^1|(u) \de u+ w(r,t) \le \int_{t-\eps}^t |\dot{\mu}^1|(u) \de u + w(r,t)  \quad \text{ if } t-\eps \le r \le t, \] 
then, for every $\vartheta, \vartheta_{\star}>1$ conjugate coefficients ($\vartheta_{\star}=\vartheta/(\vartheta-1)$), we get
\begin{equation}\label{eq:mix1}
w^2(t,t)\le\vartheta w^2(r,t)+\vartheta_{\star}\left ( \int_{t-\eps}^t |\dot{\mu}^1|(u) \de u \right )^2.
\end{equation}
Integrating \eqref{eq:mix1} w.r.t. $r$ in the interval $(t-\eps,t)$, we obtain
\begin{equation} \label{eq:lhsbdd} 
e^{-2\lambda t} w^2(t,t) \le \frac{\vartheta}{\eps} \int_{t-\eps}^t e^{-2\lambda t} w^2(r,t)\de r + \vartheta_{\star}\left ( \int_{t-\eps}^t |\dot{\mu}^1|(u) \de u \right )^2 \max\{1, e^{2|\lambda|T}\}.
\end{equation}
Finally, we have the following inequality
\begin{equation} \label{eq:fbound} \eps^{-1} \iint_{Q^{\eps}_{0,t}} e^{-2 \lambda s} f(r,s) \de r \de s \le
2|\frF|_2(\mu_0) \int_0^{\eps} e^{-2\lambda s}w(0,s)\de s.
\end{equation}
Summing up \eqref{eq:lhsbdd} and \eqref{eq:fbound} we obtain 

\begin{equation*}
e^{-2\lambda t} w^2(t) \le \vartheta \left ( w^2(0) + 2|\frF|_2(\mu_0) \int_0^{\eps} e^{-2\lambda s}w(0,s)\de s \right) + \vartheta_{\star}\left ( \int_{t-\eps}^t |\dot{\mu}^1|(u) \de u \right )^2 \max\{1, e^{2|\lambda|T}\}.
\end{equation*}
where we have used the notation $w(s)= w(s,s)$. Taking the limit as
$\eps \downarrow 0$ and $\vartheta \downarrow 1$, we obtain the
thesis.
\end{proof} 

\begin{corollary}[Local Lipschitz estimate]
  Let $\frF$ be a $\lambda$-dissipative \MPVF
  and let $\mu :(0,T) \to \overline{\dom(\frF)}$, $T\in (0, + \infty]$, be 
  a $\lambda$-\wEVI solution to \eqref{eq:CP}. 
  If at least one of the following two conditions holds
  \begin{enumerate}[label=\rm(\alph*)]
  \item $\mu$ is strict and \eqref{eq:EE} is locally solvable in $\dom(\frF)$,
  \item $\mu$ is locally absolutely continuous
    and \eqref{eq:157} holds,
  \end{enumerate}
  then $\mu$ is locally Lipschitz and
  \begin{equation}
    \label{eq:100}
    t\mapsto \mathrm e^{-\lambda t} |\dot \mu|_+(t)
    \quad\text{is decreasing in }(0,T).
\end{equation}
\end{corollary}
\begin{proof}
  Since for every $h>0$ the curve
  $t\mapsto \mu_{t+h}$ is a $\lambda$-\EVI solution,
  \eqref{eq:94bis} yields 
  \begin{equation*}
    \mathrm e^{-\lambda (t-s)}W_2(\mu_{t+h},\mu_t)
    \le W_2(\mu_{s+h},\mu_s)
    \quad \text{for every $0<s<t$.}
  \end{equation*}
  Dividing by $h$ and taking the limsup as $h\downarrow0$, we
  get \eqref{eq:100}, which in turn shows the local Lipschitz character
  of $\mu$.  
\end{proof}

\subsection{Global existence and generation of \texorpdfstring{$\lambda$}{l}-flows}\label{sec:globexist}
We collect here a few simple results on the existence of global solutions and the generation of a $\lambda$-flow.
A first result can be deduced from the
global solvability of the Explicit Euler scheme.
\begin{theorem}
  [Global existence]
  \label{thm:global}
  Let $\frF$ be a $\lambda$-dissipative \MPVF.
  If the Explicit Euler Scheme is globally solvable at $\mu_0\in \dom(\frF)$,
  then 
  there exists a unique global $\lambda$-\EVI solution
  $\mu\in \Lip_{\rm loc}([0,\infty);\overline{\dom(\frF)})$ starting from $\mu_0$.
\end{theorem}
\begin{proof}
  We can argue as in the proof of
  Theorem \ref{thm:existence}(a), observing that
  the global solvability of \eqref{eq:EE} allows for
  the construction of a limit solution on every interval $[0,T]$, $T>0$.
\end{proof}
Let us provide a simple condition ensuring global solvability,
whose proof is deferred to Section \ref{sec:EulerScheme}.
\begin{proposition} \label{prop:gsolvcond} 
  Let $\frF$ be a $\lambda$-dissipative \MPVF
  such that for every $R>0$ there exist $M=\mathrm M(R)>0$
  and $\bar\tau=\bar\tau(R)>0$
  such that
  \begin{equation}
    \label{eq:101pre}
    \mu\in \dom(\frF),\
    \rsqm \mu\le R,
    \ 0<\tau\le \bar\tau\quad
    \Rightarrow\quad
    \exists \,\Phi\in \frF[\mu]:
    |\Phi|_2\le \mathrm M(R),\
    \exp^\tau_\sharp\Phi\in \dom(\frF).
  \end{equation}
  Then the Explicit Euler scheme is globally solvable in $\dom(\frF)$.
  \end{proposition}

Global existence of $\lambda$-\EVI solution is also related to
the existence of a $\lambda$-flow.
\begin{definition}
  We say that the $\lambda$-dissipative \MPVF $\frF$ \emph{generates a $\lambda$-flow} if
  for every $\mu_0\in \overline{\dom(\frF)}$ there exists a unique
  $\lambda$-\EVI solution $\mu=\mathrm S[\mu_0]$ starting from $\mu_0$
  and the maps $\mu_0\mapsto \mathrm S_t[\mu_0]=(\mathrm S[\mu_0])_t$
  induce a semigroup of Lipschitz transformations
  $(\rmS_t)_{t\ge0}$ of $\overline{\dom(\frF)}$ satisfying
  \begin{equation}
    \label{eq:143}
    W_2(\mathrm S_t[\mu_0],\mathrm S_t[\mu_1])\le
    \mathrm e^{\lambda t}W_2(\mu_0,\mu_1)\quad
    \text{for every }t\ge0.
  \end{equation}
\end{definition}

\begin{theorem}
  [Generation of a $\lambda$-flow]
  \label{thm:globall}
  Let $\frF$ be a $\lambda$-dissipative \MPVF.
  If at least one of the following properties is satisfied:
  \begin{enumerate}[label=\rm(\alph*)]
  \item
    the Explicit Euler Scheme is globally solvable
    for every $\mu_0$ in a dense subset of $\dom(\frF)$;
  \item the Explicit Euler Scheme is locally solvable
    in $\dom(\frF)$ and,
    for every $\mu_0$ in a dense subset of $\dom(\frF)$,
    there exists a strict global $\lambda$-\EVI solution starting from
    $\mu_0$;
      \item
    the Explicit Euler Scheme is locally solvable
    in $\dom(\frF)$ and $\dom(\frF)$ is closed;
  \item for every $\mu_0\in \dom(\frF)$, $\mu_1 \in \overline{\dom(\frF)}$
    $\CondGammao\frF{\mu_0}{\mu_1}{0}\neq\emptyset$ 
    and,
    for every $\mu_0$ in a dense subset of $\dom(\frF)$,
    there exists a locally absolutely continuous strict global
    $\lambda$-\EVI solution starting from $\mu_0$;
  \item for every 
      $\mu_0$ in a dense subset of $\dom(\frF)$,
      there exists a locally absolutely continuous 
      solution of \eqref{eq:CPW} starting from $\mu_0$,
  \end{enumerate}
  then $\frF$ generates a $\lambda$-flow.
\end{theorem}
\begin{proof}
  (a) Let $D$ be the dense subset of $\dom(\frF)$ for which
  \eqref{eq:EE} is globally solvable.
  For every $\mu_0\in D$ we define
  $\rmS_t[\mu_0]$, $t\ge0$, as the value at time $t$ of the unique
  $\lambda$-\EVI solution starting from $\mu_0$, whose existence
  is guaranteed by Theorem \ref{thm:global}.

  If $\mu_0,\mu_1\in D$, $T>0$, we can find
  $\boldsymbol \tau, L$
  such that $\mathscr M(\mu_0,\tau,T,L)$
  and
  $\mathscr M(\mu_1,\tau,T,L)$ are not empty
  for every $\tau\in (0,\ttau)$.
  We can then pass to the limit in the uniform estimate
  \eqref{eq:71}
  for every choice of $M^i_\tau\in \mathscr M(\mu_i,\tau,T,L)$, $i=0,1$, obtaining
  \eqref{eq:143} for every $\mu_0,\mu_1\in D$.

  We can then extend the map $\rmS_t$ to $\overline D=\overline{\dom(\frF)}$
  still preserving the same property. Proposition \ref{prop:stability}
  shows that for every $\mu_0\in \overline{\dom(\frF)}$
  the continuous curve $t\mapsto \rmS_t[\mu_0]$ is a
  $\lambda$-\EVI solution starting from $\mu_0$.

  Finally, if $\mu\in \rmC([0,T');\overline{\dom(\frF)})$ is
  any $\lambda$-\EVI solution starting from $\mu_0$,
  we can apply \eqref{eq:139} to get
  \begin{equation}
    \label{eq:137}
    W_2(\mu_t,M^1_\tau(t))\le \Big(2W_2(\mu_0,\mu_1)+C(\boldsymbol{\tau}, L,T)\sqrt \tau\Big)\mathrm e^{\lambda_+ t}
    \quad\text{for every }t\in [0,T],
  \end{equation}
  for every $T<T'$, $\tau<\ttau$, where $C(\boldsymbol{\tau}, L,T)>0$ is a suitable constant.
  Passing to the limit as $\tau\downarrow0$ in \eqref{eq:137}
  we obtain
  \begin{equation}
    \label{eq:137bis}
    W_2(\mu_t,\rmS_t[\mu_1])\le 2W_2(\mu_0,\mu_1)\mathrm e^{\lambda_+ t}
    \quad\text{for every }t\in [0,T].
  \end{equation}
  Choosing now a sequence $\mu_{1,n}$ in $D$ converging to $\mu_0$
  and observing that we can choose arbitrary $T<T'$, we eventually get
  $\mu_t=\rmS_t[\mu_0]$ for every $t\in [0,T')$.

  \medskip\noindent
  (b)
  Let $D$ be the dense subset of $\dom(\frF)$ such that
  there exists a global strict $\lambda$-\EVI solution starting from $D$.
  By Theorem \ref{theo:uniqsol} such a solution is unique
  and the corresponding family of solution maps $\rmS_t:D\to \dom(\frF)$
  satisfy \eqref{eq:143}. Arguing as in the previous claim,
  we can extend $\rmS_t$ to $\overline{\dom(\frF)}$ still preserving \eqref{eq:143}
  and the fact that $t\mapsto \rmS_t[\mu_0]$ is a $\lambda$-\EVI solution.

  If $\mu$ is $\lambda$-\EVI solution starting from $\mu_0$,
  Theorem \ref{theo:uniqsol} shows that \eqref{eq:137bis}
  holds for every $\mu_1\in D$. By approximation we conclude
  that $\mu_t=\rmS_t[\mu_0]$.

  \medskip\noindent
  (c)
  Corollary \ref{cor:uniqueness} shows that for every
  initial datum $\mu_0\in \dom(\frF)$
  there exists a global $\lambda$-\EVI solution. We can then apply
  Claim (b).  

  \medskip\noindent
  (d)
  Let $D$ be the dense subset of $\dom(\frF)$ such that
  there exists a locally absolutely continuous strict global $\lambda$-\EVI solution starting from $D$.
  By Theorem \ref{theo:uniqsol2} such a solution is the unique locally absolutely continuous solution starting from $\mu_0$ and the corresponding family of solution maps $\rmS_t:D\to \dom(\frF)$
  satisfy \eqref{eq:143}. Arguing as in the previous claim (b),
  we can extend $\rmS_t$ to $\overline{\dom(\frF)}$ still preserving \eqref{eq:143} (again thanks to Theorem \ref{theo:uniqsol2}) and the fact that $t\mapsto \rmS_t[\mu_0]$ is a $\lambda$-\EVI solution.

  If $\mu$ is a $\lambda$-\EVI solution starting from $\mu_0 \in \overline{\dom(\frF)}$ and $(\mu_0^n)_{n \in \N} \subset D$ is a sequence converging to $\mu_0$, we can apply Theorem \ref{theo:duesolcomp}(1) and conclude
  that $\mu_t=\rmS_t[\mu_0]$.
  
  \medskip\noindent
  (e)
  The proof follows by the same argument of the previous claim,
  eventually applying Theorem \ref{theo:duesolcomp}(2).  
\end{proof}
By Lemma \ref{le:open-easy} we
immediately get the following result.
\begin{corollary}
  If $\frF$ is locally bounded $\lambda$-dissipative \MPVF with
  $\dom(\frF)=\prob_2(\X)$ then for every $\mu_0\in \prob_2(\X)$
  there exists a unique global $\lambda$-\EVI solution
  starting from $\mu_0$.
\end{corollary}
We conclude this section by showing a consistency result
with the Hilbertian theory, related to the example
of Section \ref{subsec:graphB}.
\begin{corollary}[Consistency with the theory of contraction
  semigroups in Hilbert spaces]
  \label{cor:consistency}
    Let $F \subset \X \times \X$ be a
    dissipative maximal subset
    generating the 
  semigroup $(R_t)_{t \ge 0}$ 
  of nonlinear contractions {\upshape \cite[Theorem
  3.1]{BrezisFR}}.
  Let $\frF$ be the dissipative \MPVF
\[ \frF := \{ \Phi \in \prob_2(\TX) \mid \Phi \text{ is concentrated
    on } F \}.\]
The semigroup
$\mu_0\mapsto \mathrm S_t[\mu_0]:=(R_t)_{\sharp}\mu_0$, $t\ge0$,
is the $0$-flow generated by $\frF$ in $\overline{\dom(\frF)}$.
\end{corollary}
\begin{proof}
  Let $D$ be the set of
  discrete measures $\frac 1n\sum_{j=1}^n\delta_{x_j}$
  with $x_j\in \dom(F)$.
  Since every $\mu_0\in \overline{\dom(\frF)}$ is supported
  in $\overline{\dom (F)}$, $D$ is dense in $\overline{\dom(\frF)}$.
  Our thesis follows by applying Theorem \ref{thm:globall}(e)
  if we show that for every
  $\mu_0=\frac1n\sum_{j=1}^n\delta_{x_{j,0}}\in D$
  there exists a locally absolutely continuous solution
  $\mu:[0,\infty)\to D$ of \eqref{eq:CPW} starting from $\mu_0$.

  It can be directly checked that
\begin{displaymath}
  \mu_t:=(R_t)_\sharp \mu_0=
  \frac1n\sum_{j=1}^n\delta_{x_{j,t}},\quad
  x_{j,t}:=R_t(x_{j,0})
\end{displaymath}
satisfies the continuity equation with Wasserstein velocity vector
$\vv_t$ (defined on the finite support of $\mu_t$) satisfying
\begin{displaymath}
  \vv_{t}(x_{j,t})=\dot x_{j,t}=F^\circ (x_{j,t}),\quad
  |\vv_t(x_{j,t})|\le |F^\circ(x_{j,0})|
  \quad\text{for every }j=1,\cdots, n,\ \text{and a.e.~}t>0,
\end{displaymath}
where $F^\circ$ is the minimal selection of $F$. It follows that 
\[ (\ii_\X, \vv_t)_{\sharp}\mu_t \in \frF[\mu_t] \quad \text{ for a.e. } t>0,\]
so that $\mu$ is a Lipschitz \EVI solution for $\frF$ starting from
$\mu_0$. We can thus conclude observing that
the map $\mu_0\mapsto (R_t)_\sharp\mu_0$
are contractions in $\prob_2(\X)$ and
the curve $\mu_t=(R_t)_\sharp\mu_0 $
is continuous with values in $\overline{\dom(\frF)}$.
\end{proof}

\subsection{Barycentric property}\label{sec:bar}
If we assume that the \MPVF $\frF$ is a sequentially closed
subset of $\prob_2^{sw}(\TX)$ with
convex sections, we are able to provide a stronger result showing a particular property satisfied by the solutions of \eqref{eq:CP} (see Theorem \ref{theo:barycentric}).
This is called \emph{barycentric property} and it is strictly connected with the weaker definition of solution discussed in \cites{Piccoli_2019, Piccoli_MDI, Camilli_MDE}.

We first introduce a directional closure of $\frF$ along
smooth cylindrical deformations.
We set
\begin{equation*}
  \expcyl {}\varphi(x):=x+\nabla\varphi(x)\quad
  \text{for every }\varphi\in \Cyl(\X)
\end{equation*}
and
\begin{equation} \label{eq:deform}
  \begin{aligned}
    \pclo\frF[\mu]:={}
    \Big\{\Phi\in \prob_2(\X)\mid{}&\exists\,\varphi\in \Cyl(\X),\ (r_n)_{n\in \N}\subset [0,+\infty),\ r_n\downarrow0,
    \ \Phi_n\in \frF[\expcyl{r_n}\varphi_\sharp\mu]:
    \\&
                        \Phi_n\to\Phi\text{ in }\prob_2^{sw}(\TX)\Big\}.
  \end{aligned}
\end{equation}
\begin{definition}[Barycentric property]\label{def:barprop}
Let $\frF$ be a \MPVF.
We say that a locally absolutely continuous curve
$\mu: \interval \to \dom(\frF)$ satisfies the
\emph{barycentric property} 
(resp.~the \emph{relaxed
  barycentric property})
if for a.e. $t \in \interval$ there exists
$\Phi_t \in \frF[\mu_t]$
(resp.~$\Phi_t \in \cloco{\pclo\frF[\mu_t]}$)
s.t.
\begin{equation} \label{eq:barycentricproperty}
 \frac{\de}{\de t} \int_\X \varphi(x) \de \mu_t (x)= \int_{\TX} \scalprod{\nabla \varphi(x)}{v} \de \Phi_t(x,v) \quad \forall \, \varphi \in \Cyl(\X).
 \end{equation}
\end{definition}
Notice that
$\frF\subset\pclo\frF\subset \clo\frF$ and $\pclo\frF=\frF$
if $\frF$ is sequentially closed in $\prob_2^{sw}(\TX)$. Recalling Proposition \ref{prop:maximal}(a)
we also get
\begin{equation*}
  \cloco{\pclo\frF}\subset \hat\frF,
\end{equation*}
so that the relaxed barycentric property implies
the corresponding property for the extended MPVF $\hat\frF$.
\begin{remark} If $\X = \R^d$, the property stated in Definition \ref{def:barprop} coincides with the weak definition of solution to \eqref{eq:CP} given in \cite{Piccoli_MDI}. 
\end{remark}

The aim is to prove that the $\lambda$-\EVI solution of \eqref{eq:CP} enjoys the barycentric property of Definition \ref{def:barprop}, under suitable mild conditions on $\frF$. This is
strictly related to the behaviour of $\frF$ along the family of smooth deformations
induced by cylindrical functions.
Let us denote by $\proj_\mu$ the orthogonal projection in
$L^2_\mu(\X;\X)$ onto the tangent space $\Tan_\mu\prob_2(\X)$ and by $\bry{\Phi}$ the barycenter of $\Phi$ as in Definition \ref{def:bary}.

\begin{theorem} \label{theo:barycentric} Let $\frF$ be a
  $\lambda$-dissipative \MPVF
  such that for every $\mu\in \dom(\frF)$
  there exist constants
  $M,\eps>0$
  such that 
  \begin{equation}
    \label{eq:138}
      \forall\varphi\in \Cyl(\X):\ \sup_\X|\nabla\varphi|\le \eps
      \quad\Rightarrow\quad
      \expcyl {}\varphi_\sharp\mu\in \dom(\frF),\
      |\frF|_2(\expcyl {}\varphi_\sharp\mu)<M.
  \end{equation}
  If $\mu: \interval \to \dom(\frF)$ is a locally absolutely continuous 
  $\lambda$-\EVI solution of \eqref{eq:CP}
  with Wasserstein velocity
  field
  $\vv$ satisfying \eqref{eq:74}
  for every $t$ in the subset $A(\mu)\subset \interval$ of full
  Lebesgue measure, then 
  \begin{equation}
\text{for every $t \in A(\mu)$ there exists $\Phi_t \in \cloco
  {\pclo \frF}[\mu_t]$ such that}\quad
\vv_t = \proj_{\mu_t} \circ \bry{\Phi_t}.\label{eq:141}
\end{equation}
In particular,
$\mu$ satisfies the relaxed barycentric property.

If moreover $\pclo\frF=\frF$ and for every $\nu\in \dom(\frF)$ $\frF[\nu]$ is a convex subset of $\prob_2(\TX)$, then
$\mu$ satisfies \eqref{eq:barycentricproperty}.
\end{theorem}
\begin{proof}
  In the following $t$ is a fixed element of $A(\mu)$ and $M$ is the constant associated to the measure $\mu_t$ 
in \eqref{eq:138}.
For every 
$\zeta \in \Cyl(\X)$ there exists $\delta=\delta(\zeta)>0$ such that
$\nu^{\zeta}:=\expcyl{}{-\delta\zeta}_\sharp\mu_t
\in \dom(\frF)$ and $\ssigma^\zeta:=(\ii_\X, \expcyl{}{-\delta\zeta})_\sharp
\mu_t\in\CondGammao \frF{\mu_t}{\nu^\zeta}{01}$ is the unique optimal transport plan between $\mu_t$ and $\nu^{\zeta}$.\\
Thanks to Theorem \ref{thm:refdiff}, the map $s \mapsto W_2^2(\mu_s, \nu^{\zeta})$ is differentiable at $s=t$, moreover by employing also \eqref{eq:87b}, it holds
\begin{equation}
\label{eq:startingpoint2}
  \delta\int_{\X} \scalprod{\vv_t(x)}{\nabla \zeta (x)} \de \mu_t(x) = \frac{\de}{\de t} \frac{1}{2}W_2^2(\mu_t, \nu^{\zeta}) \le
 \directional\frF{\ssigma^\zeta}{0+} =
\lim_{s \downarrow 0}\, \directionalp\frF{\ssigma^\zeta} s.
\end{equation}
We can choose a decreasing vanishing sequence $(s_k)_{k \in\N} \subset
(0,1)$, measures $\nu_k^\zeta:=\sfx^{s_k}_\sharp\ssigma^\zeta$
and $\Phi_k^\zeta\in \frF[\nu_k^\zeta]$ such that
$\sup_k|\Phi_k^\zeta|_2\le M$ and $\Phi_k^\zeta\to \Phi^\zeta$
in $\prob_2^{sw}(\TX)$.
Then,
by \eqref{eq:139}, we get $\Phi^{\zeta} \in \pclo\frF[\mu_t]$ with
 $|\Phi^{\zeta}|_2 \le M$
 and by \eqref{eq:startingpoint2} and the upper semicontinuity of $\brap{\cdot}{\cdot}$ (see Lemma \ref{lem:lsc})
we get
\begin{equation}\label{eq:finalestim}
  \delta \int_{\X} \scalprod{\vv_t(x)}{\nabla \zeta (x)} \de \mu_t(x) \le \brap{\Phi^{\zeta}}{\nu^{\zeta}} =
  \delta \int_{\TX}\scalprod{v}{\nabla \zeta(x)} \de \Phi^{\zeta}(x,v).
\end{equation}
Indeed, notice that, by \cite[Lemma 5.3.2]{ags}, we have
$\Lambda(\Phi^{\zeta}, \nu^\zeta)=\{\Phi^{\zeta}\otimes\nu^\zeta\}$
with
$(\sfx^0,\sfx^1)_{\sharp}(\Phi^{\zeta}\otimes\nu^\zeta)=\ssigma^\zeta$.

By means of the identity highlighted in Remark \ref{rmk:barycenter}, the expression in \eqref{eq:finalestim} can be written as follows
\[ \scalprod{\vv_t}{\nabla \zeta}_{L^2_{\mu_t}(\X;\X)} \le \scalprod{\bry{\Phi^{\zeta}}}{\nabla \zeta}_{L^2_{\mu_t}(\X;\X)}
  =
  \scalprod{\proj_{\mu_t}(\bry{\Phi^{\zeta}})}{\nabla \zeta}_{L^2_{\mu_t}(\X;\X)}\]
so that 
\begin{equation*}
  \scalprod{\vv_t}{\nabla \zeta}_{L^2_{\mu_t}(\X;\X)} \le
  \sup_{\boldsymbol b \in K}\,\scalprod{\boldsymbol b}{\nabla \zeta}_{L^2_{\mu_t}(\X;\X)} \quad \text{for all } \, \zeta \in \Cyl(\X)
\end{equation*}
where
\begin{equation} \label{eq:defK}
K := \left\{ \proj_{\mu_t}(\bry{\Phi})\,:\, \Phi \in \pclo\frF[\mu_t], \, |\Phi|_2 \le M\right\}\subset \Tan_{\mu_t} \prob_2(\X).
\end{equation}
Applying Lemma \ref{lem:abstract} in
$\Tan_{\mu_t} \prob_2(\X)\subset L^2_{\mu_t}(\X;\X)$
we  obtain that $\vv_t \in \cloco K$.
In order to obtain \eqref{eq:141} it is sufficient to prove that
$ \vv_t$ is the $L^2$-projection of the barycenter of an element of
$\cloco{\pclo\frF[\mu_t]}$.

Notice that an element $\vv\in \Tan_\mu \prob_2(\X)$
coincides with $\proj_{\mu}(\bry{\Phi})$ for $\Phi\in \relcP 2{\mu}\TX$ if and only
if
\begin{equation}
\label{eq:97}
\int \la \vv,\nabla\zeta\ra\,\d\mu=
\int \la v,\nabla\zeta\ra\,\d\Phi(x,v)\quad\text{for every }\zeta\in \mathrm{Cyl}(\X).
\end{equation}
It is easy to check that
any element $\vv\in \conv K$ can be represented as
$\proj_{\mu_t}(\bry{\Phi}) $ (and thus as in \eqref{eq:97}) for some
$\Phi\in \conv{\pclo\frF[\mu_t]}$.
If $\vv\in \cloco K$ we can find
a sequence $\Phi_n\in \conv{\pclo\frF[\mu_t]}$ such that
$|\Phi_n|_2\le M$ and
$\vv_n=\proj_{\mu_t}(\bry{\Phi_n})\to \vv$ in $L^2_{\mu_t}(\X;\X)$.
Since the sequence $(\Phi_n)_{n \in\N}$ is relatively compact in
$\prob_2^{sw}(\TX)$ by Proposition \ref{prop:finalmente}(2),
we can extract a (not relabeled) 
subsequence converging to
a limit $\Phi$ in $\prob_2^{sw}(\TX)$, as $n\to+\infty$.
By definition $\Phi\in \cloco{\pclo{\frF}[\mu_t]}$ with 
$|\Phi|_2\le M$.
We can eventually pass to the limit in \eqref{eq:97}
written for $\vv_n$ and $\Phi_n$ 
thanks to $\prob_2^{sw}(\TX)$ convergence, obtaining
the corresponding identity for $\vv$ and $\Phi$ in the limit.
\\Finally, being $\mu$ locally absolutely continuous, it satisfies the continuity equation driven by $\vv$ in the sense of distributions (see Theorem \ref{thm:tangentv}), so that
\[ \frac{\de}{\de t} \int_\X \zeta(x) \de \mu_t(x) = \int_{\X} \scalprod{ \nabla \zeta(x)}{\vv_t(x)} \de \mu_t(x) = \int_{\TX} \scalprod{\nabla \zeta (x)}{v} \de \Phi_t(x,v) \quad\forall\zeta \in \text{Cyl}(\X).\qedhere
\]
\end{proof}

\begin{remark}
  We notice that it is always possible to
  estimate the 
  value of $M$ in \eqref{eq:defK} by $  |\frF|_{2\star}(\mu_t)$.
\end{remark}

\begin{remark}
Using a standard approximation argument (see for example the proof of Lemma 5.1.12(f) in \cite{ags}) it is possible to show that actually the barycentric property \eqref{eq:barycentricproperty} holds for every $\varphi \in \rmC^{1,1}(\X; \R)$ (indeed, in this case, $\nabla \varphi \in \Tan_{\mu} \prob_2(\X)$ for every $\mu \in \prob_2(\X)$).
\end{remark}

\medskip
As a complement to the studies investigated in this section, we prove
the converse characterization of Theorem \ref{theo:barycentric} in the
particular case of \emph{regular measures}
or \emph{regular vector fields}.
We refer to \cite[Definitions 6.2.1, 6.2.2]{ags} for the definition of
$\prob_2^r(\X)$, that is the space of regular measures on $\X$.
When $\X=\R^d$ has finite dimension, $\prob_2^r(\X)$
is just the subset of measures in $\prob_2(\X)$ which are absolutely
continuous
w.r.t.~the Lebesgue measure $\mathcal L^d$.

\begin{theorem}\label{thm:RegularM}
  Let $\frF$ be a $\lambda$-dissipative \MPVF. Let $\mu: \interval \to \dom(\frF)$
  be a locally absolutely continuous curve satisfying the relaxed
  barycentric property of Definition \ref{def:barprop}.
  If for a.e.~$t\in \interval$ at least one of the following properties holds:
  \begin{enumerate}
  \item $\mu_t\in\prob_2^r(\X)$,
  \item $\pclo\frF[\mu_t]$
    contains a unique element $\Phi_t$ concentrated
    on a map, i.e.~$\Phi_t=(\ii_\X, \bry{\Phi_t})_\sharp\mu_t$
  \end{enumerate}
  then $\mu$ is $\lambda$-\EVI solution of \eqref{eq:CP}.
\end{theorem}
\begin{proof}
  Take $\varphi \in \Cyl(\X)$ and observe that, since $\mu$ has the relaxed barycentric
  property, then
  for a.e. $t\in \interval$ (recall
  Theorem \ref{thm:refdiff})
  there exists $\Phi_t\in \cloco{\pclo\frF[\mu_t]}$ such that
\begin{equation*}
\frac{\de}{\de
  t}\int_\X\varphi(x)\de\mu_t(x)=\int_{\TX}\scalprod{\nabla\varphi(x)}{v}\de\Phi_t                         
  =\int_\X\scalprod{\nabla\varphi}{\proj_{\mu_t} \circ \bry{\Phi_t}}\de\mu_t=
    \int_\X \scalprod{\vv_t}{\nabla\varphi}\de\mu_t,
\end{equation*}
hence $\mu$ solves the continuity equation
$\partial_t\mu_t+\text{div}(\vv_t\mu_t)=0$, with
$\vv_t=\proj_{\mu_t} \circ
\bry{\Phi_t}\in\Tan_{\mu_t}\prob_2(\X)$.
By Theorem \ref{thm:refdiff}, we also know that 
\begin{equation}
\label{eq:162}
\frac{\de}{\de
  t}\frac{1}{2}W_2^2(\mu_t,\nu)=\int_{\X^2}\scalprod{\vv_t(x_0)}{x_0-x_1}\de\ggamma_t(x_0,x_1),\quad
t\in A(\mu,\nu), \
\ggamma_t\in \Gamma_o(\mu_t,\nu), \nu \in \prob_2(\X).
\end{equation}
Possibly disregarding a Lebesgue negligible set, we can
decompose the set $A(\mu,\nu)$ in the union
$A_1\cup A_2$, where $A_1, A_2$ correspond to
the times $t$ for which the properties (1) and (2) hold.

If $t \in A_1$ and $\nu\in\dom(\frF)$, then by \cite[Theorem
6.2.10]{ags}, since $\mu_t\in\prob_2^r(\X)$, there exists a unique
$\ggamma_t\in\Gamma_o(\mu_t,\nu)$ and
$\ggamma_t=(\ii_\X, \rr_t)_\sharp\mu_t$ for some map $\rr_t$ s.t. $\ii_{\X}-\rr_t\in
\Tan_{\mu_t}\prob_2(\X)\subset L^2_{\mu_t}(\X;\X)$
(recall \cite[Proposition 8.5.2]{ags}), so that
\begin{align}
\notag
&\int_{\X^2}\scalprod{\vv_t(x_0)}{x_0-x_1}\de\ggamma_t(x_0,x_1)
=\int_\X\scalprod{\vv_t(x_0)}{x_0-\rr_t(x_0)}\de\mu_t(x_0)\\
&=\int_\X\scalprod{\bry{\Phi_t}}{x_0-\rr_t(x_0)}\de\mu_t(x_0) =
\int_{\TX}\scalprod{v}{x-\rr_t(x)}\de\Phi_t(x,v)=
  \bram{\Phi_t}{\nu},
\label{eq:163}
\end{align}
where we also applied Theorem \ref{thm:characterization}
and Remark \ref{rem:particular},
recalling that in this case $\Lambda(\Phi_t,\nu)$ is a singleton.

If $t\in A_2$ we can select
the optimal plan $\ggamma_t\in \Gamma_o(\mu_t,\nu)$
along which
\begin{displaymath}
  \bram{\Phi_t}{\nu}=
  \directionalm{\Phi_t}{\ggamma_t}0=
  \int_\X\scalprod{\bry{\Phi_t}(x_0)}{x_0-x_1}\de\ggamma_t(x_0,x_1).
\end{displaymath}
If $\rr_t$ is the barycenter of $\ggamma_t$
with respect to its first marginal $\mu_t$,
recalling that $\ii_{\X}-\rr_t\in \Tan_{\mu_t}\prob_2(\X)$
(see also the proof of \cite[Thm.~12.4.4]{ags}) we also get
\begin{align}
  \notag
  \int_{\X^2}&\scalprod{\vv_t(x_0)}{x_0-x_1}\de\ggamma_t(x_0,x_1)
  =
  \int_\X\scalprod{\vv_t(x_0)}{x_0-\rr_t(x_0)}\de\mu_t(x_0)
    \\&=
  \int_\X\scalprod{\bry{\Phi_t}(x_0)}{x_0-\rr_t(x_0)}\de\mu_t(x_0)
  =
  \int_\X\scalprod{\bry{\Phi_t}(x_0)}{x_0-x_1}\de\ggamma_t(x_0,x_1)=
    \bram{\Phi_t}{\nu}
  \label{eq:161}
\end{align}
where we still applied Theorem \ref{thm:characterization}
and Remark \ref{rem:particular}.

Combining \eqref{eq:162} with \eqref{eq:163} and \eqref{eq:161}
we eventually get
\begin{align*}
\frac{\de}{\de
  t}\frac{1}{2}W_2^2(\mu_t,\nu)
  &
  =\bram{\Phi_t}{\nu}
  \le -\bram{\Psi}{\mu_t}+\lambda W_2^2(\mu_t,\nu), \quad\forall\,\Psi\in\frF[\nu],
\end{align*}
by definition of $\hat\frF$ and the fact that $ \cloco{\overline{\frF}}[\mu_t]\subset \hat \frF[\mu_t]$.
\end{proof}
  Thanks to Theorem \ref{thm:RegularM},
  we can apply to barycentric solutions
  the uniqueness and approximation results
  of the previous Sections.
We conclude this section with a general result on the existence of a $\lambda$-flow
for $\lambda$-dissipative {\MPVF}s, which is the natural refinement of Proposition
\ref{prop:local-bounded}
\begin{theorem}[Generation of $\lambda$-flow]
  \label{thm:global-bound}
  Let $\frF$ be a $\lambda$-dissipative \MPVF such that $\prob_b(\X)\subset \dom(\frF)$
  and for every $\mu_0\in \prob_b(\X)$ there exist $\varrho>0$ and $L>0$
  such that
  \begin{equation}
    \label{eq:144b}
    \supp(\mu)\subset \supp(\mu_0)+\mathrm B_\X(\varrho)\quad\Rightarrow\quad
    \exists \Phi\in \frF[\mu]: \supp(\sfv_\sharp\Phi)\subset \mathrm B_\X(L).
  \end{equation}
  Let $\frF_b := \frF \cap \prob_b(\TX)$. If there exists $a\ge 0$ such that for every $\Phi\in \frF_b$
  \begin{equation}
    \label{eq:145}
    \supp(\Phi)\subset
    \Big\{(x,v)\in \TX:
    \langle v,x\rangle\le a(1+|x|^2)\Big\},
  \end{equation}
  then $\frF$ generates a $\lambda$-flow.
\end{theorem}
\begin{proof} It is enough to prove that $\frF_b$ generates a $\lambda$-flow.
  Applying Proposition \ref{prop:local-bounded} to the \MPVF $\frF_b$,
  we know that for every $\mu_0\in \dom(\frF_b)$
  there exists a unique maximal strict $\lambda$-\EVI solution
  $\mu\in\Lip_{\rm loc}([0,T); \prob_b(\X))$ driven by $\frF_b$ and satisfying \eqref{eq:79}.
  We argue by contradiction, and we assume that $T<+\infty$.
  Notice that by \eqref{eq:144b} $\frF$ satisfies \eqref{eq:138},
  so that $\mu$ is a relaxed barycentric solution for $\frF_b$. Since
  $\mu_0\in \prob_b(\X)$, we know that $\supp(\mu_0)\subset \mathrm B_\X(r_0)$ for some
  $r_0>1$.

  It is easy to check that \eqref{eq:145} holds
  also for every $\Phi\in \cloco{\pclo\frF_b}$.
  Moreover, setting $b:=2a$, condition \eqref{eq:145} yields
  \begin{equation}
    \label{eq:165}
    \langle v,x\rangle \le b|x|^2\quad\text{for every}\quad
    (x,v)\in \supp\Phi \in \frF_b,\ |x|\ge 1.      
  \end{equation}
  Let $\phi(r):\R\to\R$ be any smooth increasing function
  such that $\phi(r)=0$ if $r\le r_0$ and $\phi(r)=1$ if $r\ge r_0+1$,
  and
  let $\varphi(t,x):=\phi(|x|\mathrm e^{-b t})$.
  Clearly $\varphi\in \mathrm C^{1,1}(\X\times [0,+\infty))$,
  with $\nabla \varphi(t,x)=\frac{x}{|x|}\phi'(|x|\mathrm e^{-b t})\mathrm e^{-b t}$
  if $x\neq 0$, $\nabla\varphi(t,0)=0$,
  and $\partial_t \varphi(t,x)=-b\phi'(|x|\mathrm e^{-b t})|x|\mathrm e^{-bt}.$
  We thus have for a.e.~$t\in [0,T)$
  \begin{align*}
    \frac\d{\d t}\int_\X \varphi(t,x)\,\d\mu_t
    &
      =
      \mathrm e^{-b t}\int_\TX \Big(-b\phi'(|x|\mathrm e^{-b t})|x|+
      \langle v,x\rangle|x|^{-1}\phi'(|x|\mathrm e^{-b t})\Big)\mathrm d\Phi_t(v,x)
    \\&
    \le
    \mathrm e^{-b t}\int_\TX \Big(-b\phi'(|x|\mathrm e^{-b t})|x|+
    b|x|\phi'(|x|\mathrm e^{-b t})\Big)\mathrm d\Phi_t(v,x)= 0
  \end{align*}
  where in the last inequality we used \eqref{eq:165} and 
  the fact that the integrand vanishes if $|x|\le 1$.
  We get 
  \begin{displaymath}
    \int_\X \varphi(t,x)\,\d\mu_t=0\quad\text{in }[0,T);
  \end{displaymath}
  this implies that $\supp(\mu_t)\subset \mathrm B_\X((r_0+1)\mathrm e^{bt})$
  so that the limit measure $\mu_T$ belongs to $\prob_b(\X)$ as well, leading to a contradiction with \eqref{eq:79} for $\frF_b$.

  We deduce that $\mu$ is a global strict $\lambda$-\EVI solution for $\frF_b$. We can then apply
  Theorem \ref{thm:globall}(b) to $\frF_b$.
\end{proof}
\subsection{A few borderline examples}
We conclude this section with a few examples which reveal
the importance of some of the technical tools we developed so far.
First of all we exhibit an example of dissipative \MPVF generating a
$0$-flow,
for which
solutions starting from initial data are merely continuous
(in particular the nice regularizing effect of gradient flows does not
hold
for general dissipative evolutions).
This clarifies the interest in a definition of continuous, not necessarily absolutely continuous, solution.
\begin{example}[Lifting of dissipative evolutions and lack of
  regularizing effect]
  \label{ex:rulla}
  Let us consider the situation of Corollary \ref{cor:consistency},
  choosing the Hilbert space $\X=\ell^2(\N)$.
  Following \cite[Example 3]{Rulla}
  we can easily find a maximal linear dissipative operator
  $A: \dom(A)\subset \ell^2(\N) \to \ell^2(\N)$
  whose semigroup 
  does not provide a regularizing effect.

  The domain of $A$ is
$\dom(A):= \{x\in \ell^2(\N):\sum_{k=1}^\infty k^2 |x_k|^2<\infty\} $
and $A$ is defined as
\[ A(x_1, x_2, \dots, x_{2k-1}, x_{2k}, \dots ) = (-x_2, x_1, \dots,
  -kx_{2k}, kx_{2k-1}, \dots ), \quad
  x\in \dom(A),
\]
so that there is no regularizing effect for the semigroup $(R_t)_{t
  \ge 0}$ generated by (the graph of) $A$:
evolutions starting outside the domain $\dom(A)$ stay
outside the domain and do not give raise to locally Lipschitz
or a.e.~differentiable curves.
Corollary \ref{cor:consistency}
shows that the $0$-flow $(S_t)_{t \ge 0}$ generated
by $\frF$ on $\prob_2(X)$
is given by 
\[ S_t[\mu_0] = (R_t)_{\sharp}\mu_0 \quad \text{ for every } \mu_0 \in \overline{\dom(\frF)} = \prob_2(\X)\]
so that there is the same lack of regularizing effect on probability
measures.
\end{example}
In the next example we show that a constant \MPVF generates a barycentric solution.
\begin{example}[Constant PVF and barycentric evolutions]
  Given $\theta \in \prob_2(\X)$, we consider the constant PVF
\[ \frF[\mu] := \mu \otimes \theta.\]
$\frF$ is dissipative: in fact,
if $\Phi_i=\mu_i\otimes \theta$, $i=0,1$, 
$\mmu\in \Gamma_o(\mu_0,\mu_1)$, and
$\rr:\X\times \X\times X\to \TX\times \TX$ is defined by
$\rr(x_0,x_1,v):=(x_0,v;x_1,v)$,
then
\begin{displaymath}
  \Theta=\rr_\sharp (\mmu\otimes \theta)\in \Lambda(\Phi_0,\Phi_1)
\end{displaymath}
so that \eqref{eq:22} yields
\begin{displaymath}
  \bram{\Phi_0}{\Phi_1}\le \int \la x_0-x_1,v-v\rangle \,\d
  (\mmu\otimes \theta)(x_0,x_1,v)=0.  
\end{displaymath}
Applying Proposition \ref{prop:gsolvcond} and Theorem
\ref{thm:global}
we immediately see that $\frF$ generates a $0$-flow $(\mathrm
S_t)_{t\ge0}$ in $\prob_2(\X)$, obtained as a
limit
of the Explicit Euler scheme.
It is also straightforward to notice that we can apply Theorem
\ref{theo:barycentric} to $\frF$ so that for every $\mu_0
\in \prob_2(\X)$ the unique \EVI solution $\mu_t=\mathrm S_t\mu_0$
satisfies the continuity equation
\begin{displaymath}
  \partial_t \mu_t+\nabla\cdot (\bb\mu_t)=0,\quad
  \bb=\int_\X x\,\d\theta(x).
\end{displaymath}
Since $\bb$ is constant, we deduce that $\mathrm S_t$
acts as a translation with constant velocity $\bb$, i.e.
\begin{displaymath}
  \mu_t=(\ii_\X+t\bb)_\sharp \mu_0,
\end{displaymath}
so that $\mathrm S_t$ coincides with the semigroup generated
by the PVF $\frF'[\mu]:=(\ii_\X,\bb)_\sharp \mu$.
\end{example}

We conclude this section with a $1$-dimensional example of a curve which satisfies the barycentric property but it is not an \EVI solution. 
\begin{example}
  \label{ex:bary}
  Let $\X= \R$.
  It is well known (see e.g.~\cite{Natile-Savare})
  that $\prob_2(\R)$ is isometric to
  the closed convex subset $\mathcal{K}
  \subset L^2(0,1)$ of the (essentially) increasing maps
  and the 
  isometry $\iso: \prob_2(\R) \to \mathcal{K}$ maps 
  each measure $\mu\in \prob_2(\R)$ 
  into the
  pseudo inverse of its cumulative distribution function.

  It follows that for every $\bar\nu\in \prob_2(\R)$
  the functional $\func: \prob_2(\R) \to \R$ defined as
  \[ \func(\mu) :=\frac{1}{2} W_2^2(\mu, \bar{\nu})\]
  is $1$-convex, since it satisfies $\func(\mu)=\lfunc(\iso(\mu))$
  where
  $\lfunc : L^2(0,1) \to \R$ is defined as
  \[ \lfunc(u) := \frac{1}{2}\|u-\iso(\bar{\nu}) \|^2\quad\text{for
      every }u\in L^2(0,1).\]
  Thus $\func$ generates a gradient flow $(\mathrm S_t)_{t \ge 0}$
  which is a semigroup of contractions in $\prob_2(\R)$;
  for every $\mu_0\in \prob_2(\R)$
  $\mathrm S_t[\mu_0]$ is the unique $(-1)$-\EVI solution for the
  \MPVF $-\boldsymbol\partial \func$ starting from $\mu_0 \in
  \prob_2(\X)$ (see Proposition \ref{prop:gfvsevi}).
  Since the notion of gradient flow is purely metric, the gradient flow of $\lfunc$ starting from $\iso(\mu_0)$ is just the image through $\iso$ of the gradient flow of $\func$ starting from $\mu_0 \in \prob_2(\X)$. It is easy to check that
\[ u(t) := \mathrm e^{-t} \iso(\mu_0) + (1-\mathrm e^{-t})\iso(\bar{\nu})\]
is the gradient flow of $\lfunc$ starting from $u_0=\iso(\mu_0)$. Note
that $u(t)$ is the $L^2(0,1)$ geodesic from $\iso(\bar{\nu})$ to
$\iso(\mu_0)$ evaluated at the rescaled time $e^{-t}$, so that
$\mathrm S_t[\mu_0]$ must coincide with the evaluation at time
$\mathrm e^{-t}$ of the (unique) geodesic connecting $\bar{\nu}$ to $\mu_0$ i.e.
\[ \mathrm S_t[\mu_0] = \sfx^{s}_{\sharp} \ggamma, \quad s=\mathrm
e^{-t}\in (0,1],\]
where $\ggamma\in\Gamma_o(\bar{\nu},\mu_0)$.

Let us now consider
the particular case $\bar{\nu}=\frac{1}{2} \delta_{-a} + \frac{1}{2}
\delta_a$, where $a>0$ is a fixed parameter and $\mu_0 = \delta_0$. It is straightforward to see that 
\[ \mu_t=\mathrm S_t[\delta_0] = \frac{1}{2} \delta_{a(1-\mathrm e^{-t})} +
  \frac{1}{2} \delta_{a(\mathrm e^{-t}-1)}, \quad t \ge 0\]
so that 
\[ (\ii_\X, \vv_t)_{\sharp} \mu_t=
  \frac{1}{2} \delta_{((1-e^{-t})a, e^{-t}a)} + \frac{1}{2}
  \delta_{((e^{-t}-1)a, -e^{-t}a)} \in -\boldsymbol \partial
  \func(\mu_t), \quad \text{ a.e. } t>0, \]
where $\vv$ is the Wasserstein velocity field of $\mu_t$.
On the other hand, \cite[Lemma 10.3.8]{ags} shows that
\[ \delta_0 \otimes \left ( \frac{1}{2}\delta_{-a} + \frac{1}{2} \delta_{a} \right ) \in - \boldsymbol \partial \func(\delta_0)\]
so that the constant curve $\bar \mu_t := \delta_0$ for $t\ge 0$
has the barycentric property for the \MPVF $-\boldsymbol \partial
\func$ but it is not a \EVI solution for $-\boldsymbol \partial
\func$,
being different from $\mu_t=\mathrm S_t[\delta_0]$.
\end{example}

\section{Explicit Euler Scheme}\label{sec:EulerScheme}
In this section, we collect all
the main estimates concerning the
Explicit Euler scheme \eqref{eq:EE}.

\subsection{The Explicit Euler Scheme: preliminary estimates}

Our first step is to prove simple a priori estimates and a discrete
version of \eqref{eq:EVI} as a consequence of Proposition
\ref{prop:concavity}.

\begin{proposition} \label{prop:discrete-evi}
  Every solution  
  $(M_\tau,\fF_\tau)\in \mathscr E(\mu_0,\tau,T,L)$ of \eqref{eq:EE} satisfies
  \begin{equation}
  W_2(M_{\tau}(t),\mu_0) \le L t,\quad
  |\fF_\tau(t)|_2 \le L
  \quad
  \text{for every $t \in [0,T]$,}\label{eq:108}
\end{equation}
\begin{equation}
  W_2(M_{\tau}(t), M_{\tau}(s)) \le L|t-s| \quad 
  \text{for every } s,t \in [0,T],
  \label{eq:109}
\end{equation}
  \begin{equation}
\frac{\de}{\de t} \frac{1}{2} W_2^2(M_{\tau}(t), \nu) \le 
\bram{\fF_{\tau}(t)}{\nu}
+ \tau
|\fF_{\tau}(t)|_2^2
\le \bram{\fF_{\tau}(t)}{\nu}
+ \tau L^2\label{eq:IEVI} \tag*{(IEVI)}
\quad
\text{in $[0,T]$, } \forall \nu \in \prob_2(\X),
\end{equation}
with possibly countable exceptions.
In particular
\begin{equation}\label{eq:ievigeneral}
  \frac{1}{2}W_2^2(M_{\tau}^{n+1}, \nu)-
  \frac{1}{2}W_2^2(M_{\tau}^{n}, \nu)\le
  \tau\bram{\fF_{\tau}^n}{\nu}+
  \frac12 {\tau^2 } L^2
  \quad\text{for every }
  0\le n<\finalstep T\tau, \forall \nu \in \prob_2(\X).
\end{equation}
\end{proposition}
\begin{proof}
  The second inequality of \eqref{eq:108} is a trivial consequence of the definition of
  $\mathscr E(\mu_0,\tau,T,L)$, the first inequality is a particular case of \eqref{eq:109}.
  The estimate \eqref{eq:109} is immediate if
$n\tau\le s<t\le (n+1)\tau$ since
\begin{align*} 
  W_2 (M_{\tau}(s), M_{\tau}(t)) &= W_2(  (\exp^{s-n\tau})_{\sharp}\fF_{\tau}^n, (\exp^{t-n\tau})_{\sharp}\fF_{\tau}^n)
  \le \sqrt{\int_{\TX} |(t-s) v)|^2 \de \fF_{\tau}^n} 
  \\&=(t-s) \sqrt{\int_{\TX} |v|^2 \de \fF_{\tau}^n} 
   \le (t-s)L.
\end{align*}
This implies that the metric velocity of $M_\tau$ is bounded by $L$
in $[0,T]$
and therefore $M_\tau$ is $L$-Lipschitz.

  Let us recall
  that
  for every $\nu \in \prob_2(\X)$ and $\Phi \in \prob_2(\TX)$
  the function
  $g(t):=  \frac{1}{2} W_2^2(\exp^t_{\sharp} \Phi,\nu)$
  satisfies
  \begin{equation}
   t\mapsto g(t)-\frac 12 t^2|\Phi|_2^2\text{ is concave},\quad
    g'_r(0) = \bram{\Phi}{\nu},\quad
  g'(t)\le \bram{\Phi}{\nu}+t|\Phi|_2^2\quad t\ge0,\label{eq:65s}
\end{equation}
by Definition \ref{def:scalarprodop} and Proposition \ref{prop:concavity}.
In particular, the concavity yields the differentiability of $g$ with at most countable exceptions.
Thus, taking any $n\in\N$, $0\le n<\finalstep T\tau$,
  $t\in[n\tau,(n+1)\tau)$ and $\Phi =
  \fF_{\tau}^n$ so that $\exp^t_{\sharp} \Phi =M_\tau(t)$,
  \eqref{eq:65s} yields \ref{eq:IEVI}.
  \eqref{eq:ievigeneral} follows by integration in each interval $[n\tau,(n+1)\tau]$.
\end{proof}

In the following, we prove a uniform bound on curves $M_\tau\in \mathscr M(\mu_0,\tau,T,L)$
which is useful to prove global solvability of the Explicit Euler scheme,
as stated in Proposition \ref{prop:gsolvcond}.
We will use the following discrete Gronwall estimate: if
a sequence $(x_n)_{n\in\N}$ of positive real numbers satisfies
\begin{equation*}
  x_{n+1}-x_{n}\le \tau y+\tau\alpha x_{n},\quad
  1\le n\le  N,\, \alpha\ge0,\,y\ge0,\, \tau>0,
\end{equation*}
then
\begin{equation}
  \label{eq:107}
  x_n\le (x_0+ \tau ny)\mathrm e^{\alpha n\tau}
  \quad 0\le n\le N+1.
\end{equation}
\begin{proposition}
  Let $\frF$ be a $\lambda$-dissipative \MPVF
  such that for every $R>0$ there exist $M=\mathrm M(R)>0$
  and $\bar\tau=\bar\tau(R)>0$
  such that
  \begin{equation}
    \label{eq:101}
    \mu\in \dom(\frF),\
    \rsqm \mu\le R,
    \ 0<\tau\le \bar\tau\quad
    \Rightarrow\quad
    \exists \,\Phi\in \frF[\mu]:
    |\Phi|_2\le \mathrm M(R),\
    \exp^\tau_\sharp\Phi\in \dom(\frF),
  \end{equation}
  then the Explicit Euler scheme is globally solvable in $\dom(\frF)$.
  More precisely,
  if for a given $\mu_0\in \dom(\frF)$
  with $\Psi_0\in \frF[\mu_0]$,   $\m_0:=\rsqm{\mu_0},$ and
  we set
\begin{equation}
  R:= \m_0 +\Big(|\Psi_0|_2+1\Big) \sqrt{2T}\mathrm e^{(1+2\lambda_+)T},\quad
  L:=\mathrm M(R),
  \quad
  \ttau=  \frac{1}{L^2} \land
  \bar\tau(R) \wedge T,
  \label{eq:110}
\end{equation}
then
for every $\tau\in (0,\ttau]$ the set 
$\mathscr E(\mu_0,\tau,T,L)$ is not empty.
\end{proposition}
\begin{proof} 
  We want to prove by induction that
  for every integer $N\le \finalstep T\tau$,
  \eqref{eq:EE} has a solution up to the index $N$ satisfying the
  upper bound
  \begin{equation}
    \label{eq:106}
    \rsqm{M^{N}_\tau}\le R,
  \end{equation}
  corresponding to the constants $R,L$ given by \eqref{eq:110}.
  For $N=0$ the statement is trivially satisfied.
  Assuming that 
  $0\le N<\finalstep T\tau$ and elements $(M^n_\tau,\frF^n_\tau)$,
  $0\le n< N$, $M^N_\tau$,
  are given satisfying \eqref{eq:EE} and \eqref{eq:106},
  we want to show that we can perform a further step of the Euler Scheme
  so that \eqref{eq:EE} is solvable up to the index $N+1$ and
  $\rsqm{M^{N+1}_\tau}\le R$.  
  
  Notice that by the induction hypothesis, for $n=0,
  \dots, N-1$, we have
  $|\fF^n_\tau|_2\le L$; since
  $\rsqm{M^N_\tau}\le R$, by \eqref{eq:101} 
  we can select $\fF^N_\tau\in \frF[M^N_\tau]$
  with $|\fF^N_\tau|_2\le L$ such that $M^{N+1}_\tau=\exp^\tau_\sharp\fF^N_\tau\in \dom(\frF)$.
Using \eqref{eq:ievigeneral} with $\nu=\mu_0$,
the $\lambda$-dissipativity with $\Psi_0\in \frF[\mu_0]$ 
\[\bram{\fF_{\tau}^n}{\mu_0}\le\lambda W_2^2(M_{\tau}^{n},\mu_0)-\bram{\Psi_0}{M_{\tau}^{n}},\]
and the bound
\begin{equation*}
  -\bram{\Psi_0}{M_{\tau}^n}\le\frac{1}{2}W_2^2(M_{\tau}^n,\mu_0)+\frac{1}{2}|\Psi_0|_2^2,
\end{equation*}
we end up with
\begin{equation*}
  \frac{1}{2}W_2^2(M_{\tau}^{n+1}, \mu_0)-
  \frac{1}{2}W_2^2(M_{\tau}^{n}, \mu_0) \le
  \frac{\tau^2}{2} L^2 +
  \tau \left (\frac{1}{2} + \lambda_+ \right)\,
  W_2^2(M_{\tau}^{n}, \mu_0) +
  \frac{\tau}{2} |\Psi_0|_2^2 ,
\end{equation*}
for every $n\le N$.
Using the Gronwall estimate \eqref{eq:107}
we get
\begin{align*}
  W_2(M_{\tau}^n, \mu_0) 
  &\le \sqrt {T+\tau} \Big(
    |\Psi_0|_2+\sqrt\tau L\Big)
    \mathrm e^{(\frac{1}{2}+\lambda_+)\, (T+\tau)}
    \le
    \sqrt {2T}\Big(
    |\Psi_0|_2+1\Big)
    \mathrm e^{(1+2\lambda_+) T}
\end{align*}
for every $n\le N+1$, so that
\begin{equation*}
  \rsqm{M^{N+1}_\tau}\le \m_0+ \sqrt {2T}\Big(
    |\Psi_0|_2+1\Big)
    \mathrm e^{(1+2\lambda_+) T}\le R.
    \qedhere
\end{equation*}
\end{proof}

We conclude this section by proving the stability estimate \eqref{eq:71} of Theorem
\ref{thm:apriori-estimate}.
We introduce the notation
  \begin{equation*}
  I_\kappa (t):=\int_0^t \mathrm e^{\kappa  r}\,\d r=
  \frac{1}{\kappa}(\mathrm e^{\kappa t}-1)\quad\text{if }\kappa\neq0;\quad
    I_0(t):=t.
  \end{equation*}
  Notice that for every $t\ge0$
  \begin{equation}
    \label{eq:136}
    I_\kappa(t)\le t\mathrm e^{\kappa t}\quad\text{if }\kappa\ge0;\qquad
  \end{equation}

\begin{proposition}\label{prop:samestep}
  Let $M_\tau \in \mathscr M(\mu_0,\tau,T,L)$ and
  $M_\tau'\in \mathscr M(\mu_0',\tau,T,L)$.
  If $\lambda_+\tau\le 2$ then
  \begin{equation*}
    W_2(M_\tau(t),M_\tau'(t))\le
    W_2(\mu_0,\mu_0')\mathrm e^{\lambda  t}+
    8L\sqrt {t\tau}\Big(1+|\lambda|\sqrt{t\tau}\Big)\mathrm e^{\lambda_+ t}
  \end{equation*}
  for every $t\in [0,T]$.
\end{proposition}
\begin{proof}
  Let us set
  $w(t):=W_2(M_\tau(t),M_\tau'(t))$.
  Since by Proposition \ref{prop:concavity}(2),
  in every interval $[n\tau,(n+1)\tau]$ the function $t\mapsto w^2(t)-4L^2 (t-n\tau)^2$ is concave, with
  \begin{displaymath}
    \frac{\d }{\d t}w^2(t)\bigg|_{t=n\tau+}=
    2\bram{\fF_\tau(t)}{\fF_\tau'(t)}
    \le 2\lambda W_2^2(\bar M_\tau(t),\bar M_\tau'(t)),
  \end{displaymath}
  we obtain
  \begin{equation*}
    \frac{\d}{\d t} w^2(t)\le  2\lambda W_2^2(\bar M_\tau(t),\bar M_\tau'(t))+8L^2\tau
    \quad t\in [0,T],
  \end{equation*}
  with possibly countable exceptions.
  Using the identity $a^2-b^2=2b(a-b)+|a-b|^2$
  with $a=W_2(\bar M_\tau(t),\bar M_\tau'(t))$ and
  $b=W_2(M_\tau(t),M_\tau'(t))$ and observing that
  $|a-b|\le W_2(\bar M_\tau(t),M_\tau(t))+W_2(\bar M_\tau'(t),M_\tau'(t))\le 2L\tau$,
  we eventually get
  \begin{align*}
    \frac{\d}{\d t} w^2(t)
    &\le
      2\lambda w^2(t) +8L^2\tau+ 8|\lambda|L\tau  w(t)
    +\lambda_+ 8L^2\tau^2
    \\&\le
    2\lambda w^2(t) +8|\lambda|L\tau  w(t)
    +24 L^2\tau,
  \end{align*}
  since $\lambda_+\tau\le 2$ by assumption.
  The Gronwall lemma \cite[Lemma 4.1.8]{ags} and \eqref{eq:136} yield
  \begin{align*}
    w(t)
    &
      \le
      \Big(w^2(0) \mathrm e^{2\lambda t}+24L^2\tau \mathrm I_{2\lambda}(t)\Big)^{1/2}+8|\lambda| L\tau \mathrm I_{\lambda}(t)
    \\&\le
    w(0) \mathrm e^{\lambda t}+8L\sqrt {t\tau}\Big(1+|\lambda|\sqrt{t\tau}\Big)
    \mathrm e^{\lambda_+ t}.
    \qedhere
  \end{align*}
\end{proof}

\subsection{Error estimates for the Explicit Euler scheme}\label{subsec:exist}

\begin{theorem} \label{theo:convergence} Let $\frF$ be a
  $\lambda$-dissipative \MPVF.
  If $M_\tau\in \mathscr M(M^0_\tau,\tau,T,L)$, $M_\eta\in \mathscr M(M^0_\eta,\eta,T,L)$
  with $\lambda\sqrt{T(\tau+\eta)}\le 1$,
  then for every $\vartheta>1$
  there exists a constant $C(\vartheta)$ such that 
    \begin{equation*}
    W_2(M_\tau(t),M_\eta(t))\le
    \Big(\sqrt{\vartheta} W_2(M^0_\tau,M^0_\eta)+
    C(\vartheta)
    L \sqrt{(\tau+\eta)(t+\tau+\eta)}\Big)\mathrm e^{\lambda_+\, t}
  \end{equation*}
  for every $t \in [0,T]$.
\end{theorem}
\begin{proof}
We argue as in the proof of Theorem \ref{theo:duesolcomp}
 with the aim to gain a convenient order of convergence.
  Since $\lambda$-dissipativity implies $\lambda'$-dissipativity for $\lambda'\ge \lambda$,
  it is not restrictive to assume
  $\lambda> 0$. We set $\sigma:=\tau+\eta$.
We will extensively use the a priori bounds \eqref{eq:108} and \eqref{eq:109};
in particular,
\begin{equation*}
  W_2(M_{\tau}(t), \bar M_{\tau}(t))\le L\tau,\quad
  W_2(M_{\eta}(t), \bar M_{\eta}(t))\le L\eta.
\end{equation*}

We will also extend $M_\tau$ and $\bar M_\tau$ for negative times by setting
\begin{equation}
  \label{eq:117}
  M_\tau(t)=\bar M_\tau(t)=M^0_\tau,\quad
  \frF_\tau(t)=M^0_\tau\otimes \delta_0\quad\text{if }t<0.
\end{equation}
 The proof is divided into several steps.

\smallskip\noindent
\emph{1. Doubling variables.}\par
We fix a final time $t\in [0,T]$ and
two variables $r,s\in [0,t]$ together with
the functions
\begin{equation}
  \label{eq:124}
  \begin{aligned}
    w(r,s):={}&W_2(M_\tau(r),M_\eta(s)),\quad&
  w_\tau(r,s):={}&W_2(\bar M_\tau(r),M_\eta(s)),\\
  w_\eta(r,s):={}&W_2(M_\tau(r),\bar M_\eta(s)),\quad&
  w_{\tau,\eta}(r,s):={}&W_2(\bar M_\tau(r),\bar M_\eta(s)),
\end{aligned}
\end{equation}
observing that
\begin{equation}
  \label{eq:125}
  |w-w_\tau|\lor |w_\eta-w_{\tau,\eta}|\le L\tau,\quad
  |w-w_\eta|\lor |w_\tau-w_{\tau,\eta}|\le L\eta.
\end{equation}
By Proposition \ref{prop:discrete-evi}, we can write \ref{eq:IEVI} both for $M_{\tau}$ and $M_{\eta}$ and we obtain
\begin{align}
  \label{eq:120}
  \tag{$\text{IEVI}_{\tau}$}
  \frac{\partial}{\partial r} \frac{1}{2} W_2^2(M_{\tau}(r), \nu_1)
  &\le
                                                                    \tau
                                                                    |\fF_\tau(r)
                                                                    |_2^2
                                                                    +
                                                                    \bram{\fF_\tau(r)
                                                                    }{\nu_1}
  \quad\forall\nu_1 \in \prob_2(\X)\\
  \notag
  \frac{\partial}{\partial s} \frac{1}{2} W_2^2(M_{\eta}(s), \nu_2)
  &\le
    \eta
    |\fF_\eta(s)
    |_2^2 +
    \bram{\fF_\eta(s)
    }{\nu_2}
  \\\label{eq:123}&\le \tag{$\text{IEVI}_{\eta}$}
  \eta
    |\fF_\eta(s)
  |_2^2+
  \lambda W_2^2(\bar M_\eta(s),\nu_2)-
  \bram{\Phi
  }{\bar M_\eta(s)},\quad\forall\nu_2 \in \dom(\frF),\,\Phi\in \frF[\nu_2].
\end{align}
Apart from possible countable exceptions, \eqref{eq:120} holds for $r\in (-\infty,t]$ and
\eqref{eq:123} for $s \in [0,t]$.
Taking
$\nu_1 =
\bar{M}_{\eta}(s)$,
$\nu_2=\bar{M}_{\tau}(r)$,
$\Phi= \fF_\tau(r\lor 0)\in \frF[\bar M_\tau(r)]$,
summing the two inequalities
$(\text{IEVI}_{\tau,\eta})$, setting
\begin{displaymath}
  f(r,s):=
  \begin{cases}
    2L W_2(\bar M_\eta(s),M_\tau(0))=2Lw_\eta(0,
    s)&   \text{if }r<0,\\
    0&\text{if }r\ge 0,
\end{cases}
\end{displaymath}
using \eqref{eq:108}
and the $\lambda$-dissipativity of $\frF$, we obtain
\begin{equation*}
  \frac{\partial}{\partial r}
  w_\eta^2(r,s)
  +  \frac{\partial}{\partial s}
  w_\tau^2(r,s)
  \le 2\lambda
  w_{\tau,\eta}^2(r,s)
  + 2L^2\sigma+f(r,s)
\end{equation*}
in $(-\infty,t]\times [0,t]$
(see also \cite[Lemma 6.15]{NochettoSavare}).
By multiplying both sides by $e^{-2\lambda s}$, we have
\begin{equation} \label{eq:starting-inequality12}
\begin{split}
  &\frac{\partial}{\partial r}  
  \mathrm e^{-2\lambda s} w_\eta^2
  +  \frac{\partial}{\partial s}
  \mathrm e^{-2\lambda s}
  w_\tau^2
  \le \Big(2\lambda \left(
    w_{\tau,\eta}^2
    -w_\tau^2
  \right) + f
    +2L^2\sigma\Big)\mathrm e^{-2\lambda s}.
  \end{split}
\end{equation}
Using \eqref{eq:125}, the inequality
\begin{displaymath}
  w_{\tau,\eta}+w_\tau=w_{\tau,\eta}-w_\tau+2(w_\tau-w)+2w\le 2L\sigma +2w,\quad
  |w(r,s)-w(s,s)|\le L|r-s|
\end{displaymath}
and the elementary inequality
$a^2-b^2\le |a-b||a+b|,$
we get
\begin{equation*}
  2\big(w_{\tau,\eta}^2(r,s)-w_\tau^2(r,s)\big)\le
  R,\quad
  R:=4L^2\sigma(\sigma+|r-s|)+4L\sigma w(s,s)\quad\text{if }r,s\le t.
\end{equation*}
Thus \eqref{eq:starting-inequality12} becomes
\begin{equation} \label{eq:starting-inequality1}
\frac{\partial}{\partial r}  \mathrm e^{-2\lambda s}
w_\eta^2
+  \frac{\partial}{\partial s}
\mathrm e^{-2\lambda s}
w_\tau^2
\le
Z,
\quad
Z:=\Big(R \lambda+f+
2L^2\sigma\Big) \mathrm e^{-2\lambda s}.
\end{equation}
\smallskip\noindent
\emph{2. Penalization.}\par
We fix any $\eps>0$
and apply the Divergence Theorem to the inequality \eqref{eq:starting-inequality1} in the two-dimensional strip $Q^{\eps}_{0,t}$ as in Figure \ref{F:strip} and we get
\begin{align} 
\begin{split} \label{eq:gg}
  \int_{t-\eps}^t&
  \mathrm e^{-2\lambda t}
  w_\tau^2(r,t)
  \de r \le \int_{-\eps}^0
  w^2_\tau(r,0)\,\d r +\\
  &+ \int_0^t
  \mathrm e^{-2\lambda s}
  \left ( 
    w_\tau^2(s,s)-w_\eta^2(s,s)
  \right )\de s 
+ \int_0^t \mathrm e^{-2\lambda s} \left ( 
  w_\eta^2(s-\eps,s)-
  w_\tau^2(s-\eps,s)
\right )\de s\\
&+\iint_{Q_{0,t}^\eps} Z\,\d r \d s.
\end{split}
\end{align}
\emph{3. Estimates of the RHS.}\par
We want to estimate the integrals
(say $I_0,I_1,I_2,I_3)$ 
of the right hand side of \eqref{eq:gg}
in terms of
\begin{equation*}
  w(s):=w(s,s)\quad\text{and}\quad
  W(t):=\sup_{0\le s\le t}\mathrm e^{-\lambda s}w(s).
\end{equation*}
We easily get
\begin{equation*}
  I_0=\int_{-\eps}^0  w_\tau^2(r,0)\,\d r=\eps w^2(0).
\end{equation*}
\eqref{eq:125} yields
\begin{align*}
  |w_\tau(s,s)-w_\eta(s,s)|
  \le L(\tau+\eta)=L\sigma
\end{align*}
and
\begin{equation*}
  |w^2_\tau(s,s)-w^2_\eta(s,s)|
  \le L\sigma\Big( L\sigma+2w(s)
  \Big);
\end{equation*}
after an integration,
\begin{equation*}
  \begin{aligned}
    I_1&
    \le {L^2
      \sigma^2 t}
    +2L\sigma \int_0^t e^{-2\lambda s}w(s)
    \de s\le
    L^2\sigma^2t +2L\sigma tW(t).
  \end{aligned}
\end{equation*}
Performing the same computations for the third integral term at the RHS of  \eqref{eq:gg} we end up with
\begin{equation*}
\begin{split}
  I_2 &
  =\int_0^t e^{-2\lambda s} \left ( 
  w_\eta^2(s-\eps,s)-w_\tau^2(s-\eps,s)
\right )\de s 
\le
{L^2
  t
  \sigma^2}
+ 2L\sigma \int_0^t e^{-2\lambda s} w(s-\eps,s) \de s\\
&\le
L^2\sigma^2t+2L^2\sigma\eps t+ 
 2L\sigma \int_0^t e^{-2\lambda s} w(s)
 \de s
 \le L^2\sigma^2t+2L^2\sigma\eps t+ 2L\sigma t W(t).
\end{split}
\end{equation*}
Eventually, using the elementary inequalities,
\begin{displaymath}
  \iint_{Q^\eps_{0,t}} \lambda\mathrm e^{-2\lambda s}\,\d r\,\d s\le \frac \eps2,\quad
  \iint_{Q^\eps_{0,t}} \mathrm e^{-2\lambda s}w(s,s)\,\d r\,\d s=
  \eps\int_0^t \mathrm e^{-2\lambda s} w(s)\,\d s,
\end{displaymath}
and
$f(r,s)\le 2L^2(\eta+s)+2Lw(s)$ for $r<0$ and $f(r,s)=0$ for $r\ge0$,
we get
\begin{align*}
  \notag
  I_3&=\iint_{Q^\eps_{0,t}}Z\,\d r \d s
  \le 2L^2\sigma\eps (\sigma+\eps)+
    4L\lambda \sigma\eps \int_0^t \mathrm e^{-2\lambda s} w(s)\,\d s+
    2L^2\sigma\eps t \\\notag
  &\qquad+
    2\iint_{Q^\eps_{0,\eps \wedge t}}(L^2(\eta+s)+Lw(s))\mathrm e^{-2\lambda s}\,\d r \d s
  \\
  &
    \le
    2L^2\sigma\eps (\sigma+\eps)+
    2L^2\eps^2 (\sigma+\eps)+
    2L^2\sigma \eps t +
    4L\lambda \sigma \eps tW(t)+
    2L\eps^2 W(t\land \eps).
\end{align*}
We eventually get
\begin{equation}
  \label{eq:149}
  \sum_{k=0}^3 I_k\le \eps w^2(0)+2L^2\sigma^2 t+4L^2\sigma\eps t+
  2L^2\eps(\sigma+\eps)^2+
  4L\sigma(1+\lambda\eps) tW(t)
  +2L\eps^2 W(t\land \eps).
\end{equation}

\smallskip\noindent
\emph{4. LHS and penalization}\par
We want to use the first integral term in \eqref{eq:gg} to
derive a pointwise estimate for $w(t)$;

\eqref{eq:109} and \eqref{eq:124} yield
\begin{align*}
  w(t)=w(t,t)\le L(t-r)+w(r,t)\le L(\tau+|t-r|) +w_\tau(r,t)
\end{align*}
so that
we get for every $\vartheta,\vartheta_\star>1$ conjugate coefficients
\begin{equation*}
  \mathrm e^{-2\lambda t}
  w^2(t)
  \le 
  \frac\vartheta\eps\int_{t-\eps}^t e^{-2\lambda t}
  w_\tau^2(r,t)
  \de r 
  + \vartheta_\star L^2 (\tau+\eps)^2
  \le 
  \frac\vartheta\eps(I_0+I_1+I_2+I_3)
  + \vartheta_\star L^2 (\tau+\eps)^2.
\end{equation*}
\eqref{eq:149} yields
\begin{align*} 
\begin{split}
  \mathrm e^{-2\lambda t}
  w^2(t)
   \le&
  (2\vartheta+\vartheta_\star) L^2(\sigma+\eps)^2+
  \vartheta\Big(w^2(0)+
  2L^2\sigma^2t/\eps+4 L^2\sigma t
  \Big)
  \\&+\frac{4L(1+\lambda\eps)\sigma \vartheta}{\eps} tW(t)+
  2L\eps\vartheta W(t\land \eps).
\end{split}
\end{align*}
\smallskip\noindent
\emph{5. Conclusion.}\par
Choosing $\eps:=\sqrt{\sigma(\sigma\lor t)}$ and assuming $\lambda\sqrt{T \sigma}\le 1$,
we obtain
  \begin{align}
      \mathrm e^{-2\lambda t}
    w^2(t)
    &\le
      \vartheta w^2(0)
      +
      (14\vartheta+4\vartheta_\star) L^2\sigma(\sigma\lor t)
      +10\vartheta L\sqrt{\sigma(\sigma\lor t)} W(t).
      \label{eq:118}
  \end{align}
  Since the right hand side of \eqref{eq:118} is an increasing
  function of $t$,
  \eqref{eq:118} holds even if we substitute
  the left hand side with 
  $\mathrm e^{-2\lambda s}w^2(s)$ for every $s\in [0,t]$;
  we thus obtain the inequality
  \begin{equation*}
    W^2(t)
    \le
      \vartheta w^2(0)
      +
      (14\vartheta+4\vartheta_\star) L^2\sigma(\sigma\lor t)
      +10\vartheta L\sqrt{\sigma(\sigma\lor t)} W(t).
  \end{equation*}
Using the elementary property for positive $a,b$
\begin{equation}
  \label{eq:elementary}
    W^2\le a+2bW\quad\Rightarrow\quad
    W\le b+\sqrt{b^2+a}\le 2b+\sqrt a,
  \end{equation}
we eventually obtain
  \begin{align*}
    \mathrm e^{-\lambda t}w(t)
    &\le
    \Big(\vartheta w^2(0)+
    (14\vartheta+4\vartheta_\star)L^2\sigma(\sigma\lor t)\Big)^{1/2}
      +10\vartheta L\sqrt{\sigma(\sigma\lor t)}
    \\&\le
    \sqrt{\vartheta} w(0)+
    C(\vartheta)L\sqrt{\sigma(\sigma\lor t)},\quad
    C(\vartheta):=(14\vartheta+4\vartheta_\star)^{1/2}+
    10\vartheta.
    \qedhere
  \end{align*}
\end{proof}

\subsection{Error estimates between discrete and \EVI solutions}
\label{subsec:error2}
\begin{theorem} \label{prop:rate} 
  Let $\frF$ be a $\lambda$-dissipative \MPVF.
  If $\mu\in \rmC([0,T];\overline{\dom(\frF)})$ is a $\lambda$-\EVI solution
  and $M_\tau\in \mathscr M(M^0_\tau,\tau,T,L)$,
  then for every $\vartheta>1$
  there exists a constant $C(\vartheta)$ such that 
\begin{equation*}
  W_2(\mu(t), M_{\tau}(t))\le
  \Big(\sqrt{\vartheta}\, W_2(\mu_0,M^0_\tau)+C(\vartheta)
  L\sqrt{\tau(t+\tau)}\Big)
  \mathrm e^{\lambda_+ t}
  \quad
  \text{for every }t\in [0,T].
\end{equation*}
\end{theorem}
\begin{remark}
  When $\mu_0=M^0_\tau$ and $\lambda\le 0$ we obtain the optimal error
  estimate
  \begin{equation*}
    W_2(\mu(t), M_{\tau}(t))\le13 L\sqrt {\tau(t+\tau)}.
  \end{equation*}
\end{remark}
\begin{proof}
  We repeat the same argument of the previous proof, still assuming $\lambda>0$,
  extending $M_\tau,\bar M_\tau,\bar \fF_\tau$ as in \eqref{eq:117}
  and setting
  \begin{equation*}
    w(r,s):=W_2(M_\tau(r),\mu(s)),\quad
    w_\tau(r,s):=W_2(\bar M_\tau(r),\mu(s)).
  \end{equation*}
  We use \eqref{eq:EVI} for $\mu(s)$ with $\nu=\bar M_\tau(r)$ and $\Phi=\frF_\tau(r\lor 0)$ and 
  \ref{eq:IEVI} for $M_{\tau}(r)$ with $\nu=\mu(s)$ obtaining
\begin{align*}
  \frac{\partial}{\partial r} \frac{\mathrm e^{-2\lambda s}}{2} W_2^2(M_{\tau}(r), \mu(s))
  &\le
    \mathrm e^{-2\lambda s}\Big(\tau
    |\fF_\tau(r)
    |_2^2 +
    \bram{ \fF_{\tau}(r)
    }{\mu(s)}\Big)  & s\in [0,T], r \in (-\infty, T)
  \\
  \frac{\partial}{\partial s} \frac{\mathrm e^{-2\lambda s}}{2} W_2^2(\mu(s), \bar M_\tau(r))
  &\le
    -\mathrm e^{-2\lambda s}   \bram{\fF_\tau(r \lor 0)
                                                                    }{\mu(s)}
  &\text{in $\mathscr D'(0,T)$, } r\in (-\infty,T).
\end{align*}
Using \cite[Lemma 6.15]{NochettoSavare}
we can sum the two contributions
obtaining
\begin{equation*}
  \frac{\partial}{\partial r}
    \mathrm e^{-2\lambda s}
  w^2(r,s)+
  \frac{\partial}{\partial s}
    \mathrm  e^{-2\lambda s}
  w_\tau^2(r,s)
\le Z,\quad
Z:=(2L^2 \tau+2f(r,s))\mathrm e^{-2\lambda s},
\end{equation*}
where
\begin{displaymath}
  f(r,s):=
  \begin{cases}
    L W_2(M_\tau(0),\mu(s))=Lw(0,
    s)&   \text{if }r<0,\\
    0&\text{if }r\ge 0.
\end{cases}
\end{displaymath}

Let $t \in [0,T]$ and $\eps>0$. Applying the Divergence Theorem in $Q_{0,t}^\eps$ (see Figure \ref{F:strip})  we get
\begin{align} 
\begin{split} \label{eq:gg2}
  \int_{t-\eps}^t&
  \mathrm e^{-2\lambda t}
  w_\tau^2(r,t)
  \de r \le
  \int_{-\eps}^0
  w_\tau^2(r,0)\,\d r\\
  &+ \int_0^t
  \mathrm e^{-2\lambda s}
  \left ( 
    w_\tau^2(s,s)-w^2(s,s)
  \right )\de s 
+ \int_0^t \mathrm e^{-2\lambda s} \left ( 
  w^2(s-\eps,s)-
  w_\tau^2(s-\eps,s)
\right )\de s\\
&+
\iint_{Q_{0,t}^\eps} Z \,\d r\d s.
\end{split}
\end{align}
Using
\begin{equation*}
  w(t,t)\le w(r,t)+L(t-r)\le w_\tau(r,t)+L(\tau+\eps)\quad\text{if }t-\eps\le r\le t,
\end{equation*}
we get for every $\vartheta,\vartheta_\star>1$ conjugate coefficients ($\vartheta_\star=\vartheta/(\vartheta-1)$)
\begin{equation}
  \label{eq:firstbound-2}
  \mathrm e^{-2\lambda t}
  w^2(t)
  \le 
  \frac{\vartheta}{\eps}\int_{t-\eps}^t e^{-2\lambda t}
  w_\tau^2(r,t)
  \de r 
  + \vartheta_\star L^2 (\tau+\eps)^2.
\end{equation}
Similarly to \eqref{eq:125} we have
\begin{align*}
  |w_\tau(s,s)-w(s,s)|
  \le L\tau,\quad
    |w^2_\tau(s,s)-w^2(s,s)|
  \le L\tau\Big( L\tau+2w(s)
  \Big)
\end{align*}
and, after an integration,
\begin{equation}\label{eq:firstb-2}
  \begin{aligned}
    &\int_0^t e^{-2\lambda s} \left ( 
      w^2_\tau(s,s)-w^2(s,s) 
    \right )\de s
    \le {L^2 t
      \tau^2}
    +2L\tau \int_0^t e^{-2\lambda s}w(s)
    \de s. 
\end{aligned}
\end{equation}
Performing the same computations for the third integral term at the RHS of  \eqref{eq:gg2} we end up with
\begin{equation}\label{eq:secondb1eps2}
\begin{split}
\int_0^t &e^{-2\lambda s} \left ( 
  w^2(s-\eps,s)-w_\tau^2(s-\eps,s)
\right )\de s 
\le
{L^2t
  \tau^2}
+ 2L\tau \int_0^t e^{-2\lambda s} w(s-\eps,s) \de s\\
&\le
L^2 t
\tau(\tau+2 \eps)
+ 2L\tau \int_0^t e^{-2\lambda s} w(s)
\de s.
\end{split}
\end{equation}
Finally, since if $r<0$ we have $f(r,s)=Lw(0,s)\le L^2s+Lw(s,s)$, then
\begin{align}
  \notag
  \eps^{-1}\iint_{Q^\eps_{0,t}}Z\,\d r \d s
  &\le
    2L^2t\tau+\eps^{-1}\iint_{Q^\eps_{0,\eps \land t}}2f(r,s)\mathrm e^{-2\lambda s}\,\d r \d s
    \\
  &\le 2L^2t\tau +L^2 \eps^2+2L\eps \sup_{0\le s\le \eps\land t}\mathrm e^{-\lambda s}w(s).
  \label{eq:128}
\end{align}
Using \eqref{eq:firstb-2}, \eqref{eq:secondb1eps2}, \eqref{eq:128} in \eqref{eq:gg2}, we can rewrite the bound in \eqref{eq:firstbound-2} as
\begin{align*} 
\begin{split}
  \mathrm e^{-2\lambda t}
  w^2(t)
  & \le
  \vartheta_\star L^2(\tau+\eps)^2+  \vartheta \Big(w^2(0)+
  2L^2t\tau^2/\eps+2L^2t\tau
  +L^2\eps^2+
  2L\eps \sup_{0\le s\le \eps\land t}\mathrm e^{-\lambda s}w(s)
  \Big)\\
  & \quad +\frac{4\vartheta L\tau}{\eps}  \int_0^t e^{-2\lambda s} w(s)
  \de s.
\end{split}
\end{align*}
\smallskip\noindent
Choosing $\eps:=\sqrt{\tau(\tau\lor t)}$
we get
\begin{equation*} 
  \mathrm e^{-2\lambda t}
  w^2(t) \le
  4\vartheta_\star L^2\tau(t\lor \tau)+
  \vartheta \Big(w^2(0)+
  5L^2\tau(t\lor \tau)\Big)
  +6\vartheta L\sqrt {\tau(t\lor \tau)}
  \sup_{0\le s\le t}\mathrm e^{-\lambda s}w(s).
\end{equation*}
A further application of
\eqref{eq:elementary} yields
  \begin{align*}
    \mathrm e^{-\lambda t}
    w(t)
    &\le
      \Big(\vartheta w^2(0)+
      (5\vartheta+4\vartheta_\star) L^2\tau(t\lor \tau)\Big)^{1/2}
      +6\vartheta L\sqrt {\tau(t\lor \tau)}
    \\&\le
    \sqrt \vartheta w(0)+ C(\vartheta)L\sqrt{t+\tau}\sqrt\tau ,\quad
    C(\vartheta):=(5\vartheta+4\vartheta_\star)^{1/2}+6\vartheta.
    \qedhere
  \end{align*}
\end{proof}

As proved in the following, the limit curve of the interpolants
$(M_{\tau})_{\tau>0}$ of the Euler Scheme defined in
\eqref{eq:affineM} is actually a $\lambda$-\EVI solution of \eqref{eq:CP}.
\begin{theorem} \label{theo:strong-solution} Let $\frF$ be a $\lambda$-dissipative \MPVF and let $n\mapsto \tau(n)$ be a vanishing sequence of
  time steps,
  let $(\mu_{0,n})_{n\in \N}$ be a sequence in $\dom(\frF)$ converging to
  $\mu_0\in \overline{\dom(\frF)}$ in $\prob_2(\X)$ and
  let $M_n\in \mathscr M(\mu_{0,n},\tau(n),T,L)$.
  Then $M_n$ is uniformly converging to a limit curve
  $
  \mu\in \Lip([0,T];\overline{\dom(\frF)})$ which
  is a $\lambda$-\EVI solution
  starting from $\mu_0$.
\end{theorem}
\begin{proof} Theorem \ref{theo:convergence}
  shows that $M_n$ is a Cauchy sequence in
  $\mathrm C([0,T];\overline{\dom(\frF)})$,
  so that there exists a unique limit curve $\mu$
  as $n\to\infty$. $\mu$ is also $L$-Lipschitz; moreover we observe that
  \begin{equation}\label{eq:alsounif}
W_2(\bar{M}_{\tau}(t), M_{\tau}(t))= W_2\left(M_\tau\left(\floor{\frac{t}{\tau}}\tau\right),M_\tau(t)\right)\le L\tau,\quad\text{ for any }t\in[0,T]
\end{equation}
so that $\mu$ is also the uniform limit of $\bar{M}_{\tau(n)}$.\\
Let us fix a reference measure $\nu\in \dom(\frF)$ and $\Phi\in
  \frF[\nu]$.
  \ref{eq:IEVI} and the $\lambda$-dissipativity of $\frF$ yield
\begin{align*}
\frac{\de}{\de t} \frac{1}{2} W_2^2(M_{n}(t), \nu) &\le  \tau(n) |
                                                        \fF_{\tau(n)}(t)|_2^2
                                                        + \bram{
                                                        \fF_{\tau(n)}}{\nu}\\
                                                        &\le \tau(n)\, L^2 + \lambda W_2^2(\bar{M}_{\tau(n)}(t), \nu) -\bram{\Phi}{\bar{M}_{\tau(n)}(t)}
\end{align*}
for a.e. $t \in [0,T]$.
Integrating the above inequality in an interval $(t,t+h)\subset [0,T]$
we get
\begin{align}
  \label{eq:127}
  \frac {W_2^2(M_n(t+h),\nu)-
  W_2^2(M_n(t),\nu) }{2h}
  &\le
  \tau(n) L^2\\ \notag
  &+\frac{1}{h}\int_t^{t+h}
  \Big(\lambda W_2^2(\bar{M}_{\tau(n)}(s), \nu)
  -\bram{\Phi}{\bar{M}_{\tau(n)}(s)}\Big)\,\d s.
\end{align}
Notice that
as $n \to + \infty$, by \eqref{eq:alsounif}, we have
\begin{displaymath}
  \liminf_{n \to + \infty}\bram{\Phi}{\bar M_{\tau(n)}(s)}\ge
  \bram{\Phi}{\mu(s)}
  \quad\text{for every } s \in [0,T]
\end{displaymath}
together with the uniform bound given by
\begin{align*}
\left | \bram{\Phi}{\bar{M}_{\tau(n)}(s)} \right |  \le
  \frac{1}{2}W_2^2(\bar{M}_{\tau(n)}(s), \nu) +
  \frac{1}{2} |\Phi|_2^2 \quad\text{for every } s \in [0,T].
\end{align*}
Thanks to Fatou's Lemma and the uniform convergence given by Theorem \ref{theo:convergence},
we can pass to the limit as 
$n \to + \infty$ in \eqref{eq:127}
obtaining
\begin{equation*}
  \frac {W_2^2(\mu(t+h),\nu)-
  W_2^2(\mu(t),\nu) }{2h}
  \le
  \frac{1}{h}\int_t^{t+h}
  \Big(\lambda W_2^2(\mu(s), \nu)
  -\bram{\Phi}{\mu(s)}\Big)\,\d s.
\end{equation*}
A further limit as $h\downarrow0$ yields
\begin{displaymath}
  \frac 12\updt W_2^2(\mu(t),\nu)\le \lambda W_2^2(\mu(t),\nu)-
  \bram{\Phi}{\mu(t)}
\end{displaymath}
which provides \eqref{eq:EVI}.
\end{proof}

\appendix 

\section{Comparison with \texorpdfstring{\cite{Piccoli_2019}}{cp}}\label{sec:cfrPic}
In this section, we provide a brief comparison between the assumptions we required in order to develop a strong concept of solution to \eqref{eq:CP} and the hypotheses assumed in \cite{Piccoli_2019}. We remind that the relation between our solution and the weaker notion studied in \cite{Piccoli_2019} was exploited in Section \ref{sec:bar}.
Here, we conclude with a further remark coming from the connections between our approximating scheme proposed in \eqref{eq:EE} and the schemes proposed in \cite{Camilli_MDE} and \cite{Piccoli_2019}.

\medskip

We consider a finite time horizon $[0,T]$ with $T>0$, the space
$\X=\R^d$ and we deal with measures in $\prob_b(\R^d)$ and in $\prob_b(\TRd)$,
i.e.~compactly supported. We also deal with \emph{single-valued}
probability vector fields (PVF) for simplicity, which can be considered as everywhere defined
maps $\frF:\prob_b(\R^d)\to\prob_b(\TRd)$ such that
$\sfx_{\sharp}\frF[\nu]=\nu$. This is indeed the framework examined in
\cite{Piccoli_2019}.

We start by recalling the assumptions required in \cite{Piccoli_2019} for a PVF $\frF:\prob_b(\R^d)\to\prob_b(\TRd)$.
\begin{enumerate}[label=(H\arabic*)]
\item\label{Pgrowth}  there exists a constant $M>0$ such that for all $\nu \in\prob_b(\R^d)$,
\[\sup_{(x,v) \in \supp(\frF[\nu])}|v| \leq M\left(1 + \sup_{x \in \supp(\nu)}|x|\right);\]
\item\label{Plip} $\frF$ satisfies the following
  Lipschitz condition:
  there exists a constant $L\ge0$ such that
  for every  $\Phi=\frF[\nu],\ \Phi'=\frF[\nu']$
  there exists $\Ttheta\in \Lambda(\Phi,\Phi')$ satisfying 
\begin{equation*}
  \int_{\TRd \times \TRd} |v_0 - v_1|^2 \de\Ttheta(x_0,v_0,x_1,v_1)
  \le L^2 W_2^2(\nu,\nu'),
\end{equation*}
with $\Lambda(\cdot,\cdot)$ as in Definition \ref{def:lambda}.
\end{enumerate}
\begin{remark} We stress that actually in \cite{Piccoli_2019}
  condition {\rm \ref{Plip}} is local, meaning that $L$ is allowed to
  depend on
  the radius $R$ of a ball centered at $0$ and containing
  the supports of $\nu$ and $\nu'$. Thanks to assumption (H1),
  it is easy to show that for every final time $T$ all the discrete
  solutions of the Explicit Euler scheme
  and of the scheme of \cite{Piccoli_2019} 
  starting from an initial measure with support in $\mathrm B(0,R)$
  are supported in a ball $\mathrm B(0,R')$ where $R'$ solely depends
  on
  $R$ and $T$.
  We can thus restrict the PVF $\frF$ to the (geodesically convex) set of measures
  with support in $\mathrm B(0,R')$
  and act as $L$ does not depend on the support of the measures.
\end{remark}
\begin{proposition}\label{prop:compH}
  If $\frF:\prob_b(\R^d)\to\prob_b(\TRd)$ is a {\rm PVF}
  satisfying {\rm \ref{Plip}},
  then $\frF$ is $\lambda$-dissipative for $\lambda=\frac{L^2+1}{2}$,
  the Explicit Euler scheme is globally solvable in $\dom(\frF)$, and
  $\frF$ generates a $\lambda$-flow, whose trajectories are the limit
  of the Explicit Euler scheme in each finite interval $[0,T]$.
\end{proposition}
\begin{proof} The $\lambda$-dissipativity comes from Lemma \ref{le:pic1}. We prove that \eqref{eq:101pre} holds.
  Let $\nu\in\dom(\frF) $ and take $\Ttheta\in\Lambda(\frF[\nu],\frF[\delta_0])$ such that 
\begin{equation*}
\int_{\TRd\times \TRd}|v'-v''|^2\de\Ttheta\le
L^2W_2^2(\nu,\delta_0)=L^2
\sqm{\nu}.
\end{equation*}
Since $\frF[\delta_0] \in\prob_c(\TRd)$ by assumption, there exists $D>0$ such that $\supp(\sfv_\sharp \frF[\delta_0])\subset B_D(0)$. Hence, we have
\begin{align*}
L^2\sqm{\nu}&\ge \int_{\TRd\times \TRd}|v'-v''|^2\de\Ttheta\ge \int_{\TRd\times \TRd}[|v'|-D]_+^2\de\Ttheta\\
&\ge\int_{\TRd}|v'|^2\de \frF[\nu]-2D\int_{\TRd}|v'|\de \frF[\nu],
\end{align*}
where $[\,.\,]_+$ denotes the positive part. By the trivial estimate $|v'|\le D+\frac{|v'|^2}{4D}$, we conclude
\[|\frF[\nu]|_2^2\le 2\left(2D^2+L^2\sqm{\nu}\right).\]
Hence \eqref{eq:101pre} and thus the global solvability of the
Explicit Euler scheme in $\dom(\frF)$ by Proposition
\ref{prop:gsolvcond}. To conclude it is enough to apply Theorem
\ref{thm:globall}(a) and Theorem \ref{theo:strong-solution}.
\end{proof}

It is immediate to notice that the semi-discrete Lagrangian scheme
proposed in \cite{Camilli_MDE} coincides with the Explicit Euler
Scheme given in Definition \ref{def:EEscheme}. In particular,
we can state the following
comparison between the limit obtained by the Explicit Euler scheme
\eqref{eq:EE} (leading to the $\lambda$-\EVI solution of
\eqref{eq:CP}) and that of the approximating LASs scheme proposed in
\cite{Piccoli_2019} (leading to a barycentric solution to
\eqref{eq:CP} in the sense of Definition \ref{def:barprop}).

\begin{corollary}\label{cor:cfrLAS}
Let $\frF$ be a {\rm PVF} satisfying \ref{Pgrowth}-\ref{Plip},
$\mu_0\in\prob_{b}(\R^d)$ and let $T\in(0,+\infty)$. Let
$(n_k)_{k\in\N}$ be a sequence such that the LASs scheme
$(\mu^{n_k})_{k\in\N}$ of \cite[Definition 3.1]{Piccoli_2019}
converges uniformly-in-time and let $(M_{\tau_k})_{k\in\N}$ be the
affine interpolants of the Explicit Euler Scheme defined in \eqref{eq:affineM}, with $\tau_k=\frac{T}{n_k}$. Then $(\mu^{n_k})_{k\in\N}$ and $(M_{\tau_k})_{k\in\N}$ converge to the same limit curve $\mu\in \rmC([0,T];\prob_b(\R^d))$,
which is the unique $\lambda$-\EVI solution of \eqref{eq:CP} in $[0,T]$.
\end{corollary}
\begin{proof} 
  By Proposition \ref{prop:compH}, $\frF$ is a
  $\left(\frac{L^2+1}{2}\right)$-dissipative  \MPVF s.t. $\mathrm
  \mathcal{M}(\mu_0, \tau, T, \tilde{L}) \ne \emptyset$ for every $\tau>0$,
  where $\tilde{L}>0$ is a suitable constant depending on $\mu_0$ and
  $\frF$. Thus by Theorem \ref{theo:strong-solution}, $(M_{\tau_k})_{k
    \in \N}$ uniformly converges to a $\lambda$-\EVI solution $\mu \in
  \rmC([0,T]; \prob_2(\R^d))$ which is unique since $\frF$ generates a
  $\left(\frac{L^2+1}{2}\right)$-flow.
  Since we start from a compactly supported $\mu_0$, the semi-discrete
  Lagrangian scheme of \cite{Camilli_MDE} and our Euler Scheme
  actually coincide. To conclude we apply \cite[Theorem 4.1]{Camilli_MDE} obtaining that $\mu$ is also the limit of the LASs scheme.
\end{proof}
We conclude that among the possibly not-unique (see \cite{Camilli_MDE}) barycentric solutions to \eqref{eq:CP} - i.e. the solutions in the sense of \cite{Piccoli_2019}/Definition \ref{def:barprop} -  we are \emph{selecting} only one (the $\lambda$-\EVI solution), which turns out to be the one associated with the LASs approximating scheme.

In light of this observation, we revisit an interesting example studied in \cite[Section 7.1]{Piccoli_2019} and \cite[Section 6]{Camilli_MDE}.

\begin{example}[Splitting particle]
For every $\nu\in\prob_b(\R)$ define:
\begin{equation*}
B(\nu):=\sup \left\{x:\nu(]-\infty,x])\leq \frac12\right\},\qquad
\eta(\nu):=\nu(]-\infty,B(\nu)]) - \frac12,
\end{equation*}
so that $\nu(\{B(\nu)\})=\eta(\nu)+\frac12-\nu(]-\infty,B(\nu)[)$.
We define the PVF $\frF[\nu]:=
\int \frF_x[\nu]\,\d\nu(x)$, by
\begin{equation*}
\frF_x[\nu]:=\left\{
\begin{array}{ll}
\delta_{-1} & \textrm{if}\ x<B(\nu)\\
\delta_{1} & \textrm{if}\ x>B(\nu)\\
\frac{1}{\nu(\{B(\nu)\})}
\left(\eta\delta_{1}+
\left(\frac12-\nu(]-\infty,B(\nu)[)\right)\delta_{-1}\right) & \textrm{if}\ x=B(\nu),
\nu(\{B(\nu)\})>0.
\end{array}
\right.
\end{equation*}
By \cite[Proposition 7.2]{Piccoli_2019}, $\frF$ satisfies assumptions \ref{Pgrowth}-\ref{Plip} with $L=0$ and the LASs scheme admits a unique limit. Moreover, the solution $\mu:[0,T]\to\prob_b(\R)$ obtained as limit of LASs, is given by
\begin{equation}\label{solex}
  \begin{split}
\mu_t(A) = &\mu_0((A\cap ]-\infty,B(\mu_0)-t[)+t)
+ \mu_0((A\cap ]B(\mu_0)+t,+\infty[)-t)\\
&+\frac{1}{\mu_0(\{B(\mu_0)\})}
\left( \eta \delta_{B(\mu_0)+t}(A)+ (\frac12-\mu_0(]-\infty,B(\mu_0)[))\delta_{B(\mu_0)-t}(A)\right).
\end{split}
\end{equation}
By Corollary \ref{cor:cfrLAS}, \eqref{solex} is the (unique) $\lambda$-\EVI solution of \eqref{eq:CP}.
In particular:
\begin{itemize}
\item[i)]if $\mu_0=\frac{1}{b-a}\mathcal L\llcorner_{[a,b]}$, i.e.~the normalized Lebesgue measure restricted to $[a,b]$, we get $\mu_t=\frac{1}{b-a}\mathcal L\llcorner_{[a-t,\frac{a+b}{2}-t]} +\frac{1}{b-a}
\mathcal L\llcorner_{[\frac{a+b}{2}+t,b+t]}$;
\item[ii)] if $\mu_0=\delta_{x_0}$, we get $\mu_t=\frac12 \delta_{x_0+t}+\frac12 \delta_{x_0-t}$.
\end{itemize}

Notice that, in case (i), since $\mu_t\ll\mathcal{L}$  for all $t\in(0,T)$, i.e. $\mu_t\in\prob_2^r(\R)$, we can also apply Theorem \ref{thm:RegularM} to conclude that $\mu$ is the $\lambda$-\EVI solution of \eqref{eq:CP} with $\mu_0=\frac{1}{b-a}\mathcal L\llcorner_{[a,b]}$. Moreover, take $\eps>0$, and consider case (i) where we denote by $\mu^\eps_0$ the initial datum and by $\mu^\eps$ the corresponding $\lambda$-\EVI solution to \eqref{eq:CP} with $a=x_0-\eps$, $b=x_0+\eps$. We can apply \eqref{eq:143} with $\mu_0=\mu_0^{\eps}$ and $\mu_1=\delta_{x_0}$ in order to give another proof that, for all $t\in[0,T]$, the $W_2$-limit of $S_t[\mu_0^{\eps}]$ as $\eps\downarrow0$, that is $S_t[\delta_{x_0}]=\frac12 \delta_{x_0+t}+\frac12 \delta_{x_0-t}$, is a $\lambda$-\EVI solution starting from $\delta_{x_0}$. Thus we end up with (ii).

Dealing with case (ii), we recall that, if $\mu_0=\delta_{x_0}$ then
also the stationary curve $\bar\mu_t=\delta_{x_0}$, for all
$t\in[0,T]$, satisfies the barycentric property of Definition
\ref{def:barprop} (see \cite[Example 6.1]{Camilli_MDE}), thus it is a
solution in the sense of \cite{Piccoli_2019}.
However, $\bar\mu$ is not a $\lambda$-\EVI solution since it does not coincide with the curve given by ii).
This fact can also be checked by a direct calculation as follows:
we find $\nu\in\prob_b(\R)$ such that
\begin{equation}\label{eq:counterP1}
  \frac{\de}{\de t}\frac{1}{2}W_2^2(\bar\mu_t,\nu)>\lambda
  W_2^2(\bar\mu_t,\nu)-\bram{\frF[\nu]}{\bar\mu_t}\qquad t\in (0,T),
\end{equation}
where $\lambda=\frac{1}{2}$ is the dissipativity constant of the PVF
$\frF$ coming from the proof of Proposition \ref{prop:compH}. Notice
that the LHS of \eqref{eq:counterP1} is always zero since
$t\mapsto\bar\mu_t=\delta_0$ is constant. Take $\nu=\mathcal
L\llcorner_{[0,1]}$ so that we get $\frF[\nu]= \int \frF_x[\nu]\,\d\nu(x)$, with $\frF_x[\nu]=\delta_1$ if $x>\frac{1}{2}$, $\frF_x[\nu]=\delta_{-1}$ if $x<\frac{1}{2}$. Noting that $\Lambda(\frF[\nu],\delta_0)=\{\frF[\nu]\otimes\delta_0\}$, by using the characterization in Theorem \ref{thm:characterization} we compute
\begin{equation*}
\bram{\frF[\nu]}{\delta_0}=\int_{\TX}\scalprod{x}{v}\de \frF[\nu]=\int_0^{1/2}\scalprod{x}{v}\de \frF_x[\nu](v)\de x+\int_{1/2}^1\scalprod{x}{v}\de \frF_x[\nu](v)\de x=\frac{1}{4}.
\end{equation*}
Since $W_2^2(\delta_0,\nu)=\sqm{\nu}=\frac{1}{3}$, we have
\[\lambda W_2^2(\bar\mu_t,\nu)-\bram{\frF[\nu]}{\bar\mu_t}=\frac{1}{6}-\frac{1}{4}<0,\]
and thus we obtain the desired inequality \eqref{eq:counterP1} with $\nu=\mathcal L\llcorner_{[0,1]}$.

\end{example}

\section{Wasserstein differentiability along curves}\label{app:B}
 In general, if $\mu:[0, + \infty) \to \prob_2(\X)$ is a locally absolutely continuous curve and $\nu\in\prob_2(\X)$, then the map $[0, + \infty) \ni s \mapsto W_2^2(\mu_s, \nu)$ is locally absolutely continuous and thus differentiable in a set of full measure $A({\mu,\nu}) \subset (0, + \infty)$ which, in principle, depends both on $\mu$ and $\nu$. What Theorem \ref{thm:refdiff} shows is that, independently of $\nu$, there is a full measure set $A(\mu)$, depending only on $\mu$, where this map is left and right differentiable. If moreover $\nu$ and $t \in A(\mu)$ are such that there is a unique optimal transport plan between them, we can actually conclude that such a map is differentiable at $t$. \\
We want to highlight how this result is optimal giving an example of a locally absolutely continuous curve $\mu:[0,+\infty) \to \prob_2(\R^2)$ s.t.~the full measure set of differentiability points of the map $[0, + \infty) \ni s \mapsto W_2^2(\mu_s, \nu)$ depends also on $\nu\in\prob_2(\R^2)$. To do that it is enough to show that
\[ \text{ for every } t_0 \in A(\mu) \text{ there exist } \nu_0 \in \prob_2(\R^2) \text{ and } \ggamma_1, \ggamma_2 \in \Gamma_o(\mu_{t_0}, \nu_0) \text{ s.t. } L(\ggamma_1) \ne L(\ggamma_2),\]
where $A(\mu)$ is as in Theorem \ref{thm:tangentv} and, for $\ggamma \in \prob_2(\R^2\times\R^2)$ s.t. $\sfx^0_{\sharp}\ggamma = \mu_t$, we define
\[ L(\ggamma) := \int_{\X^2} \scalprod{\vv_t(x)}{x-y} \de \ggamma(x, y).\]
Indeed this will imply that $\bram{(\ii_\X, \vv_{t_0})_{\sharp}\mu_{t_0}}{\nu_0} \ne \brap{(\ii_\X , \vv_{t_0})_{\sharp}\mu_{t_0}}{\nu_0}$, hence the non differentiability at $t_0$.\\
Let us consider two regular functions $u:[0, + \infty) \to \R^2$ and $r:[0, + \infty) \to \R$ s.t. $|u_t|=1$ for every $t \ge 0$. Let $\omega:[0, + \infty) \to \R^2$ be defined as the orthogonal direction to $u_t$:
\[ \omega_t := \begin{pmatrix}
0 & -1 \\
1 & 0 
\end{pmatrix}u_t, \quad \quad t \ge 0.\]
Being the norm of $u$ constant in time, there exists some regular $\lambda: (0, + \infty) \to \R$ s.t.~$\dot{u}_t = \lambda_t \omega_t$ for every $t>0$. Finally we define
\begin{align*}
x_1: [0, + \infty) \to \R^2, &\qquad  x_1(t) := r_tu_t, \\
x_2: [0, + \infty) \to \R^2, &\qquad x_2(t) := -r_tu_t, \\
\mu:[0, + \infty) \to \prob_2(\R^2), &\qquad \mu_t := \frac{1}{2} \left ( \delta_{x_1(t)} + \delta_{x_2(t)} \right ).
\end{align*}
Observe that $\dot{x}_1(t) = \dot{r}_tu_t + r_t\dot{u}_t = - \dot{x}_2(t)$ for every $t>0$. Moreover, for every $\varphi \in \rmC^{\infty}_c(\R^2)$ and $t>0$, we have
\begin{align*}
\frac{\de}{\de t} \int_{\R^2} \varphi \de \mu_t &= \frac{\de}{\de t} \left ( \frac{1}{2} \varphi(x_1(t)) + \frac{1}{2} \varphi(x_2(t)) \right )=  \frac{1}{2} \nabla \varphi(x_1(t))\, \dot{x}_1(t) + \frac{1}{2} \nabla \varphi(x_2(t)) \,\dot{x}_2(t) \\
&= \int_{\R^2} \scalprod{v_t(x)}{\nabla \varphi(x)} \de \mu_t,
\end{align*}
where 
\[ \vv_t(x) := \begin{cases} \dot{x}_1(t) \quad &\text{ if } x = x_1(t), \\ \dot{x}_2(t) \quad &\text{ if } x = x_2(t), \end{cases} \quad t >0.\]
Hence, the above defined vector field $\vv_t$ solves the continuity equation with $\mu_t$. Let $t_0 \in A(\mu)$ and let us define $\omega_0:= \omega(t_0)$,  $\nu_0:= \frac{1}{2} \delta_{\omega_0} + \frac{1}{2} \delta_{-\omega_0}$ and the plans $\ggamma_1, \ggamma_2 \in \Gamma_o(\mu_{t_0}, \nu_0)$ by
\begin{align*}
\ggamma_1 := \frac{1}{2} \delta_{x_1(t_0)} \otimes \delta_{\omega_0} + \frac{1}{2} \delta_{x_2(t_0)} \otimes \delta_{-\omega_0},\\
\ggamma_2 := \frac{1}{2} \delta_{x_2(t_0)} \otimes \delta_{\omega_0} + \frac{1}{2} \delta_{x_1(t_0)} \otimes \delta_{-\omega_0}.
\end{align*}
Notice that they are optimal since any plan in $\Gamma(\mu_{t_0}, \nu_0)$ has the same cost, being the points $\omega_0, x_1(t_0), x_2(t_0), -\omega_0$ the vertexes of a rhombus. Finally, we compute $L(\ggamma_1)$ and $L(\ggamma_2)$:
\begin{align*}
L(\ggamma_1) &= \int_{\R^2 \times \R^2} \scalprod{x-y}{\vv_t(x)} \de \ggamma_1(x,y) = \frac{1}{2} \scalprod{\dot{x}_1(t_0)}{x_1(t_0)-\omega_0} + \frac{1}{2} \scalprod{\dot{x}_2(t_0)}{x_2(t_0)+\omega_0}\\
&= \scalprod{\dot{x}_1(t_0)}{x_1(t_0)-\omega_0}= \scalprod{\dot{r}_{t_0}u_{t_0} + r_{t_0}\dot{u}_{t_0}}{r_{t_0}u_{t_0} -\omega_0}= r_{t_0}\dot{r}_{t_0} - r_{t_0} \lambda_{t_0},\\
L(\ggamma_2) &= \int_{\R^2 \times \R^2} \scalprod{x-y}{\vv_t(x)} \de \ggamma_2(x,y) = \frac{1}{2} \scalprod{\dot{x}_2(t_0)}{x_2(t_0)-\omega_0} + \frac{1}{2} \scalprod{\dot{x}_1(t_0)}{x_1(t_0)+\omega_0}\\
&= \scalprod{\dot{x}_1(t_0)}{x_1(t_0)+\omega_0}= \scalprod{\dot{r}_{t_0}u_{t_0} + r_{t_0}\dot{u}_{t_0}}{r_{t_0}u_{t_0} +\omega_0} = r_{t_0}\dot{r}_{t_0} + r_{t_0} \lambda_{t_0}.
\end{align*}
In this way, if $r_{t_0} \ne 0$ and $\lambda_{t_0} \ne 0$ we have $L(\ggamma_1) \ne L(\ggamma_2)$. A possible choice for $u$ and $r$ satisfying the assumptions is
\[ u_t := (\cos(t), \sin(t)),  \quad \quad r_t = 1, \quad \quad t \ge 0,\]
so that $\lambda_t=1$ for every $t>0$.

\section{Support function and Dini derivatives}
We recall the following characterization of the
closed convex hull $\cloco C$ of a set $C$ (i.e.~the intersection of all the closed convex sets containing $C$)
in a Banach space.
\begin{lemma} \label{lem:abstract} Let $Z$ be a Banach space and let $C \subset Z$ be nonempty. Then $v \in \cloco C$ if and only if 
\begin{equation}\label{eq:dual}
\scalprod{z^*}{v} \le \sup_{c \in C} \,\scalprod{z^*}{c} \quad \forall \, z^* \in Z^*. \end{equation}
Moreover if $C$ is bounded, it is enough to have \eqref{eq:dual} holding for every $z^* \in W$, with $W$ a dense subset of  $Z^*$.
\end{lemma}
\begin{proof}
The result is a direct consequence of Hahn-Banach theorem.

Concerning the last assertion, observe that the function
\[ Z^* \ni z^* \mapsto \sup_{c \in C}\,  \scalprod{z^*}{c} \]
is Lipschitz continuous if $C$ is bounded. Hence, if \eqref{eq:dual} holds only for some $W \subset Z^*$ dense, then it holds for the whole $Z^*$.
\end{proof}

Let us state and prove a simple lemma that allows us to pass from a differential inequality for the right upper Dini derivative to the corresponding distributional inequality (see also \cite[Lemma A.1]{MuratoriSavare} and \cite{Gal}).
\begin{lemma}\label{lem:distrib}
Let $(a,b) \subset \R$ be an open interval (bounded or unbounded) and
let $\zeta, \eta: (a,b) \to \R$ be s.t.~$\zeta$ is continuous in
$(a,b)$
and $\eta$ is measurable
and locally bounded from above in $(a,b)$. If
\[ \updt \zeta(t) \le \eta(t) \quad \text{ for every } t \in (a,b),\]
then the above inequality holds also in the sense of distributions, meaning that 
\[ -\int_a^b \zeta(t) \varphi'(t) \de t \le \int_a^b \eta(t) \varphi(t)\de t \quad \text{ for every } \varphi \in \rmC^{\infty}_c(a,b).\]
\end{lemma}
\begin{proof} Let $\varphi \in \rmC^{\infty}_c(a,b)$, then there exist $a<x<y<b$ s.t.~the support of $\varphi$ is contained in $[x,y]$ ; since $\eta$ is locally bounded from above, there exists a positive constant $C>0$ s.t.~$\eta(t) \le C$ for every $t \in [x,y]$. Then the function $t \mapsto \zeta(t)-Ct$ is s.t.
\[ \updt (\zeta(t)-Ct) \le 0 \quad \text{ for every } t \in [x,y]\]
so that it is decreasing in $[x,y]$ and hence a function of bounded variation in $[x,y]$. Its distributional derivative is hence a non positive measure $T$ on $[x,y]$ whose absolutely continuous part (w.r.t. the $1$-dimensional Lebesgue measure on $[x,y]$) coincides a.e.~with the right upper Dini derivative. Then we have
\[ -\int_a^b (\zeta(t)-Ct) \varphi'(t) \de t = T(\varphi) = \int_a^b \updt(\zeta(t)-Ct) \varphi(t) \de t + T_s(\varphi) \le \int_a^b (\eta-C) \varphi(t) \de t,\]
where $T_s$ is the singular part of $T$. This immediately gives the thesis.
\end{proof}

\printbibliography
\end{document}